\documentclass[11pt]{article}

\usepackage{amsthm}
\usepackage{hyperref}
\usepackage{amssymb}
\usepackage{latexsym}
\usepackage{amsmath}
\usepackage{amsfonts}
\usepackage{version}
\usepackage[dvips]{graphicx}
\usepackage[LGR,T1]{fontenc}%
\usepackage[francais,english]{babel}
\usepackage{url}
\usepackage{enumerate}
\usepackage{hyperref}
\usepackage{color}
\usepackage{bm}
\usepackage{multi row}

\usepackage{color}
\definecolor{darkgreen}{rgb}{0, .5, 0}

\usepackage{color}
\definecolor{darkgreen}{rgb}{0, .5, 0}
\definecolor{darkred}{rgb}{.5, 0, 0}

\numberwithin{equation}{section}
\theoremstyle{plain}

\numberwithin{thm}{section}
\newtheorem{prop}{Proposition}
\numberwithin{prop}{section}
\newtheorem{lm}{Lemma}
\numberwithin{lm}{section}

\numberwithin{term}{section}

\numberwithin{claim}{section}
 
\numberwithin{cor}{section}

\numberwithin{problem}{section}
\theoremstyle{definition} 
 
\numberwithin{definition}{section}
 
\numberwithin{example}{section}

\numberwithin{conjecture}{section}

\numberwithin{condition}{section}
\theoremstyle{remark} 
\newtheorem{remark}{Remark}
\numberwithin{remark}{section}

\numberwithin{remark}{section}

\begin{document}

\title{Directed Chain Stochastic Differential Equations}

\vspace{1cm} 

{
\author{
NILS  DETERING 
\thanks{~Department of Statistics and Applied Probability, South Hall, University of California, Santa Barbara, CA 93106, USA (E-mail:    {\it detering@pstat.ucsb.edu})}
\and JEAN-PIERRE FOUQUE 
\thanks{~Department of Statistics and Applied Probability, South Hall, University of California, Santa Barbara, CA 93106, USA (E-mail:    {\it fouque@pstat.ucsb.edu}). }
\and TOMOYUKI ICHIBA   
\thanks{~Department of Statistics and Applied Probability, South Hall, University of California, Santa Barbara, CA 93106, USA (E-mail:    {\it ichiba@pstat.ucsb.edu}). }
}
}

\date{June 22, 2019}

\maketitle

\begin{abstract}
{We propose a particle system of diffusion processes coupled through a chain-like network structure described by an infinite-dimensional, nonlinear stochastic differential equation of McKean-Vlasov type. It has both (i) a local chain interaction and (ii)  a mean-field interaction.  It  can be approximated by a limit of finite particle systems, as the number of particles goes to infinity. Due to the local chain interaction, propagation of chaos does not necessarily hold. Furthermore, we exhibit a dichotomy of presence or absence of mean-field interaction, and we discuss the problem of detecting its presence from the observation of a single component process. }

\end{abstract} 

\medskip 

{{\it Key Words and Phrases:} Interacting stochastic processes, stochastic equation with constraints, law of large numbers, particle system approximation, detecting mean-field.}

\medskip 

\noindent{\it AMS 2010 Subject Classifications:} Primary 60H10; secondary 60K35

\section{Introduction}
\label{sec: Chain} 

Let us consider a directed graph (or oriented network) of vertices $\, \{1, \ldots , n\}\,$ on a circle in the sense that each vertex $\,i\,$ in the graph is the head of an arrow directed from its neighboring vertex $\, i+1\,$ for $\, i \, =\, 1, \ldots , n-1\,$, and the boundary vertex $\, n\,$ is the head of an arrow directed from the first vertex $\, 1\,$. On some probability space with independent Brownian motions $\, (W_{\cdot, i})\, \, $, $\, 1\le i \le n\,$, assigned to the vertices, we consider  a process $\, X_{\cdot,i}\,$ defined by the following system of equations which incarnates this graph structure through drifts:  
\begin{equation} \label{eq: chain1}
\begin{split}
{\mathrm d} X_{t,i} \, & =\,  h(X_{t, i}, X_{t, i+1}) {\mathrm d} t + {\mathrm d} W_{t,i} \, ; \quad t \ge 0 \, , \quad i \, =\, 1, \ldots , n-1 \,,  
\\
{\mathrm d} X_{t,n} \, &=\,  h(X_{t, n}, X_{t, 1}) {\mathrm d} t + {\mathrm d} W_{t,n} \, . 
\end{split}
\end{equation} 
The initial values $\, X_{0,i}\,$ are independent and  identically distributed random variables, independent of $\, (W_{\cdot, i})\,$, $\, 1 \le i \le n\,$. Furthermore, $\,h: \mathbb R^{2} \to \mathbb R \,$ is a Lipschitz function. 

We view $\, (X_{\cdot, 1}, \ldots , X_{\cdot, n}) \,$ as a particle system interacting through this particular directed graph. The system is invariant under a shift of the indexes of the particles. In particular, the law of $\,X_{\cdot, i}\,$ is the same as the law of $\, X_{\cdot, 1}\,$ for every $\, i\,$ and also the joint law of $\, (X_{\cdot, i}, X_{\cdot, i+1})\,$ is the same as the joint law of $\, (X_{\cdot, 1}, X_{\cdot, 2}) \,$ for every $\,i \,$. Let us call such interaction in (\ref{eq: chain1})  a {\it directed chain} interaction. Note that if $\,h(x_{1}, x_{2}) \, =\,  x_{2}-x_{1}\,$, $\, (x_{1}, x_{2}) \in \mathbb R^{2}\,$, it is a simple Ornstein-Uhlenbeck type system (or a Gaussian cascade). Intuitively, because of the mean-reverting feature of Ornstein-Uhlenbeck type drifts, the particle $\,X_{\cdot, i}\,$ at vertex $\,i\,$ in (\ref{eq: chain1}) tends to be close to the neighboring particle $\, X_{\cdot, i+1}\,$ {\it locally} under this particular choice of function $\,h\,$. 

\smallskip 

For comparison, on the same probability space, we also consider a typical mean-field interacting system where each particle is attracted towards the mean, defined by 
\begin{equation} \label{eq: MF1}
{\mathrm d} X_{t,i} \, =\, \frac{\,1\,}{\,n\,}  \sum_{\stackrel{j=1}{j\neq i}}^{n} h \big( X_{t,i},  X_{t,j} \big) {\mathrm d} t + {\mathrm d} W_{t,i} \, ; \quad t \ge 0 \, , \quad i \, =\, 1, \ldots , n \, .
\end{equation} 
This system (\ref{eq: MF1}) is invariant under permutations of indexes of particles, while the system (\ref{eq: chain1}) only possesses the shift invariance.  
Again, if $\,h(x_{1}, x_{2}) \, =\,  x_{2}-x_{1}\,$, $\, (x_{1}, x_{2}) \in \mathbb R^{2}\,$, the particle $\,X_{\cdot , i}\,$ at node $\,i\,$ is {\it directly} attracted towards the mean $\, (X_{\cdot, 1} + \cdots + X_{\cdot, n})/n\,$ of the system. This type of mean-field model has been considered in {Carmona, Fouque \& Sun} (2015) as a Nash equilibrium of a stochastic game in the context of financial systemic risk. The drift in this system in contrast incarnates the structure of a complete graph. 

\bigskip 

\noindent {\bf Questions}. What is the essential difference between the system (\ref{eq: chain1}) and (\ref{eq: MF1}) for large $\,n\,$? Can we detect the type of interaction from the single particle behavior at a vertex? 

\bigskip 

To answer these questions,  let us fix $\, u \in [0, 1]\,$ and introduce a mixed system: 
\begin{equation} \label{eq: NDMF}
\begin{split}
{\mathrm d} X_{t,i} \, &=\,  \Big( u \cdot h (X_{t,i}, X_{t,i+1} ) + (1-u) \cdot \frac{\,1\,}{\,n\,} \!\! \sum_{j=1, j \neq i}^{n}  \!\!  h \big( X_{t,i}, X_{t,j}\big) \Big)  {\mathrm d} t + {\mathrm d} W_{t, i} \, ,  \, \\
{\mathrm d} X_{t,n} \, &=\,  \Big( u \cdot h (X_{t,n}, X_{t,1})  + (1-u) \cdot \frac{\,1\,}{\,n\,} \!\! \sum_{j=1, j \neq n}^{n} \!\!   h \big(  X_{t,n}, X_{t,j} \big) \Big) {\mathrm d} t + {\mathrm d} W_{t, n} 
\end{split}
\end{equation}
for $\, t \ge 0 \, $, $\, i \, =\,  1, \ldots , n-1\, $ with the initial random variables $\, X_{0, i}\,$, $\, 1\le i \le n\,$. If $\,u \, =\,  1\,$, (\ref{eq: NDMF}) becomes (\ref{eq: chain1}), while if $\, u \, =\, 0\,$, (\ref{eq: NDMF}) becomes (\ref{eq: MF1}). 

\medskip 

The motivation of our study is to understand in a first instance the effect of the graph (network) structure on the stochastic system of interacting diffusions. Interacting diffusions have been studied in various contexts: nonlinear {McKean-Vlasov} equations, propagation of chaos results, large deviation results, stochastic control problems in large infinite particle systems, and their applications to Probability and Mathematical Physics, and more recently to Mathematical Economics and Finance in the context of the mean-field games.  One of the advantages of introducing the mean-field dependence (\ref{eq: MF1}) and the corresponding limits, as $\,n \to \infty\,$, is to obtain a clear description of the complicated system, in terms of a representative particle, by the law of large numbers. As a result of the invariance under permutations of the indexes of particles, it often comes with the propagation of chaos, and then consequently the local dependence in the original system disappears in the limit. The single representative particle is characterized by a non-linear single equation, and the limiting distribution of many particles can be represented as a product measure. See Remark \ref{rem: PoC} in section \ref{sec: InfParticle} below for a short list of references and related research on propagation of chaos. 

\medskip 

Here, in contrast, by breaking the invariance under permutation of particles, we consider the limit of the system (\ref{eq: NDMF}) (or its slight generalization in the next section) as $\,n \to \infty \,$ and attempt to describe the presence of both, mean-field and local directed chain dependence in the interacting particles. In our directed chain dependence, conceptually there is a pair of representative particles in the limit: a particle (say $\,X_{\cdot}\,$) which corresponds to the {\it head} of an arrow and another particle (say $\, \widetilde{X}_{\cdot}\,$) which corresponds to the {\it tail} of the same arrow, i.e., the arrow directs from the particle $\,\widetilde{X}_{\cdot}\,$ to the particle $\, X_{\cdot}\,$. The marginal laws of $\, X_{\cdot}\,$ and $\, \widetilde{X}_{\cdot}\,$ are the same as a consequence of construction,  and the dynamics of $\, X_{\cdot}\,$ is determined by its law, its position, the position of $\,\widetilde{X}_{\cdot}\,$ and a Brownian noise $\, B_{\cdot}\,$. As a result, our stochastic equation for the representative pair $\, ( X_{\cdot}, \widetilde{X}_{\cdot})\,$ is described in the limit by a weak solution to a single non-linear equation with constraints on the marginal law of particles (see (\ref{eq: NDu})-(\ref{eq: INITIAL}) below). The limiting distribution of a collection of particles is not necessarily a product measure, unless $\, u \, =\, 0\,$. When $\, u \in (0, 1]\,$, because of the local chain dependence, the single non-linear equation (\ref{eq: NDu})-(\ref{eq: Fu}) with distributional constraints (\ref{eq: NDlaw})-(\ref{eq: INITIAL}) has an infinite-dimensional nesting structure (see Remarks \ref{rem: 1.2} and \ref{rem: 3.1} below). Moreover, when $\,u \in (0, 1]\,$, essentially because of the violation of permutation invariance, the stochastic chaos does not propagate (see Remark \ref{rem: nPoC}). To our knowledge, our approach provides the first such instance in the context of particle approximation of the solution to a nonlinear stochastic equation of McKean-Vlasov type. 

\medskip

In section \ref{sec: MKV} we discuss existence and uniqueness of the solution to a directed chain 
stochastic differential equation (\ref{eq: NDu})-(\ref{eq: Fu}) for a representative pair $\, (X_{\cdot}, \widetilde{X}_{\cdot})\,$ of interacting stochastic processes  with distributional constraints (\ref{eq: NDlaw})-(\ref{eq: INITIAL}). In section \ref{sec: InfParticle} we propose a particle approximation of the solution to (\ref{eq: NDu})-(\ref{eq: Fu}), we study the convergence of joint empirical measures (\ref{eq: empMeas}) and an integral equation (\ref{eq: IntEq}) with (\ref{eq: Ainfinty}) for the limiting joint distribution in Propositions \ref{prop: 1.4}-\ref{prop: 1.5}. Moreover, we provide a simple fluctuation estimate in Proposition \ref{prop: 2.0}. We will see that the joint law of adjacent two particles in the limit of interacting particle systems of type (\ref{eq: NDMF}), as $\,n \to \infty\,$, can be described by the solution of the directed chain stochastic equation (\ref{eq: NDu})-(\ref{eq: Fu}) under some assumptions. In section \ref{sec: CaseStudy}, coming back to the above questions, we discuss the detection of the mean-field interaction as a filtering problem along with the systems of  equations of Zakai and Kushner-Stratnovich type in Propositions \ref{prop: Zakai2}-\ref{prop: KushnerStrat}. Then, we describe a connection to the infinite-dimensional Ornstein-Uhlenbeck process, and consequently,  examine the corresponding Gaussian processes under presence or absence of the mean-field interaction in section \ref{sec: iOU}. 
The appendix includes some more technical proofs.

\section{Directed chain stochastic equation with mean-field interaction} \label{sec: MKV}
On a filtered probability space $\, (\Omega, \mathcal F, (\mathcal F_{t}), \mathbb P) \,$, given a constant $\, u \in [0, 1]\,$ and a measurable functional $\, b: [0, \infty) \times \mathbb R \times {\mathcal M}(\mathbb R) \to \mathbb R\,$, let us consider a non-linear diffusion pair $\, (X_{t}^{(u)}, \widetilde{X}_{t}^{(u)})\, $, $\, t \ge 0\, $, described by the stochastic differential equation 
\begin{equation} \label{eq: NDu}
{\mathrm d} X^{(u)}_{t} \, =\, b(t, X_{t}^{(u)}, \, F^{(u)}_{t} ) \, {\mathrm d} t + {\mathrm d} B_{t} \, ; \quad t \ge 0 \, , 
\end{equation}
driven by a Brownian motion $\,(B_{t}, t \ge 0 )\,$, where $\, F^{(u)}_{\cdot}\,$  is the weighted probability measure 
\begin{equation} \label{eq: Fu}
F^{(u)}_{t} (\cdot) \, :=\,  u \cdot  \delta_{ \widetilde{X}^{(u)}_{t}} (\cdot) + (1-u) \cdot  \mathcal L_{X^{(u)}_{t}} (\cdot) 
\end{equation}
of the Dirac measure $\, \delta_{ \widetilde{X}^{(u)}_{t}} (\cdot) \,$ of $\, \widetilde{X}^{(u)}_{t}\,$ and the law $\, \mathcal L_{X^{(u)}_{t}} \, =\, \text{Law} ( X^{(u)}_{t}) \,$ of $\, X^{(u)}_{t}\,$ with corresponding weights $\, (u, 1-u) \,$ for $\, t \ge 0\,$. We shall assume that the law of $\, X^{(u)}_{\cdot}\,$ is identical to that of $\, \widetilde{X}^{(u)}_{\cdot}\,$,  and $\, \widetilde{X}^{(u)}_{\cdot}\,$ is independent of the Brownian motion, i.e., 
\begin{equation} \label{eq: NDlaw}
\, \text{Law}({(X^{(u)}_{t}, \, t \, \ge \, 0)})  \, \equiv\,  \text{Law}({( \widetilde{X}^{(u)}_{t} , \, t \, \ge \, 0 )})  \quad \text{ and }  \quad 
{\bm \sigma } (\widetilde{X}^{(u)}_{t}, \, t \, \ge\,  0 ) \, \perp \!\!\! \perp \, {\bm \sigma} (B_{t}\, , \, t \ge 0 ) \, . 
\end{equation}
Let us also assume that the Brownian motion $\,B_{\cdot}\,$ is independent of the initial value $\, (X_{0}^{(u)}, \widetilde{X}_{0}^{(u)}) \,$. We assume the joint and marginal initial distributions  of $\,(X_{0}^{(u)}, \widetilde{X}_{0}^{(u)}) \,$ are  given and denoted by 
\begin{equation} \label{eq: INITIAL} 
\,\Theta \, :=\, \text{Law} ( X_{0}^{(u)}, \widetilde{X}_{0}^{(u)}) 
\, =\,  \text{Law} (X_{0}^{(u)}) \otimes \text{Law} ( \widetilde{X}_{0}^{(u)}) \, =\,  \theta^{\otimes 2} \, , \quad \theta \, :=\,  \text{Law} (X_{0}^{(u)}) \, \equiv\,  \text{Law} ( \widetilde{X}_{0}^{(u)})\,. 
\end{equation}

Here we assume $\, \widetilde{X}_{\cdot}^{(u)}\,$ is a copy of $\, X_{\cdot}^{(u)}\,$ which has the same law (\ref{eq: NDlaw}) as a random element in the space of continuous functions, however it is not necessarily independent of $\, X_{\cdot}^{(u)}\,$. They can be independent when $\, u \, =\, 0\,$, as in Remark \ref{rem: 1} below. Rather, we are interested in the joint law of the pair $\, (X_{\cdot}^{(u)}, \widetilde{X}_{\cdot}^{(u)})\, $ which satisfies (\ref{eq: NDu}) and is generated from Brownian motion(s)  in a non-linear way through their probability law for each $\, u \in [0, 1]\,$. The description (\ref{eq: NDu}) with the constraints (\ref{eq: Fu})-(\ref{eq: NDlaw}) has an infinite-dimensional feature, because of non-trivial dependence between the unknown continuous processes $\, \widetilde{X}^{(u)}_{\cdot}\,$ and $\, X^{(u)}_{\cdot}\,$ in the space of continuous functions for every $\, u \in (0, 1]\,$. For a  precise description of the infinite-dimensional nesting structure, see remark \ref{rem: 1.2} below. 

When $\,u \in (0, 1)\,$ we shall call (\ref{eq: NDu}) with (\ref{eq: Fu})-(\ref{eq: INITIAL}) a {\it nonlinear, directed chain stochastic equation with mean-field interaction}. Let us denote by $\, \mathcal M (\mathbb R)\,$ (and $\, \mathcal M ( C([0, T], \mathbb R)) \,$, respectively) the family of probability measures on $\, \mathbb R\,$ (and the space $\, C([0, T], \mathbb R) \,$ of continuous functions equipped with the uniform topology on compact sets, respectively). Our following existence and uniqueness result relies on some standard assumptions to simplify the presentation.

\bigskip 

\begin{prop} \label{prop: 1} 
Suppose that $\,b: [0, \infty) \times \mathbb R \times \mathcal M (\mathbb R) \to \mathbb R\,$ is Lipschitz, in the sense that there exists a measurable function $\, \widetilde{b}: [0, \infty) \times \mathbb R \times \mathbb R \,$ such that $\,b\,$ is represented as 
\begin{equation} \label{eq: Lin}
b(t, x, \mu) \, =\, \int_{\mathbb R} \widetilde{b}(t, x, y) \mu({\mathrm d} y) \, ; \quad t \in [0, \infty) \, , \, \, x \in \mathbb R \, , \, \, \mu \in \mathcal M (\mathbb R) \, , 
\end{equation}
and for every $\, T > 0 \,$ there exists a constant $\,C_{T} > 0\,$ such that 
\begin{equation} \label{eq: Lip}
\lvert \widetilde{b}(t, x_{1}, y_{1}) - \widetilde{b}(t, x_{2}, y_{2}) \rvert \, \le\,  C_{T} ( \lvert x_{1} - x_{2} \rvert + \lvert y_{1} - y_{2} \rvert ) \, ; \quad 0 \le t \le T \, .
\end{equation}
With the same constant $\,C_{T}\,$, let us also assume that $\, \widetilde{b}\,$ is of linear growth, i.e., \begin{equation} \label{eq: 1.2a}
\sup_{0 \le s \le T} \lvert \widetilde{b}(s, x, y) \rvert \le C_{T} (1 + \lvert x \rvert + \lvert y\rvert) \, ; \quad x, y \in \mathbb R \, . 
\end{equation}
Then,  for each $\, u \in [0, 1]\,$ there exists a weak solution $\, (\Omega, \mathcal F, (\mathcal F_{t}), \mathbb P)\,$, $\, (X_{\cdot}^{(u)}, \widetilde{X}^{(u)}_{\cdot}, B_{\cdot}) \,$ to the stochastic equation $(\ref{eq: NDu})$ with $(\ref{eq: Fu})$-$(\ref{eq: INITIAL})$. This solution is  unique in law. 
\end{prop}


\begin{proof}
First, observe that it is reduced to the well-known existence and uniqueness results of {McKean-Vlasov} equation, when $\,u \, =\,  0\,$. In particular, because of (\ref{eq: NDlaw}), in this case the joint distribution of $\, (X_{\cdot}^{(u)} , \widetilde{X}_{\cdot}^{(u)} ) \,$ is a product measure.
Thus let us fix $\,u \in (0, 1]\,$ in the following, and also assume boundedness of the drift coefficients for the moment, i.e.,  
\begin{equation} \label{eq: Lip-a}
\lvert \widetilde{b}(t, x_{1}, y_{1}) - \widetilde{b}(t, x_{1}, y_{2}) \rvert \, \le\,  C_{T} ( ( \lvert x_{1} - x_{2} \rvert + \lvert y_{1} - y_{2} \rvert) \wedge 1 ) \, ; \quad t \ge 0 \, , 
\end{equation}
in order to simplify our proof. 
We shall evaluate the Wasserstein distance $\,D_{T} (\mu_{1}, \mu_{2})\,$ between two probability measures $\, \mu_{1}\,$ and $\, \mu_{2}\,$ on the space $\, C([0, T], \mathbb R) \,$ of continuous functions, namely    
\begin{equation} \label{eq: Wdist}
D_{t} (\mu_{1}, \mu_{2}) \, :=\, \inf \Big \{ \int ( \sup_{0 \le s \le t} \lvert  X_{s}(\omega_{1}) - X_{s}(\omega_{2}) \rvert   \wedge 1 ) {\mathrm d} \mu( \omega_{1}, \omega_{2}) \Big\} \, 
\end{equation}
for $\, 0 \le t \le T\,$, where the infimum is taken over all the joint distributions $\, \mu \in \mathcal M ( C([0, T], \mathbb R) \times C([0, T], \mathbb R) ) \,$ such that their marginal distributions are $\, \mu_{1} \,$ and $\, \mu_{2}\,$, respectively, and the initial joint and marginal distributions are $\,\Theta\,$ and $\, \theta \,$ in (\ref{eq: INITIAL}), that is, $\, \mu\vert_{\{t=0\}} \, =\,  \Theta\,$, $\, \mu_{i} \vert_{\{t=0\}} \, =\, \theta\,$, 
\[
\,\text{Law} ( X_{s} (\omega_{i}) , 0 \le s \le T ) \, =\, \mu_{i}\, \quad \text{ for } \, \, \,i \, =\, 1, 2\, \text{ and } \, \text{Law} ( X_{s}(\omega_{1}), X_{s}(\omega_{2}), 0 \le s \le T ) \, =\, \mu \,.
\]
Here $\, X_{s}(\omega) \, =\, \omega(s)\,$, $\, 0 \le s \le T\,$ is the coordinate map of $\, \omega \in C([0, T], \mathbb R)\,$. $\,D_{T}(\cdot , \cdot) \,$ defines a complete metric on $\, \mathcal M ( C([0, T], \mathbb R)) \,$, which gives the topology of weak convergence to it. 

Given a probability measure $\,  \mathrm m \in \mathcal M(C([0, T], \mathbb R)) \,$ with initial law $\, \mathrm m_{0} \, :=\,  \theta\,$ in (\ref{eq: INITIAL}) and the canonical process $\, \widetilde{X}^{\mathrm m}_{\cdot}\,$ of the law $\, \mathrm m \,$ with initial value $\, \widetilde{X}_{0}^{\mathrm m} \, :=\,  \widetilde{X}_{0}^{(u)}\,$, and the initial variables $\, (X_{0}^{(u)}, \widetilde{X}_{0}^{(u)}) \,$ from (\ref{eq: INITIAL}), let us consider a map $\,\Phi : \mathcal M (C([0, T], \mathbb R) ) \mapsto \mathcal M (C([0, T], \mathbb R) )  \,$ such that  
\begin{equation} \label{eq: Phi} 
\, 
\Phi (\mathrm m) \, :=\, \text{Law} ( X^{\mathrm m}_{t}, 0 \le t \le T) \, , 
\end{equation} 
where on a given filtered probability space $\, (\Omega, \mathcal F, \mathbb P)\,$ with filtration $\, (\mathcal F_{t})_{t\ge 0}\,$, given a fixed Brownian motion $\, B_{\cdot}\,$ on it, $\, X^{\mathrm m}_{\cdot}\,$ is defined from a solution $\, (X^{\mathrm m}_{\cdot}, \widetilde{X}^{\mathrm m}_{\cdot}) \,$ of the stochastic differential equation 
\begin{equation} \label{eq: NDu0} 
{\mathrm d} X_{t}^{\mathrm m} \, =\,  b( t, X_{t}^{\mathrm m} , u \, \delta_{ \widetilde{X}^{\mathrm m}_{t}} + ( 1- u) \mathrm m_{t}) {\mathrm d} t + {\mathrm d} B_{t} \, ; \quad 0 \le t \le T \, ,  
\end{equation}
with the initial values $\, (X_{0}^{\mathrm m} , \widetilde{X}_{0}^{\mathrm m}) \, =\,  (X_{0}^{(u)}, \widetilde{X}_{0}^{(u)}) \,$. That is, under the probability measure $\,\mathbb P\,$, $\,X_{\cdot}^{\mathrm m}\,$ is an $\, (\mathcal F_{t})\,$-adapted process and  the associated $\, (\mathcal F_{t})\,$- adapted process $\, \widetilde{X}_{\cdot}^{\mathrm m} \,$ has the law 
\[
\, \mathrm m \, =\,   \text{Law} ( \widetilde{X}_{t}^{\mathrm m}, 0 \le t \le T ) \, \quad \text{ with } \quad \, \text{Law} ( X_{0}^{\mathrm m}) \, =\, \text{Law} ( \widetilde{X}_{0}^{\mathrm m}) \, =\, \theta \,. 
\]
Here $\, \mathrm m_{t}\,$ in (\ref{eq: NDu0}) is the marginal distribution of $\, \widetilde{X}^{\mathrm m}_{t}\,$ at time $\, t \ge 0\,$. Assume  $\, B_{\cdot}\,$ is independent of the $\,\sigma\,$-field $\,{\bm \sigma} ( \widetilde{X}^{\mathrm m}_{t}, 0 \le t \le T) \vee {\bm \sigma} ( X_{0}^{\mathrm m}) \,$.  

Thanks to the theory (e.g., {Karatzas \& Shreve} (1991)) of stochastic differential equation with Lipschitz condition (\ref{eq: Lip}) and the growth condition (\ref{eq: 1.2a}), a solution $\, X_{\cdot}^{\mathrm m}\, $ of (\ref{eq: NDu0}) exists, given the probability measure $\, \mathrm m \in \mathcal M ( C([0, T], \mathbb R)) \,$, the initial values with the initial law (\ref{eq: INITIAL}) and the associated canonical process $\, \widetilde{X}_{\cdot}^{\mathrm m}\,$ of the law $\, \mathrm m\,$. Hence, the map $\, \Phi\, $ is defined. Indeed, the solution $\, X_{\cdot}^{\mathrm m}\, $ in (\ref{eq: NDu0}) can be given as a functional of $\, \mathrm m\,$, $\, \widetilde{X}_{\cdot}^{\mathrm m}\,$ and $\, B_{\cdot}\,$, i.e., there exists a functional $\, {\bm \Phi} : [0, T] \times \mathcal M ( C([0, T], \mathbb R)) \times C([0, T], \mathbb R) \times C([0, T], \mathbb R) \times \mathbb R \to \mathbb R \,$ such that 
\begin{equation} \label{eq: Phi0}
X_{t}^{\mathrm m} \, =\,  {\bm \Phi} (t,  ({\mathrm m}_{\cdot}) , (\widetilde{X}_{\cdot}^{\mathrm m}),  (B_{\cdot}), X_{0}^{(u)}) \, ; \quad 0 \le t \le T \, , 
\end{equation}
where the value $\,X_{t}^{\mathrm m}\,$ at $\,t\,$ is determined by the initial value $\, X_{0}^{\mathrm m} \, =\,  X_{0}^{(u)}\,$ with the law $\, \theta\,$ and the restrictions $\, (\mathrm m_{s})_{0 \le s \le t}\,$, $\, ( \widetilde{X}_{s}^{\mathrm m})_{0 \le s \le t}\,$, $\, (B_{s})_{0 \le s \le t}\,$ of elements on $\, [0, t]\,$ for $\, 0 \le t \le T\,$. 
Note that here the filtration generated by $\, \widetilde{X}_{\cdot}^{\mathrm m}\,$ is not the Brownian filtration $\, (\mathcal F_{t}^{B})_{t\ge 0}\,$ generated by the fixed Brownian motion $\,B_{\cdot}\,$ but we assume it is independent of $\, (\mathcal F_{t}^{B})_{t \ge 0}\,$. Thus, we cannot expect the solution pair $\, (X^{\mathrm m}_{\cdot}, \widetilde{X}^{\mathrm m}_{\cdot}) \,$ to be a strong solution with respect to the filtration $\, (\mathcal F_{t}^{B})_{t \ge 0}\,$, in general.  

\smallskip 

We shall find a fixed point $\, m^{\ast}\,$ of this map $\, \Phi\,$ in (\ref{eq: Phi}), i.e., $\, \Phi (m^{\ast}) \, =\,  m^{\ast} \, $  
to show the uniqueness of solution to (\ref{eq: NDu}) with (\ref{eq: Fu})-(\ref{eq: NDlaw}) in the sense of probability law. 

\smallskip

For $\, \mathrm m_{i} \in \mathcal M( C([0, T], \mathbb R))\,$ with the intial law $\, \mathrm m_{i}\vert _{\{t= 0\}} = \theta \,$, on a filtered probability space $\, (\Omega, \mathcal F, \mathbb P) \,$ with filtration $\, (\mathcal F_{t})_{t \ge 0}\,$, a fixed Brownian motion $\,B_{\cdot}\,$ on it, the initial values $\, (X_{0}^{(u)}, \widetilde{X}_{0}^{(u)}) \,$ with the joint law $\,\Theta \,$ in (\ref{eq: INITIAL}), and the canonical process $\, \widetilde{X}_{\cdot}^{\mathrm m_{i}}\,$ with the initial value $\, \widetilde{X}_{0}^{\mathrm m_{i}} \, :=\, \widetilde{X}_{0}^{(u)}\,$, let us consider $\, \Phi (\mathrm m_{i}) \, =\,  \text{Law} (X_{t}^{\mathrm m_{i}}, 0  \le t \le T) \,$ in (\ref{eq: Phi}), where $\, (X_{\cdot}^{\mathrm m_{i}},  \, \widetilde{X}_{\cdot}^{\mathrm m_{i}})\,$ satisfies $\, \mathrm m_{i} \, =\, \text{Law} ( \widetilde{X}_{t}^{\mathrm m_{i}}, 0 \le t \le T) \,$ and 
\[
X_{t}^{\mathrm m_{i}} \, =\, X_{0}^{(u)} + \int^{t}_{0}b( s, X_{s}^{\mathrm m_{i}} , u \, \delta_{ \widetilde{X}^{\mathrm m_{i}}_{s}} + ( 1- u) \mathrm m_{i,s}) {\mathrm d} t + {\mathrm d} B_{t} \, ; \quad 0 \le t \le T \, ,\quad i \, =\,  1, 2 \, . 
\]

Then, by the form (\ref{eq: Lin}) of $\, b\,$ with the Lipschitz property (\ref{eq: Lip}) and the standard technique (see e.g., {Sznitman} (1991)) we obtain the estimates 
\begin{equation} \label{ineq: Wd0}
\begin{split}
\lvert X_{s}^{\mathrm m_{1}} - X_{s}^{\mathrm m_{2}} \rvert
& \le \int^{s}_{0} \lvert b(v, X_{v}^{\mathrm m_{1}}, u \delta_{ \widetilde{X}_{v}^{\mathrm m_{1}}} + (1-u) \mathrm m_{1, v} ) - b(v, X_{v}^{\mathrm m_{2}}, u \delta_{ \widetilde{X}_{v}^{\mathrm m_{2}}} + (1-u) \mathrm m_{2, v} ) \rvert  {\mathrm d} v  \\
\, & =\,  \int^{s}_{0} \Big \lvert \int_{\mathbb R} \widetilde{b} ( v, X_{v}^{\mathrm m_{1}}, y) ( u \delta_{ \widetilde{X}_{v}^{\mathrm m_{1}}} ({\mathrm d} y)  + (1-u) \mathrm m_{1, v} ( {\mathrm d} y ) )  \\
& \qquad  \qquad \qquad - \int_{\mathbb R} \widetilde{b} ( v, X_{v}^{\mathrm m_{2}}, y) ( u \delta_{ \widetilde{X}_{v}^{\mathrm m_{2}}} ({\mathrm d} y)  + (1-u) \mathrm m_{2, v} ( {\mathrm d} y )  ) \Big \rvert  {\mathrm d} v \\
& \le u \int^{s}_{0} \Big \lvert  \widetilde{b} ( v, X_{v}^{\mathrm m_{1}}, \widetilde{X}_{v}^{\mathrm m_{1}}) - \widetilde{b} ( v, X_{v}^{\mathrm m_{2}}, \widetilde{X}_{v}^{\mathrm m_{2}}) \Big \rvert {\mathrm d} v  \\
& \qquad {} + (1-u) \int^{s}_{0}  \Big \lvert \int_{\mathbb R} \widetilde{b} ( v, X_{v}^{\mathrm m_{1}}, y) \mathrm m_{1, v}( {\mathrm d} y) - \int_{\mathbb R} \widetilde{b} ( v, X_{v}^{\mathrm m_{2}}, y) \mathrm m_{2, v}( {\mathrm d} y)\Big \rvert {\mathrm d} v \, , 
\end{split}
\end{equation}
where we evaluate the convex combination of the first term 
\begin{equation}  \label{ineq: Wd1}
\int^{s}_{0} \Big \lvert  \widetilde{b} ( v, X_{v}^{\mathrm m_{1}}, \widetilde{X}_{v}^{\mathrm m_{1}}) - \widetilde{b} ( v, X_{v}^{\mathrm m_{2}}, \widetilde{X}_{v}^{\mathrm m_{2}}) \Big \rvert {\mathrm d} v  \le C_{T} \Big ( \int^{s}_{0} ( (\lvert X_{v}^{\mathrm m_{1}} - X_{v}^{\mathrm m_{2}} \rvert  + \lvert \widetilde{X}_{v}^{\mathrm m_{1}} - \widetilde{X}_{v}^{\mathrm m_{2}} \rvert) \wedge 1) {\mathrm d} v \Big) \, 
\end{equation}
for every $\, 0 \le s \le T\,$, and the second term with the integrand  
\begin{equation} \label{ineq: Wd}
\begin{split}
& \Big \lvert \int_{\mathbb R} \widetilde{b} ( v, X_{v}^{\mathrm m_{1}}, y) \mathrm m_{1, v}( {\mathrm d} y) - \int_{\mathbb R} \widetilde{b} ( v, X_{v}^{\mathrm m_{2}}, y) \mathrm m_{2, v}( {\mathrm d} y)\Big \rvert  \\
& \, \le\,   \Big \lvert \int_{\mathbb R} ( \widetilde{b} ( v, X_{v}^{\mathrm m_{1}}, y) - \widetilde{b} ( v, X_{v}^{\mathrm m_{2}}, y)) \mathrm m_{1, v}( {\mathrm d} y) \Big \rvert \\
& \qquad {} +  \Big \lvert  \int_{\mathbb R} \widetilde{b} ( v, X_{v}^{\mathrm m_{2}}, z) \mathrm m_{1, v}( {\mathrm d} z) - \int_{\mathbb R} \widetilde{b} ( v, X_{v}^{\mathrm m_{2}}, y) \mathrm m_{2, v}( {\mathrm d} y)\Big \rvert \\
& \, \le  \, C_{T} ( \lvert X_{v}^{\mathrm m_{1}} - X_{v}^{\mathrm m_{2}}\rvert \wedge 1 )  + C_{T} \, D_{v}( \mathrm m_{1}, \mathrm m_{2}) \, , 
\end{split}
\end{equation}
where $\,D_{v}(\mathrm m_{1}, \mathrm m_{2}) \,$ is the Wasserstein distance in (\ref{eq: Wdist}) between $\, \mathrm m_{1}\,$ and $\, \mathrm m_{2}\,$ in $\, [0, v]\,$ for $\, 0 \le v \le T\,$. Here note that in the last equality of (\ref{ineq: Wd}), we used (\ref{eq: Lip-a}) and an almost-sure inequality 
\[
\Big \lvert  \int_{\mathbb R} \widetilde{b} ( v, X_{v}^{\mathrm m_{2}}, z) \mathrm m_{1, v}( {\mathrm d} z) - \int_{\mathbb R} \widetilde{b} ( v, X_{v}^{\mathrm m_{2}}, y) \mathrm m_{2, v}( {\mathrm d} y)\Big \rvert 
\] 
\[
\, =\,  \Big \lvert  \mathbb E^{1,2} \big [ \widetilde{b} ( v, x, \omega_{1}) - \widetilde{b} ( v, x, \omega_{2})  \big ] \vert_{\{x = X_{v}^{\mathrm m_{2}}\}} \Big \rvert \, \le  \, C_{T} \mathbb E^{1,2} \big [ \lvert \omega_{1,v} - \omega_{2,v} \rvert \wedge 1 \big] \, , 
\]
where $\, \mathbb E^{1,2}\,$ is an expectation under a joint distribution of $\, (\omega_{1,v}, \omega_{2,v})\,$ (the value of $\, (\omega_{1}, \omega_{2}) \in \Omega_{2} \, :=\, C( [0, T], \mathbb R^{2}) \,$ at time $\,v\,$) with fixed marginals $\, \mathrm m_{1,v}\,$ and $\, \mathrm m_{2,v}\,$ for every $\, 0 \le v \le T\,$. Here, since the expectation on the left of $\leq$ only depends on the marginals, taking the infimum on the right of $\leq$ over all the joint distributions with fixed marginals $\, \mathrm m_{1,v}\,$ and $\, \mathrm m_{2,v}\,$, we obtained the last inequality in (\ref{ineq: Wd}) from    
\[
\Big \lvert  \int_{\mathbb R} \widetilde{b} ( v, X_{v}^{\mathrm m_{2}}, z) \mathrm m_{1, v}( {\mathrm d} z) - \int_{\mathbb R} \widetilde{b} ( v, X_{v}^{\mathrm m_{2}}, y) \mathrm m_{2, v}( {\mathrm d} y)\Big \rvert \le C_{T} D_{v} (\mathrm m_{1}, \mathrm m_{2}) \, ; \quad 0 \le v \le T \, . 
\]

Combining (\ref{ineq: Wd0})-(\ref{ineq: Wd}) and  taking the supremum over $\, s \in [0, t]\,$, we obtain 
\[
  \sup_{0 \le s \le t} \lvert X_{s}^{\mathrm m_{1}} - X_{s}^{\mathrm m_{2}} \rvert \wedge 1   \le C_{T} \int^{t}_{0} (\lvert X_{v}^{\mathrm m_{1}} - X_{v}^{\mathrm m_{2}} \rvert \wedge 1){\mathrm d} v +  C_{T}  \int^{t}_{0} (u ( \lvert \widetilde{X}_{v}^{\mathrm m_{1}} -  \widetilde{X}_{v}^{\mathrm m_{2}} \rvert \wedge 1) + (1-u) D_{v}(\mathrm m_{1}, \mathrm m_{2})) {\mathrm d} v  
\]
\[
\le C_{T} \int^{t}_{0} (\sup_{0\le s \le v}\lvert X_{s}^{\mathrm m_{1}} - X_{s}^{\mathrm m_{2}} \rvert \wedge 1){\mathrm d} v +  C_{T}  \int^{t}_{0} (u ( \lvert \widetilde{X}_{v}^{\mathrm m_{1}} -  \widetilde{X}_{v}^{\mathrm m_{2}} \rvert \wedge 1) + (1-u) D_{v}(\mathrm m_{1}, \mathrm m_{2})) {\mathrm d} v  
\]
\noindent for every $\, 0 \le t \le T\,$. Applying  {Gronwall}'s lemma, we obtain 
\[
  \sup_{0 \le s \le t} \lvert X_{s}^{\mathrm m_{1}} - X_{s}^{\mathrm m_{2}} \rvert \wedge 1   \le C_{T} e^{C_{T} T} \int^{t}_{0} (u ( \lvert \widetilde{X}_{v}^{\mathrm m_{1}} -  \widetilde{X}_{v}^{\mathrm m_{2}} \rvert \wedge 1) + (1-u) D_{v}(\mathrm m_{1}, \mathrm m_{2})) {\mathrm d} v 
\]
\[
 \le C_{T} e^{C_{T} T} \int^{t}_{0} (u ( \sup_{0 \le s \le v}\lvert \widetilde{X}_{s}^{\mathrm m_{1}} -  \widetilde{X}_{s}^{\mathrm m_{2}} \rvert \wedge 1) + (1-u) D_{v}(\mathrm m_{1}, \mathrm m_{2})) {\mathrm d} v 
\]
for every $\, 0 \le t \le T\,$. Taking expectations of both sides and taking the infimum over all the joint measures with marginals $\, (\mathrm m_{1}, \mathrm m_{2})\,$ and initial law $\, \Theta\,$ in (\ref{eq: INITIAL}), we obtain 
\begin{equation} \label{eq: est}
\begin{split}
D_{t}( \Phi (\mathrm m_{1}), \Phi (\mathrm m_{2})) \, & \le \, C_{T} e^{C_{T} T} \int^{t}_{0} (u D_{v}(\mathrm m_{1}, \mathrm m_{2}) + (1-u) D_{v}( \mathrm m_{1}, \mathrm m_{2})) {\mathrm d} v \\
& \qquad \qquad \, =\,  C_{T} e^{C_{T} T} \int^{t}_{0} D_{v}(\mathrm m_{1}, \mathrm m_{2}) {\mathrm d} v 
\end{split}
\end{equation}
for every $\, 0 \le t \le T\,$. Note that the upper bound in (\ref{eq: est}) is uniform over $\, u \in [0, 1]\,$. 

For every $\, \mathrm m \in C([0, T], \mathbb R)\,$ with initial marginal law $\, {\mathrm m}\vert_{\{t = 0\}} \, =\,  \theta\,$, iterating (\ref{eq: est}) and the map $\, \Phi\,$, $\,k \,$ times, we observe the inequality 
\begin{equation}
D_{T} ( \Phi^{(k+1)} (\mathrm m ) , \Phi^{(k)}(\mathrm m) ) \le \frac{\,(C_{T}Te^{C_{T} T})^{k}\,}{\,k!\,} \cdot D_{T} ( \Phi (\mathrm m), \mathrm m) \, ; \quad k \in \mathbb N_{0} \, , 
\end{equation}
and hence, we claim $\, \{\Phi ^{(k)}(\mathrm m) , k \in \mathbb N_{0}\} \,$ forms a Cauchy sequence converging to a fixed point $\, \mathrm m^{\ast} \, =\,  \Phi ( \mathrm m ^{\ast}) \,$ of $\, \Phi \,$ on $\, \mathcal M ( C([0, T], \mathbb R) )\,$. This fixed point $\, \mathrm m^{\ast} (\cdot) \, =\, \mathbb P ( X_{\cdot} \in \cdot) \,$ is a weak solution to (\ref{eq: NDu}) with (\ref{eq: Fu})-(\ref{eq: NDlaw}). It is unique in the sense of probability distribution. To relax the condition (\ref{eq: Lip-a}) and to show the result under the weaker condition (\ref{eq: Lip}), we divide the time interval $\,[0, T]\,$ into time-intervals of short length and establish the uniqueness in the short time intervals, and then piece the unique solution together to get the global uniqueness by the standard method.   
\end{proof}

\begin{prop}\label{prop: 1.2} 
{In addition to the assumptions required in Proposition~\ref{prop: 1}, let  $\, \mathbb E [ \lvert X_{0} \rvert] < \infty\,$.}
Then, the solution $\,(X_{\cdot}, \widetilde{X}_{\cdot})\,$, given in Proposition \ref{prop: 1}, satisfies for every $\, T > 0\,$ 
\begin{equation} \label{eq: 1.2b}
\mathbb E [ \sup_{0 \le t \le T} \lvert X_{t}\rvert ] \le \big( \mathbb E[ \lvert X_{0}\rvert ] + \mathbb E [ \sup_{0 \le s \le T} \lvert B_{s} \rvert ]  + C_{T}T\big ) e^{2C_{T} \, T} \,. 
\end{equation}
\end{prop}
\begin{proof} Suppose that $\, (X_{\cdot}, \widetilde{X}_{\cdot})\,$ is the solution to (\ref{eq: NDu}) with (\ref{eq: Fu})-(\ref{eq: NDlaw}) for a fixed $\, u \in [0, 1]\,$. Thanks to (\ref{eq: 1.2a}) and $\,\text{Law} (X_{t}) \, =\, \text{Law} ( \widetilde{X}_{t})\,$, $\, t \ge 0\,$, we have 
\[
\lvert b(s, X_{s}, F_{s}^{(u)}) \rvert \, =\,  \Big \lvert \int_{\mathbb R} \widetilde{b}(s, X_{s}, y ) {\mathrm d} F_{s}^{(u)}(y) \Big \rvert 
\le C_{T}\big ( 1 + u (\lvert X_{s} \rvert + \lvert \widetilde{X}_{s}\rvert) + (1-u) ( \lvert X_{s}\rvert + \mathbb E [ \lvert X_{s}\rvert ] ) \big ) \,  
\]
for $\,0 \le s \le T  \,$. Then we verify (\ref{eq: 1.2b}) by an application of {Gronwall}'s lemma to 
\[
\mathbb E [ \sup_{0\le s \le t} \lvert X_{s} \rvert] \le \mathbb E [ \lvert X_{0}\rvert ]  + \mathbb E [ \sup_{0 \le s \le t} \lvert B_{s}\rvert ]  + C_{T} \mathbb E \Big[ \int^{t}_{0} ( 1 + u  ( \lvert X_{s}\rvert + \lvert \widetilde{X}_{s} \rvert ) + (1 - u)  ( \lvert X_{s}\rvert + \mathbb E [ \lvert X_{s}\rvert ] ) ) {\mathrm d} s \Big] 
\]
\[
\le \mathbb E [ \lvert X_{0}\rvert ]  + \mathbb E [ \sup_{0 \le s \le t} \lvert B_{s}\rvert ]  + C_{T}T+ \int^{t}_{0} 2C_{T} \mathbb E [ \sup_{0 \le u \le s} \lvert X_{u}\rvert ] {\mathrm d} s 
\, ; \quad 0 \le t \le T .  
\]
\end{proof}

\begin{remark}[$\,L^{p}\,$ estimates] \label{rem: 2.4} We may extend  Proposition \ref{prop: 1.2} for the estimates of $\, \mathbb E [ \sup_{0 \le t \le T} \lvert X_{t}\rvert^{p}] \,$, assuming $\, \mathbb E [ \lvert X_{0}\rvert^{p} ] < + \infty \,$ for $\, p \ge 1 \,$. \hfill $\,\square\,$
\end{remark}

\begin{remark}[Extreme cases] \label{rem: 1}
In Proposition \ref{prop: 1} the processes $\, (X^{(u)}_{\cdot}, \widetilde{X}^{(u)}_{\cdot})\,$, $\, u \in [0, 1]\,$ form a class of diffusions which contains two {\it extreme} cases $\, u \, =\, 0, 1\,$: 

\medskip 

\noindent $\,\bullet\,$ When $\, u \, =\, 0\,$, we set $\, (X_{\cdot}^{\bullet}, \widetilde{X}_{\cdot}^{\bullet}) \, :=\, (X^{(0)}_{\cdot}, \widetilde{X}_{\cdot}^{(0)}) \,$ and distinguish it from other cases. $\,X_{\cdot}^{\bullet}\,$ satisfies a {McKean-Vlasov} diffusion equation 
\begin{equation} \label{eq: NDMK}
{\mathrm d} X^{\bullet}_{t} \, =\, b(t, X_{t}^{\bullet}, \, \mathcal L_{ X_{t}^{\bullet}} ) \, {\mathrm d} t + {\mathrm d} B_{t} \, ; \quad t \ge 0 \, , 
\end{equation}
and the corresponding copy $\, \widetilde{X}^{\bullet}_{\cdot}\,$ does not appear, that is, we may take $\, \widetilde{X}_{\cdot}^{\bullet} \,$ independent of $\, X_{\cdot}^{\bullet} \,$ because of the solvability of (\ref{eq: NDMK}) and the restriction (\ref{eq: NDlaw}). 

\medskip 

\noindent $\,\bullet\,$ When $\,u \, =\,  1\,$, we set $\, (X_{\cdot}^{\dagger}, \widetilde{X}_{\cdot}^{\dagger}) \, :=\, (X^{(1)}_{\cdot}, \widetilde{X}_{\cdot}^{(1)}) \,$. The pair satisfies a stochastic equation  
\begin{equation} \label{eq: NDC}
{\mathrm d} X^{\dagger}_{t} \, =\, b(t, X_{t}^{\dagger}, \, \delta_{ \widetilde{X}_{t}^{\dagger}} ) \, {\mathrm d} t + {\mathrm d} B_{t} \, ; \quad t \ge 0 \, ,  
\end{equation}
where $\,\widetilde{X}^{\dagger}_{\cdot}\,$ has the same law as $\, X^{\dagger}_{\cdot}\,$, independent of Brownian motion, i.e., $\, \text{Law}{(X^{\dagger}_{\cdot})}  \, =\, \text{Law}{( \widetilde{X}^{\dagger}_{\cdot})} \,$ and $\, {\bm \sigma} (\widetilde{X}^{\dagger}_{t}, \, t \, \ge\,  0 ) \, \perp \!\!\! \perp \, {\bm \sigma} (B_{t}\, , \, t \ge 0 ) \, \,$. The corresponding non-linear contribution from the law $\, \text{Law}({X^{\dagger}_{\cdot}})\,$ of $\,{X^{\dagger}_{\cdot}}\,$ disappears from (\ref{eq: NDC}). \hfill $\,\square\,$
\end{remark} 

\begin{remark}[Non-uniqueness]  
When $\,u \in (0, 1] \,$, it is simple to observe that the stochastic equation (\ref{eq: NDu}) with (\ref{eq: Fu}) but {\it without} the distributional constraints in (\ref{eq: NDlaw}) does not determine uniquely the joint law  of $\,(X_{\cdot}^{(u)}, \widetilde{X}^{(u)})\,$. For example, take a two-dimensional Brownian motion $\, (B_{\cdot}, W_{\cdot})\,$ and take $\, W_{\cdot}\,$ for $\, \widetilde{X}^{(u)}_{\cdot}\,$, i.e., $\, \widetilde{X}^{(u)}_{\cdot} \equiv W_{\cdot}\,$, then by the standard theory of stochastic differential equations we may construct a weak solution $\, (X^{(u)}_{\cdot}, \widetilde{X}^{(u)}) \,$ for (\ref{eq: NDu}) with (\ref{eq: Fu}), in addition to the solution in Proposition \ref{prop: 1}. In this case, if $\, B_{\cdot}\,$ and $\, W_{\cdot}\,$ are independent, then the independence condition $\, {\bm \sigma } (\widetilde{X}^{(u)}_{t} \equiv W_{t}, \, t \, \ge\,  0 ) \, \perp \!\!\! \perp \, {\bm \sigma} (B_{t}\, , \, t \ge 0 )\, $ holds but $\,\text{Law} ( X_{\cdot}^{(u)}) \, \neq \, \text{Law} ( \widetilde{X}_{\cdot}^{(u)}) \,$, in general. Thus the requirement (\ref{eq: NDlaw}) is crucial for the uniqueness of the solution.  
A recent work of {Bayraktar, Cosso \& Pham} (2018) introduces a pair of continuous stochastic processes coupled through the distribution of initial values without distributional constraints (\ref{eq: NDlaw}) for the study of randomized dynamic programming principle. 
 \hfill $\,\square\,$
\end{remark}

\begin{remark}[Russian nesting doll structure] \label{rem: 1.2} When $\, u \in (0, 1] \,$, since $\, \widetilde{X}_{\cdot}^{(u)}\,$ has the same law as $\, X_{\cdot}^{(u)}\,$ in (\ref{eq: NDlaw}), the dynamics of  $\, \widetilde{X}_{\cdot}\,$ is described by a similar equation as in (\ref{eq: NDu}), i.e., 
\begin{equation}
{\mathrm d} \widetilde{X}^{(u)}_{t} \, =\,  b( t, \widetilde{X}_{t}^{(u)}, \widetilde{F}_{t}^{(u)}) {\mathrm d} t + {\mathrm d} \widetilde{B}_{t} \, ; \quad t \ge 0 \, , 
\end{equation}
where $\,\widetilde{B}\,$ is another Brownian motion but $\, \widetilde{F}_{t}^{(u)}\,$ is defined from another (unknown) copy $\, \widehat{X}_{\cdot}\,$ of $\, X_{\cdot}\,$,
\begin{equation}
\, \widetilde{F}_{t}^{(u)} \, =\, u \cdot \delta_{ \widehat{X}_{t}^{(u)}} + ( 1- u) \cdot \mathcal L_{ \widetilde{X}_{t}^{(u)}} \, ; \quad t \ge 0 \, , 
\end{equation} 
with $\, \text{Law} (X_{\cdot}^{(u)}) \, =\,  \text{Law} ( \widetilde{X}_{\cdot}^{(u)}) \, =\,  \text{Law} ( \widehat{X}_{\cdot}^{(u)})  \,$, and $\, {\bm \sigma} ( \widehat{X}_{t}^{(u)}, t \ge 0 ) \perp \!\!\! \perp \, {\bm \sigma} ( \widetilde{B}_{t}\, , \, t \ge 0 ) \, $. Thus it follows from Proposition \ref{prop: 1} that the dynamics of $\, \widetilde{X}_{\cdot}^{(u)}\,$ depends on $\, \widehat{X}^{(u)}_{\cdot}\,$ and $\, \widetilde{B}_{\cdot}\,$. 

Repeating this argument, we see that the dynamics of $\, \widehat{X}^{(u)}_{\cdot}\,$ may depend on yet another copy and a Brownian motion, and then another copy and a Brownian motion, and so on. This dependence continues, and thus the dynamics of $\, X_{\cdot}^{(u)}\,$ may  depend on the dynamics of infinitely many copies, as if we open  infinitely many layers of Russian nesting doll  ``matryoshka''. Thus when $\,u \in (0, 1] \,$, the infinite-dimensional nesting structure naturally arises in the system (\ref{eq: NDu})-(\ref{eq: Fu}).  
\hfill $\,\square\,$
\end{remark}





\begin{remark} [Generalization and Application] The set-up and conditions on the tamed drift function $\, b \,$ in (\ref{eq: NDu}) can certainly be generalized and relaxed. Also, a Lipschitz continuous diffusion coefficient can be introduced in (\ref{eq: NDu}), instead of the unit diffusion coefficient. For such generalization in the McKean-Vlasov equation, see e.g., Funaki (1984). In a more realistic problem of large network objects (financial networks associated with blockchains, biological networks, neural networks, data networks etc.), it is of interest to analyze a more complicated infinite (random) {\it tree}  structures rather than the simple local interaction of the infinite directed chains considered here. With these generalizations, it may also be  natural to replace the current state space $\, \mathbb R\,$ of each particle by a locally compact, separable metric space $\,E\,$. Here, we take the simplest form (\ref{eq: NDu})-(\ref{eq: INITIAL}) for the presentation of the essential idea of the infinite directed chain interaction. It can be seen as the sparse counterpart of a complete graph (as arising in the mean field setting) among the set of connected graphs. In the setup of unimodular Galton-Watson trees and related large sparse (but undirected) networks, we refer the reader to the interesting work of Lacker, Ramanan \& Wu (2019) which came out after this work had been completed. 

\smallskip 

An interesting application of such generalized models in financial markets is modeling of stochastic volatility structures among financial asset price processes, that is, each $\, {X}_{\cdot} \,$  is a volatility process of a financial asset, so that the volatility processes of the financial assets have both a network structure and a mean-field interaction. The local network structure in this case could tie together companies from similar industries, with similar investments or operating under the same jurisdiction. Similar in the study of systemic risk, particle systems with coupled diffusions generated by sparse network structures are of particular importance. Another interesting direction of research is to identify and explore the directed chain stochastic equations (\ref{eq: NDu})-(\ref{eq: INITIAL}) or their variants, as Nash equilibria of stochastic games, where the representative pair of players interact optimally in the presence of both mean-field and network structure. This program was introduced for the mean-field games in Carmona, Fouque \& Sun (2015) and  substantial work has followed in the context of mean-field games and systemic risk analysis. 
The corresponding problem on chain-like networks is the object of a manuscript under preparation. 
\hfill $\,\square\,$

\end{remark} 

\section{Particle system approximation} \label{sec: InfParticle} 

We interpret the solution pair $\,(X_{\cdot}^{(u)}, \widetilde{X}_{\cdot}^{(u)}) \,$ in Proposition \ref{prop: 1} as a representative pair in the limits of the directed chain particle system (\ref{eq: NDMF}) we introduced in section \ref{sec: Chain}, as $\, n \to \infty\,$. We view $\,X_{\cdot}^{(u)}\,$ as a particle  which corresponds to the {\it head} of an arrow and $\, \widetilde{X}_{\cdot}^{(u)}\,$ as another particle which corresponds to the {\it tail} of the same arrow. Here $\,u\,$ represents the strength of the directed chain dependence, comparative to the mean-field interaction. In this section we shall discuss this interpretation precisely by showing the limiting results in Propositions \ref{prop: 1.4}-\ref{prop: 1.5}, as an extension from the  stochastic chaos  of M. Kac (1956) (or propagation of chaos) towards a local dependence structure with mean-field interaction, and then discuss a fluctuation estimate in Proposition \ref{prop: 2.0}.  

Let us consider a sequence of finite systems of particles $\, (X_{t, i}^{(u)}\, $, $\, t\ge 0 \,$, $\, i \, =\, 1, \ldots , n)\,$, $\, n \in \mathbb N\,$ defined by the system of stochastic differential equations  
\begin{equation}  \label{eq: PAu}
{\mathrm d} X_{t, i}^{(u)} \, =\, b\big( t, X_{t, i}^{(u)}, \widehat{F}_{t, i}^{(u)} \big) {\mathrm d} t  + {\mathrm d} W_{t, i} \, ; \quad t \ge 0 \, , \quad i \, =\, 1, \ldots , n-1 \, , 
\end{equation}
where $\,b: [0, \infty) \times \mathbb R \times \mathcal M (\mathbb R) \to \mathbb R \,$ is defined in (\ref{eq: Lin}) with the same assumptions (\ref{eq: Lip})-(\ref{eq: 1.2a}) as in Proposition \ref{prop: 1}, 
\[
\widehat{F}_{t, i}^{(u)} (\cdot) \, :=\, u \cdot \delta_{ X_{t, i+1}^{(u)} 
} (\cdot) + (1-u) \cdot \frac{1}{n} \sum_{j=1}^{n} \delta_{X_{t, j}^{(u)}} (\cdot) \,,  \quad i \, =\, 1, \ldots , n-1 \, 
\]
with the boundary particle 
\begin{equation} \label{eq: PAub}
{\mathrm d} X_{t, n}^{(u)} \, =\, b \Big( t, X_{t, n}^{(u)},  u \cdot \delta_{X_{t, 1}^{(u)} 
}  + (1-u) \cdot \frac{1}{n} \sum_{j=1}^{n} \delta_{X_{t, j}^{(u)}}  \Big) {\mathrm d} t  + {\mathrm d} W_{t, n} \, . 
\end{equation}

Here $\, W_{\cdot, i}\,$, $\,i \in \mathbb N\,$ are standard independent Brownian motions on a filtered probability space, independent of the initial values $\, X_{0,i}^{(u)}\,$, $\, i \, =\,  1, \ldots , n\,$ and of $\,B_{\cdot}\,$ in (\ref{eq: NDu}). We assume the distribution of $\, X_{0, i}\,$ is {\it common} with $\, \mathbb E [ \lvert X_{0, 1} \rvert^{2} ] < + \infty\,$ for $\, i \, =\, 1, \ldots , n\,$ and {\it independent} of each other. 

Thanks to the assumption on $\,b\,$, the resulting particle system (\ref{eq: PAu})-(\ref{eq: PAub}) is well-defined, and in particular, we have the law invariance  $\,\text{Law} ( X_{\cdot, i}^{(u)}) \, =\,  \text{Law} ( X_{\cdot, 1}^{(u)})  \,$, $\, i \, =\, 1, \ldots , n \,$,  
\begin{equation}\label{eq: symmetry}
 \text{Law} ( X_{\cdot, i}^{(u)}, X_{\cdot, i+1}^{(u)}) \, =\,  \text{Law} ( X_{\cdot, 1}^{(u)}, X_{\cdot, 2}^{(u)}) \, ; \quad i \, =\, 1 \ldots ,  n-1 \, , 
\end{equation}
and more generally, the invariance under the shifts in one direction, i.e., for every fixed $\,k \le n -1 \,$,  
\begin{equation}  \label{eq: symmetry2}
\text{Law} ( X_{\cdot, i}^{(u)}, X_{\cdot, i+1}^{(u)}, \ldots , X_{\cdot, i+k-1}^{(u)}) \, =\,  \text{Law} ( X_{\cdot, 1}^{(u)}, X_{\cdot, 2}^{(u)}, \ldots , X_{\cdot, k}^{(u)}) \, ; \quad i=	1, \ldots , n-k +1\, . 
\end{equation}
Thus, it is natural to write $\, X_{\cdot, n+j}^{(u)} \equiv X_{\cdot, j}^{(u)}\,$, $\,j \, =\,  1, 2, \ldots \,$, so that (\ref{eq: PAu}) and (\ref{eq: symmetry})-(\ref{eq: symmetry2}) hold for $\, i \, =\, 1, \ldots , n\,$. The system  (\ref{eq: NDMF}) in section \ref{sec: Chain} is a time-homogeneous special case of (\ref{eq: PAu}). 

\medskip

Under the setup of Proposition \ref{prop: 1.2} we shall also consider a sequence of finite particle systems $\, \overline{X}_{t, i}\,$, $\, t \ge 0\,$, $\, i \, =\, 1, \ldots , n+1\,$, $\, n \ge 1\,$, defined recursively from the pair $\, (\overline{X}_{\cdot, n}, \overline{X}_{\cdot, n+1}) \, :=\, (X_{\cdot}^{(u)}, \widetilde{X}_{\cdot}^{(u)}) \,$  of the solution to (\ref{eq: NDu}) with (\ref{eq: Fu})-(\ref{eq: INITIAL}), that is, the corresponding stochastic equation 
\begin{equation} \label{eq: 3.4} 
{\mathrm d} \overline{X}_{t, n} \, =\,  b(t, \overline{X}_{t, n}, u \cdot \delta_{ \overline{X}_{t,n+1}} + (1 - u) \cdot \mathcal L_{ \overline{X}_{t,n}} ) {\mathrm d} t + {\mathrm d} W_{t, n} \, ; \quad t \ge 0 \, , 
\end{equation}
and then for $\, j \, =\, n-1, n-2, \ldots , 1\,$,  given $\, \overline{X}_{\cdot, j+1}\,$, we solve 
\begin{equation} \label{eq: 3.5}
{\mathrm d} \overline{X}_{t, j} \, =\,  b( t, \overline{X}_{t, j}, u \cdot \delta_{ \overline{X}_{t, j+1}} + ( 1- u) \cdot \mathcal L_{ \overline{X}_{t, j}} ) {\mathrm d} t + {\mathrm d} W_{t, j} \, ; \quad t \ge 0 \, 
\end{equation}
with the restrictions for each pair $\, ( \overline{X}_{\cdot, j}, \overline{X}_{\cdot, j+1}) \,$, corresponding to (\ref{eq: NDlaw}). As a consequence of the proof of Proposition \ref{prop: 1.2}, we set the common law $\, {\mathrm m}^{\ast} \, =\, \text{Law} ( \overline{X}_{\cdot, i})\,$ for $\, i \, =\, 1, \ldots , n+1\,$, and we also assume the initial values are the same as $\, X_{0, i}^{(u)} \, =\, \overline{X}_{0, i}\,$, $\, i \, =\, 1, \ldots , n\,$ almost surely. Note that when $\,u \, =\,  0\,$, the particle system $\,\overline{X}_{t,i}\,$, $\, i \, =\,  1, \ldots , n+1\,$ induces a product measure; When $\, u \in (0, 1] \,$, the particle system forms a Russian doll nesting structure (see Remark \ref{rem: 1.2} after the proof of Proposition \ref{prop: 1})

For $\, n \ge 1\,$ with $\, X_{\cdot, n+1}^{(u)} \equiv X_{\cdot, 1}^{(u)}\,$ let us assign the weight $\,1/ n\,$ to the Dirac measure at $\, (X_{t,i}^{(u)}, X_{t,i+1}^{(u)})\,$ for $\,i \, =\,  1, \ldots , n\,$, and consider the law of the joint empirical measure process 
\begin{equation} \label{eq: empMeas}
\mathrm M_{t,n} \, :=\,  \frac{\,1\,}{\,n\,} \sum_{i=1}^{n} \delta_{(X_{t,i}^{(u)}, X_{t,i+1}^{(u)})} \,, \quad \text{ with the marginal } 
 \quad  \mathrm m_{t, n} \, :=\, \frac{1}{\, n \, }\sum_{i=1}^{n} \delta_{X_{t,i}^{(u)}} \,, \quad  0 \le t \le T\, , 
\end{equation}
in the space $\mathcal M ( \Omega_{2}) $ of probability measures on the topological space $\, \Omega_{2} \, :=\, D ([0, T], ( \mathcal M ( \mathbb R^{2}) , \lVert \cdot \rVert_{1}))  \,$ of c\`adl\`ag functions on $\,[0, T]\,$ equipped with the Skorokhod topology, where $ (\mathcal M ( \mathbb R^{2}), \lVert \cdot \rVert_{1}) $ is the space of probability measures on $\,\mathbb R^{2}\,$ equipped with the metric $\, \lVert \mu - \nu \rVert_{1} \, :=\,  \sup_{f } \int_{\mathbb R^{2}} f(x) {\mathrm d} (\mu - \nu)(x)  \,$. Here the supremum is taken over the bounded Lipschitz functions $\,f: \mathbb R^{2} \to \mathbb R\,$  with $\,\sup_{x \in \mathbb R^{2}} \lvert f(x) \rvert \le 1\,$ and $\, \sup_{x,y \in \mathbb R^{2}} \lvert f(x) - f(y) \rvert/ \lVert x - y\rVert \le 1\,$. By the construction the sequence of the law of the initial empirical measure converges to the Dirac measure concentrated in $\, \mathrm M_{0}\,$ (say), i.e., 
\begin{equation} \label{eq: convInit}
\text{Law} ( \mathrm M_{0, n}) \xrightarrow[n\to \infty]{} \delta_{\mathrm M_{0}} \quad \text{ weakly in } \quad \mathcal M ( (\mathcal M(\mathbb R^{2}), \lVert \cdot \rVert_{1} )) \, . 
\end{equation}
We denote by $\, \mathrm m_{0} ({\mathrm d} y) \, :=\, \mathrm M_{0} ( \mathbb R \times {\mathrm d} y) \, =\, \mathrm M_{0}  (  {\mathrm d} y \times \mathbb R) \,$ the marginal of $\, \mathrm M_{0} \, =\, \mathrm m_{0} \otimes \mathrm m_{0}\,$. 

\begin{prop} \label{prop: 1.4} Fix $\,u \in [0, 1]\,$. Under the same assumptions for the functional $\,b\,$ as in Proposition \ref{prop: 1.2}, the law of joint empirical measure process $\, \mathrm M_{\cdot , n }\,$, defined in $(\ref{eq: empMeas})$, of the finite particle system $(\ref{eq: PAu})$ with $\, X_{\cdot, n+1}^{(u)} \equiv X_{\cdot, 1}^{(u)}\,$ converges in $\, \mathcal M ( \Omega_{2}) \,$ to the Dirac measure concentrated in the deterministic measure-valued process $\, \mathrm M_{t}, 0 \le t \le T\,$, as $\,n \to \infty\,$, i.e., 
\begin{equation} \label{eq: LLN1}
\lim_{n\to \infty} \text{\rm Law} \big( \mathrm M_{t,n}, 0 \le t \le T \big) \, =\,  \delta_{(\mathrm M_{t}, 0 \le t \le T)} \quad \text{ in } \quad \mathcal M (\Omega_{2}) \, . 
\end{equation}
The marginal laws of $\, \mathrm M_{\cdot}\,$ are the same, i.e., 
\begin{equation} \label{eq: CT-I}
\, \mathrm M_{t}( \mathbb R \times  {\mathrm d} y) \, =\, \mathrm M_{t} ( {\mathrm d} y \times \mathbb R) \, =:\, \mathrm m_{t} ({\mathrm d} y) \,; \quad \, 0 \le t \le T\,, 
\end{equation}
and the joint $\,\mathrm M_{\cdot}\,$ and its marginal $\, \mathrm m_{\cdot}\,$ satisfy the integral equation 
\begin{equation} \label{eq: IntEq}
\int_{\mathbb R} g(x)  \mathrm m_{t}( {\mathrm d} x) \, =\,   \int_{\mathbb R} g(x) \mathrm m_{0}( {\mathrm d} x) + \int^{t}_{0}[ \mathcal A_{s}(\mathrm M) g] \, {\mathrm d}s \, ; \quad 0 \le t \le T 
\end{equation}
for every test function $\,g \in C^{2}_{c}(\mathbb R)\,$, where 
\[
\mathcal A_{s}(\mathrm M) g \, :=\, u \int_{\mathbb R^{2}} \widetilde{b}(s, y_{1}, y_{2}) g^{\prime}(y_{1}) \mathrm M_{s}( {\mathrm d} y_{1} {\mathrm d} y_{2}) + (1-u) \int_{\mathbb R^{2}} \widetilde{b}(s, y_{1}, y_{2}) g^{\prime}(y_{1}) \mathrm m_{s}( {\mathrm d} y_{1}) \mathrm m_{s} ( {\mathrm d} y_{2}) 
\]
\begin{equation} \label{eq: Ainfinty}
\quad {} + \frac{\,1\,}{\,2\,} \int_{\mathbb R} g^{\prime\prime}(y_{1}) \mathrm m_{s}( {\mathrm d} y_{1}) \, ; \quad 0 \le s \le T \, .  
\end{equation}

Moreover, $\, \mathrm M_{\cdot}\,$ is the joint distribution of the solution pair $\, (X_{\cdot}^{(u)}, \widetilde{X}_{\cdot}^{(u)})\,$ of $(\ref{eq: NDu})$ with $(\ref{eq: Fu})$-$(\ref{eq: INITIAL})$, uniquely characterized by $(\ref{eq: CT-I})$-$(\ref{eq: Ainfinty})$  in $\,\Omega_{1}\,$ with the common marginal $\, \mathrm m_{\cdot} \, =\,  \text{Law} ( X_{\cdot}) \, =\, \text{Law} ( \widetilde{X}_{\cdot}) \,$ 
in Proposition \ref{prop: 1}. 
\end{prop}

\begin{remark} \label{rem: 3.1} 
$\,\bullet\,$ When $\,u \, =\, 0\,$, the integral equation (\ref{eq: IntEq}) for $\, \mathrm M_{\cdot}\,$ reduces to the \textsc{McKean-Vlasov} nonlinear integral equation only for the marginal $\, \mathrm m_{\cdot}\,$, i.e., for $\, 0 \le t \le T\,$ and $\, g \in C^{2}_{c} (\mathbb R)\,$ 
\[
\int_{\mathbb R} g(x) \mathrm m_{t} ({\mathrm d} x ) \, =\, \int_{\mathbb R} g(x) \mathrm m_{0} ( {\mathrm d} x) + \int^{t}_{0} {\mathrm d} s \Big[ \int_{\mathbb R^{2}} \widetilde{b}(s, y_{1}, y_{2}) g^{\prime}(y_{1}) \mathrm m_{s}({\mathrm d} y_{1}) \mathrm m_{s} ({\mathrm d} y_{2}) + \frac{\,1\,}{\,2\,} \int_{\mathbb R} g^{\prime\prime} (y_{1}) \mathrm m_{s}({\mathrm d} y_{1}) \Big] \, .  
\]

\noindent $\,\bullet\,$ When $\, u \in ( 0, 1] \,$, the integral equation (\ref{eq: IntEq}) has an infinite-dimensional feature, because of marginal distributional constraints (\ref{eq: CT-I}), as we discussed in Remark \ref{rem: 1.2}, i.e., the joint distribution $\, \mathrm M_{\cdot}\,$ appears in the infinitesimal generator (\ref{eq: Ainfinty}) for the marginal distribution.  \hfill $\,\square\,$
\end{remark}

\begin{proof}
The idea of the proof utilizes the assumptions on the coefficient $\,b\,$ as in Proposition \ref{prop: 1.2} and the law invariance (\ref{eq: symmetry}) of the finite particle system (\ref{eq: PAu}). We take the martingale approach discussed in {Oelschl\"ager} (1984). By the standard argument with  {Gronwall}'s lemma we claim 

\begin{lm} \label{lm: 1} $(a)$ 
For the joint empirical measure processes $\,\mathrm M_{\cdot, n}  \,$ and its marginal $\, \mathrm m_{\cdot, n} \,$ there exist constants $\,c_{k} (> 0) \,$, $\,k \, =\, 1, \ldots, 4 \,$ such that 
\[
e^{- c_{1}t} \int_{\mathbb R} \lvert x \rvert^{2} {\mathrm d} \mathrm m_{t,n}(x)  - c_{2} t  \,, 0 \le t \le T \quad \Big( e^{ c_{3}t} \int_{\mathbb R} \lvert x \rvert^{2} {\mathrm d} \mathrm m_{t,n}(x) + c_{4} t \,, 0 \le t \le T , \text{\it respectively } \Big) 
\]
is a supermartingale (submartingale, respectively), and hence, so is 
\[
e^{- c_{1} t} \int_{\mathbb R^{2}} \lVert y \rVert^{2} {\mathrm d} \mathrm M_{t,n}(y)  - 2 c_{2} t  \,, 0 \le t \le T \quad \Big( e^{ c_{3} t}   \int_{\mathbb R^{2}} \lVert y \rVert^{2} {\mathrm d} \mathrm M_{t,n}(y) + 2 c_{4} t    \,, 0 \le t \le T , \text{ respectively } \Big) \, , 
\] 
because $\, \sum_{i=1}^{n} \lvert X_{\cdot,i}^{(u)}\rvert^{2} \, =\,  \sum_{i=1}^{n} \lvert X_{\cdot, i+1}^{(u)} \rvert^{2} \, =\, (1/2) \sum_{i=1}^{n} (\lvert X_{\cdot,i}^{(u)}\rvert^{2} +  \lvert X_{\cdot, i+1}^{(u)} \rvert^{2} )\,$.  

$(b)$ Moreover, there exist constants $\,c_{k}\,$, $\, k \, =\,  5, 6\,$, such that for every $\, t \le s \le T\,$
\[
\frac{\,1\,}{\,n\,} \sum_{i=1}^{n}\mathbb E \big[ 2 b(s, X_{s,i}^{(u)}, \widehat{F}_{t,i}^{(u)} ) (X_{s, i}- X_{t, i}) \vert \mathcal F_{t}] \le c_{5} \int_{\mathbb R} \lvert x\rvert^{2} {\mathrm d} m_{t, n}(x) + c_{6} \, . 
\]
\end{lm}

Using this lemma and the {Cauchy-Schwarz}  inequality, we claim that 
\begin{equation*}
\begin{split}
\frac{\,1\,}{\,n\,} \sum_{i=1}^{n} \mathbb E [ \lvert X_{t+\delta, i}^{(u)} - X_{t, i}^{(u)} \rvert \vert \mathcal F_{t} ] &\le \Big( \frac{\,1\,}{\,n\,} \sum_{i=1}^{n}  \mathbb E [ \lvert X_{t+\delta, i}^{(u)} - X_{t, i}^{(u)} \rvert^{2} \vert \mathcal F_{t} ]  \Big)^{1/2}
\\
\, &=\,  \Big(  \int^{t+\delta}_{t} \frac{\,1\,}{\,n\,} \sum_{i=1}^{n} \Big( \mathbb E \Big[  2 (b, s, X_{s, i}, \widehat{F}_{s}^{(u)})(X_{s, i} - X_{t, i}) \Big \vert \mathcal F_{t}\Big] + 1 \Big) {\mathrm d} s \Big)^{1/2}
\\
& \le \Big( c_{5} \int_{\mathbb R} \lvert x\rvert^{2} {\mathrm d} m_{t, n}(x) + c_{6} + 1 \Big) ^{1/2} \sqrt{\delta} 
\end{split}
\end{equation*} 
for $\, 0 \le t \le T-\delta\,$. 
Thus, using these inequalities again with the super martingale property, we claim that there exists an $\,\mathcal F_{T}\,$-measurable random variable $\,\, \mathfrak f(\delta) \,$, such that 
\begin{equation} \label{eq: diffbnd}
\frac{\,1\,}{\,n\,} \sum_{i=1}^{n} \mathbb E [ (\lvert X_{t+\delta, i}^{(u)} - X_{t, i}^{(u)} \rvert^{2}+ \lvert X_{t+\delta, i+1}^{(u)} - X_{t, i+1}^{(u)} \rvert^{2})^{1/2}  \vert \mathcal F_{t} ] \le \mathbb E \big [ \, \mathfrak f(\delta) \, \vert \, \mathcal F_{t} \big ] \, ; \quad 0 \le t \le T-\delta \, , 
\end{equation}
with $\, \lim_{\delta \to 0} \sup_{n \ge 1}\mathbb E [ \mathfrak f(\delta) ] \, =\,  0 \,$. Here we set $\, X_{\cdot, n+1}^{(u)} \equiv X_{\cdot, 1}^{(u)}\,$. 

Moreover, by the super/submartingale properties in Lemma \ref{lm: 1} (a) we may evaluate the total variation $\,\mathrm  \lVert \mathrm M_{t,n} \vert_{B_{\lambda}^{c}} \rVert_{TV}\,$ of $\, \mathrm M_{\cdot, n}\,$ restricted outside the ball $\,B_{\lambda} \, :=\, \{x \in \mathbb R^{2} : \lVert x \rVert \le \lambda \} \,$ of radius $\, \lambda ( > 0) \,$, i.e., for every $\, \varepsilon > 0 \,$ 
\begin{equation*}
\begin{split}
\mathbb P ( \sup_{0 \le t \le T} \lVert \mathrm M_{t,n} \vert_{ B_{\lambda}^{c} }\rVert_{TV} > \varepsilon) & \le  \mathbb P \Big( \sup_{0 \le t \le T} \int_{\mathbb R^{2}} \lVert y\rVert^{2} {\mathrm d} \mathrm M_{t, n} > \lambda^{2} \varepsilon \Big) \\
&  \le \mathbb P \Big( \sup_{0 \le t \le T}   e^{c_{1} t} \Big( \int_{\mathbb R^{2}} \lVert y\rVert^{2} {\mathrm d} \mathrm M_{t, n}  +2 c_{2} t \Big) > \lambda^{2} \varepsilon \Big) \\
& \le \frac{\,1\,}{\,\lambda^{2}\varepsilon\,} \mathbb E \Big  [ e^{c_{1}T} \Big( \int_{\mathbb R^{2}} \lVert y\rVert^{2} {\mathrm d} \mathrm M_{T, n}  + 2c_{2} T \Big)  \Big]  \\
& \le \frac{\,1\,}{\,\lambda^{2}\varepsilon\,} \mathbb E \Big  [  \Big( \int_{\mathbb R^{2}} \lVert y\rVert^{2} {\mathrm d} \mathrm M_{0, n}  + 2c_{4} T \Big) e^{(c_{1} + c_{3})T} + 2c_{2} T  \Big] \, . 
\end{split}
\end{equation*}
Taking sufficiently large $\, \lambda\,$, using {Prohorov}'s theorem, we claim that $\, (\mathrm M_{t, n}\,$, $\, 0 \le t \le T)\,$, $\,n \ge 1\,$ of the empirical measures is tight in $\, (\mathcal M ( \mathbb R^{2}), \lvert \cdot \rvert_{1})\,$. Then combining this observation with (\ref{eq: diffbnd}), we claim by Theorem 8.6 (b)  of {Ethier \& Kurtz} (1986) that the sequence $\,( \mathrm M_{t, n}\,$, $\, 0 \le t \le T)\,$, $\,n \ge 1\,$ is relatively compact in the space $\, \mathcal M ( \Omega_{2}) \,$, where $\, \mathcal M (\Omega_{2}) \,$ is equipped with the weak topology.

We shall characterize the limit points of $\, (\mathrm M_{t, n}, 0 \le t \le T)_{n \ge 1} \,$ as $\,n \to \infty\,$. Let us call a limit law $\, \mathrm M_{t}, 0 \le t \le T\,$. Thanks to the law invariance in the construction of (\ref{eq: PAu}), its marginals must be the same for every limit point, i.e., $\, \mathrm M_{t} ( \mathbb R \times {\mathrm d} y) \, =\,  \mathrm M_{t}( {\mathrm d} y \times \mathbb R ) \, =: \, \mathrm m_{t}( {\mathrm d} y)  \,$, $\, y \in \mathbb R\,$ with the initial marginal measure $\,  \mathrm m_{0} ( {\mathrm d} y) \,$.  Applying  {It\^o}'s formula to the system (\ref{eq: PAu}), we see 
\[
f ( \langle \mathrm m_{t, n}, g \rangle) \, -\,  f( \langle \mathrm m_{0, n}, g \rangle)  -  \int^{t}_{0} f^{\prime}( \langle \mathrm m_{s, n}, g \rangle) \Big(\frac{\,1\,}{\,n\,} \sum_{i=1}^{n} g^{\prime} ( X_{s, i}^{(u)}) b ( s, X_{s, i}^{(u)}, \widehat{F}_{s, i}^{(u)} ) + \frac{\,1\,}{\,2\,} \langle \mathrm m_{s, n}, g^{\prime\prime} \rangle\Big) {\mathrm d} s 
\]
\[
- \frac{\,1\,}{\,2\,}\int^{t}_{0} f^{\prime\prime} ( \langle \mathrm m_{s, n}, g  \rangle ) \frac{1}{\, n^{2}\, } \sum_{i=1}^{n} \lvert g^{\prime} ( X_{s, i}^{(u)})\rvert^{2} {\mathrm d} s  \, =\,  \int^{t}_{0} f^{\prime} ( \langle \mathrm m_{s, n}, g \rangle) \frac{\,1\,}{\,n\,} \sum_{i=1}^{n} g^{\prime}( X_{s,i}^{(u)}) {\mathrm d} W_{s, i}  \, , 
\]
is a martingale for every $\, f \in C^{2}_{b}( \mathbb R) \,$, $\, g \in C^{2}_{c}(\mathbb R)\,$, where we use the notation $\, \langle \mu, g \rangle \, :=\, \int_{\mathbb R} g(x) {\mathrm d} \mu(x) \,$ for $\mu \in \mathcal M (\mathbb R) $. Taking the limits with (\ref{eq: convInit}) and using the equivalence of certain martingales, we observe that $\,\exp (\sqrt{-1}\,  \theta \, \eta_{t} )\,$, $\, 0 \le t \le T\,$ is a martingale for every $\, \theta \in \mathbb R\,$, where we define 
\[
\eta_{t}\, :=\,  \int_{\mathbb R} g (x) {\mathrm d} \mathrm m_{t}(x) - \int^{t}_{0} [\mathcal A_{s} ( \mathrm M_{s}) g] {\mathrm d} s \, 
\]
and $\, \mathcal A_{t} ( \mathrm M_{t}) g\,$ as in (\ref{eq: Ainfinty}) for $\, 0 \le t \le T\,$. This implies that the characteristic function of $\,\eta_{t}\,$ satisfies $\, 
\mathbb E [ e^{\sqrt{ -1} \theta \eta_{t}} ] \, =\,  \mathbb E [ e^{\sqrt{ -1} \theta \eta_{0}} ] \, =\,  e^{\sqrt{ -1} \theta \langle \mathrm m_{0}, g \rangle} \, $ for $\, 0 \le t \le T\,$, $\, \theta \in \mathbb R\,$, and hence, $\, \eta_{t} \, =\, \langle \mathrm m_{0}, g \rangle \,$ for every $\,t\,$ in any countable subset of $\, [ 0, T]\,$ and for every $\,g \,$ in any countable subset of $\, C^{2}_{c} ( \mathbb R)\,$. Because of the separability of $\,C^{2}_{c} (\mathbb R) \,$ and right continuity of $\, t \mapsto \mathrm M_{t}\,$, we obtain 
\[
\int_{\mathbb R} g (x) {\mathrm d} \mathrm m_{t}(x) - \int^{t}_{0} [\mathcal A_{s} ( \mathrm M_{s}) g] {\mathrm d} s  \, = \, \eta_{t} \, =\, \langle \mathrm m_{0}, g \rangle \,  \, =\, \int_{\mathbb R} g (x) {\mathrm d} \mathrm m_{0}(x)
\]
 for every $\, 0 \le t \le T\,$ and $\, g \in C^{2}_{c} (\mathbb R)\,$. Thus we claim $\, \mathrm M_{\cdot}\,$ satisfies the integral equation (\ref{eq: IntEq}). With the uniqueness in Proposition \ref{prop: 1} the last part of Proposition \ref{prop: 1.4} can be shown as in Lemmas 8-10 of {Oelschl\"ager} (1984). 
\end{proof}

Proposition \ref{prop: 1.4} describes the limiting system of (\ref{eq: PAu})-(\ref{eq: PAub}), in terms of the joint distribution of two adjacent particles of the directed chain structure (\ref{eq: NDu})-(\ref{eq: Fu}). 
Now let us fix $\,k (\ge 2)\,$ and define the empirical measure process $\, \mathrm M_{t,n}^{(j)}\,$, $\, t \ge 0\,$ of $\,j\,$ consecutive particles from (\ref{eq: PAu})-(\ref{eq: PAub}) 
\begin{equation} \label{eq: empMeasure2}
\mathrm M_{t,n}^{(j)} \, :=\, \frac{1}{\, n \, }\sum_{i=1}^{n} \delta_{(X_{t,i}^{(u)}, \, \ldots , \,X_{t,i+j-1}^{(u)})}  \, ; \quad j \, =\, 2, \ldots , k \, 
\end{equation} 
with $\, \mathrm M_{t,n}^{(1)} \, :=\, \mathrm m_{t,n}\,$ and $\, \mathrm M_{t,n}^{(2)} \equiv  \mathrm M_{t,n} \,$ as in (\ref{eq: empMeas}) in the space $\, \mathcal M (\Omega_{j})\,$ of probability measures on the topological space $\, \Omega_{j} \, :=\, D([0, T], (\mathcal M(\mathbb R^{j}), \lVert \cdot \rVert_{1}) )\,$ of c\'adl\'ag functions on $\,[0, T]\,$, equipped with the Skorokhod topology, where $\, (\mathcal M (\mathbb R^{j}), \lvert \cdot \rvert_{1}) \,$ is the space of probability measures on $\, \mathbb R^{j}\,$, a natural extension of  $\, (\mathcal M (\mathbb R^{2}), \lvert \cdot \rvert_{1}) \,$ defined in the above for $\, j \, =\, 2, \ldots ,  k \,$. We shall consider their limits.

By the construction and the law of large numbers for the initial empirical measure, as in (\ref{eq: convInit}),  
\begin{equation}
\text{Law} ( \mathrm M_{0,n}^{(k)}) \xrightarrow[n\to \infty]{} \delta_{\mathrm M_{0}^{(k)}} \quad \text{ weakly in }  \quad \mathcal M ( (\mathcal M (\mathbb R^{k}), \lVert \cdot \rVert_{1} ) ) \, ,  
\end{equation} 
where $\, \mathrm M_{0}^{(k)} \, :=\, \mathrm m_{0}^{\otimes k} \,$ is the $\,k\,$-tuple product measure of $\,\mathrm m_{0} \, =\,  \text{Law} (X_{0,1}^{(u)})\,$. For $\, j \, =\,  1, \ldots , k+1 \, $ let us denote by $\,\mathrm M_{\cdot}^{(j)}\,$ the joint probability measure induced by $\, (\overline{X}_{\cdot,1}, \ldots , \overline{X}_{\cdot, j}) \,$ in (\ref{eq: 3.4})-(\ref{eq: 3.5}), i.e., 
\begin{equation} \label{eq: j-induced}
(\mathrm M_{t}^{(j)}, t \ge 0)   \, :=\, \text{Law} (  (\overline{X}_{t,1}, \ldots , \overline{X}_{t, j}), t \ge 0 ) \, ; \quad j \, =\,  1, \ldots , k +1\, . 
\end{equation}

\begin{prop} \label{prop: 1.5} Fix $\, u \in [0, 1]\,$. Under the same assumptions as in Proposition \ref{prop: 1.4}, the law of the joint empirical measure process $\, {\mathrm M}_{\cdot, n}^{(k)}\,$ defined in $(\ref{eq: empMeasure2})$ converges in $\, \mathcal M (\Omega_{k}) \,$ to the Dirac measure concentrated in the deterministic measure-valued process $\, {\mathrm M}_{\cdot}^{(k)}\,$ in $($\ref{eq: j-induced}$)$, i.e.,  
\begin{equation} \label{eq: LLN2}
\lim_{n \to \infty} \text{\rm Law} ( {\mathrm{M}}_{t,n}^{(k)}, 0 \le t \le T) \, =\,  \delta_{( {\mathrm{M}}_{t}^{(k)}, \,\, 0 \, \le \, t \, \le \, T)} \quad \text{ in } \quad \mathcal M (\Omega_{k}) \, . 
\end{equation} 
All the consecutive marginals of $\, \mathrm{M}_{\cdot}^{(k)}\,$ are the same, i.e., 
\begin{equation} \label{eq: CT-I2}
\begin{split}
\mathrm{M}_{t}^{(j)}(\mathbb R \times  {\mathrm d} y_{1} \times \cdots \times {\mathrm d} y_{j-1} ) \, &=\,  {\mathrm{M}}_{t}^{(j)}({\mathrm d} y_{1} \times \cdots \times {\mathrm d} y_{j-1} \times \mathbb R) \, , 
\\
\mathrm{M}_{t}^{(j)}(\mathbb R^{2} \times  {\mathrm d} y_{1} \times \cdots \times {\mathrm d} y_{j-2} ) \, &=\,   {\mathrm{M}}_{t}^{(j)}(\mathbb R \times {\mathrm d} y_{1} \times \cdots \times {\mathrm d} y_{j-2} \times \mathbb R) \, =\, {\mathrm{M}}_{t}^{(j)}({\mathrm d} y_{1} \times \cdots \times {\mathrm d} y_{j-2} \times \mathbb R^{2}) \,,  
\\
 \ldots \, ,\quad \mathrm{M}_{t}^{(j)}(\mathbb R^{j-1} \times  {\mathrm d} y_{1} ) \, &=\, \mathrm{M}_{t}^{(j)} ( \mathbb R^{j-2} \times {\mathrm d} y_{1} \times \mathbb R) \, =\,  \cdots \, =\,  {\mathrm{M}}_{t}^{(j)}({\mathrm d} y_{1} \times \mathbb R^{j-1}) \,  
\end{split}
\end{equation} 
for $\, j \, =\, 2, \ldots , k \,$, $\, 0 \le t \le T \,$, and they also satisfy the system of integral equations  
\begin{equation} \label{eq: IntEq2}
\int_{\mathbb R^{j}} g(x) {\mathrm M}_{t}^{(j)} ({\mathrm d} x) \, =\,  \int_{\mathbb R^{j}} g(x) {\mathrm M}_{0}^{(j)}( {\mathrm d} x ) + \int^{t}_{0} [ \mathcal A_{s}^{(j)}( {\mathrm M}^{(j+1)} ) g] {\mathrm d} s \, ; \quad 0 \le t \le T\, , \, \,  j \, =\, 2,  \ldots , k-1 \, 
\end{equation} 
for every test function $\, g \in C^{2}_{c}(\mathbb R^{j}) \,$, where 
\begin{equation} \label{eq: Ainfinity2}
\begin{split}
\mathcal A_{s}^{(j)}({\mathrm M}^{(j+1)}) g \, :&=\, u \int_{\mathbb R^{j+1}} \sum_{\ell=1}^{j} \widetilde{b}(s, y_{\ell}, y_{\ell+1}) \frac{\partial g}{\partial x_{\ell}} ( y_{1}, \dots , y_{j}) {\mathrm M}^{(j+1)}_{s}( {\mathrm d} y_{1} \cdots {\mathrm d} y_{j+1}) 
\\
& \, \, {} + (1-u) \int_{\mathbb R^{j+1}}  \sum_{\ell=1}^{j} \widetilde{b} (s, y_{\ell}, y_{\ell+1})  \frac{\partial g}{\partial x_{\ell}} ( y_{1}, \dots , y_{j}) {\mathrm M}_{s}^{(j)}({\mathrm d}y_{1} \cdots {\mathrm d} y_{j}) {\mathrm m}_{s}( {\mathrm d} y_{j+1}) \, 
\\
& \, \, {} + \frac{1}{\, 2 \, } \int_{\mathbb R^{j}} \sum_{\ell=1}^{j} \frac{\partial^{2} g}{\partial x_{\ell}^{2}} ( y_{1}, \dots , y_{j}) {\mathrm M}_{s}^{(j)}({\mathrm d}y_{1} \cdots {\mathrm d} y_{j}) \, ; \quad 0 \le s \le T \, , \, \, j \, =\,  2, \ldots , k-1 \, .  
\end{split}
\end{equation} 
\end{prop}

\begin{proof} We have shown (\ref{eq: LLN2}) in the case $\,k \, =\, 2\,$ in Proposition \ref{prop: 1.4}. The relative compactness proof of $\, (\mathrm M_{t, n}^{(j)}, 0 \le t \le T)\, $, $\, n \ge 1\,$ in $\, \mathcal M (\Omega_{j}) \,$ follows as in Proposition \ref{prop: 1.4} {\it mutatis mutandis} for $\,j \, =\, 1, \ldots, k \,$. The limit points of $\, (M_{t, n}^{(k)}, 0 \le t \le T)\,$ in (\ref{eq: LLN2}) as $\, n \to \infty\,$ are characterized by (\ref{eq: IntEq2}), because for every test function $\, g \in C^{2}_{c}(\mathbb R^{j})\,$, thanks to {It\^o}'s formula,  it follows from 
\[
\hspace{-4cm} {\mathrm d} g ( X_{t, i}^{(u)}, \ldots , X_{t, i+j-1}^{(u)}) \, =\, \sum_{\ell=1}^{j} \frac{\partial g}{\partial x_{\ell}} ( X_{t,i}^{(u)}, \ldots , X_{t, i+j-1}^{(u)}) {\mathrm d} X_{t, \ell + i - 1}^{(u)} 
\]
\[
\hspace{4cm} {} + \frac{1}{\, 2\, }\sum_{\ell=1}^{j} \sum_{m=1}^{j} \frac{\partial^{2} g}{\partial x_{\ell} \partial x_{m}} ( X_{t,i}^{(u)}, \ldots , X_{t, i+j-1}^{(u)}) {\mathrm d} \langle X_{\cdot,\ell+i-1}^{(u)}, X_{\cdot, m+i-1}^{(u)}\rangle_{t} 
\]
that  $\, \langle \mathrm M^{(j)}_{t, n} , g \rangle \, :=\, \int_{\mathbb R^{j}} g(x) {\mathrm d} \mathrm  M^{(j)}_{t, n}( {\mathrm d} x) \,$, $\,j \, =\,  1, \ldots , k\,$ satisfy that for $\, 0 \le t \le T\,$, $\,j \, =\, 1, \ldots , k\,$, 
\[
\hspace{-5cm} f( \langle \mathrm  M_{t,n}^{(j)}, g \rangle) - f( \langle \mathrm  M_{0, n}^{(j)}, g \rangle - \int^{t}_{0} f^{\prime} ( \langle \mathrm  M_{s, n}^{(j)}, g \rangle ) [\mathcal A^{(j)} \mathrm  M_{\cdot, n}^{(j+1)} g ] {\mathrm d} s 
\]
\[
{} -  \frac{1}{2 n^{2}} \int^{t}_{0} f^{\prime\prime} ( \langle \mathrm M^{(j)}_{s, n} , g \rangle ) \sum_{i=1}^{n} \sum_{\ell=1}^{j} \Big( \frac{ \partial g }{\partial x_{\ell}}\Big)^{2}{\mathrm d} s \, =\,  \int^{t}_{0} f^{\prime}( \langle \mathrm M_{s, n}^{(j)}, g \rangle ) \frac{1}{n}\sum_{i=1}^{n} \sum_{\ell=1}^{j} \Big( \frac{\partial g}{\partial x_{\ell}} \Big) {\mathrm d} W_{s, \ell}\, ,  
\]
for $\,f \in C^{2}(\mathbb R)\,$, where $\, [ \mathcal A^{(j)}_{s} \mathrm M_{\cdot, n}^{(j+1)} g ]\,$ is defined as in (\ref{eq: Ainfinity2}). The condition (\ref{eq: CT-I2}) follows from the construction of (\ref{eq: 3.4})-(\ref{eq: 3.5}). Applying the martingale argument again as in the proof of Proposition \ref{prop: 1.4}, we conclude the proof. 
\end{proof}


\begin{remark}[Propagation of chaos] \label{rem: PoC} If $\, u \, =\,  0\,$ and if the initial law of $\, (X_{0}^{(u)}, \widetilde{X}_{0}^{(u)})\,$ is a product measure, then $\,X_{\cdot}^{(u)}\,$ and $\, \widetilde{X}_{\cdot}^{(u)}\,$ in (\ref{eq: NDu})-(\ref{eq: Fu}) are independent, as in Remark \ref{rem: 1}, and hence, the joint law of $\,( \overline{X}_{\cdot, 1}, \ldots , \overline{X}_{\cdot, k}) \,$ in (\ref{eq: 3.4})-(\ref{eq: 3.5}) is the product measure. Thus in this case of $\,u \, =\,  0\,$, Proposition \ref{prop: 1.5} corresponds to the classic propagation of chaos result (see {Kac} (1958) for the original result for Boltzmann equation in Kinetic Theory, {McKean} (1967), {Tanaka} (1978), {Sznitman} (1984, 1991), {Graham} (1992), {M\'el\'eard} (1995), {Graham \& M\'el\'eard} (1997) for the advancement of  theory for McKean-Vlasov and Boltzmann equations, {Bolley, Guillin \& Malrieu} (2010), {Kolokoltsov} (2010), {Mischler \& Mouhot} (2013) and {Mischler, Mouhot \& Wennberg}  (2015) for recent developments of quantitative approach in propagation of chaos and references within them), where the limiting joint law  takes the product form. This means the dependence between each  particle $\,X_{\cdot, i}^{(u)}\,$ and another particle $\, X_{\cdot, j}^{(u)}\,$, $\, j \neq i \,$ diminishes in the limit, as $\, n \to \infty\,$. {Chong} \& {Kl\"uppelberg} (2017) investigate linear partial mean field systems based on fairly general network structures in which both, propagation of chaos and local dependency arises jointly. 
\hfill $\,\square\,$
\end{remark} 
\begin{remark}[Breaking invariance under permutations]  \label{rem: nPoC} When $\, u \in (0, 1] \,$, Proposition \ref{prop: 1.5} implies that the local directed chain dependence among consecutive particles is preserved even in the limit as the number of particles go to infinity, in general. Thus if $\, u \in (0, 1] \,$, the limiting system (\ref{eq: 3.4})-(\ref{eq: 3.5}) of (\ref{eq: PAu})-(\ref{eq: PAub}) does not propagate the stochastic chaos, in contrast to the case $\,u \, =\,  0\,$. 

This phenomenon can be seen as  a consequence of breaking the invariance under permutations in the finite particle system (\ref{eq: PAu})-(\ref{eq: PAub}), that is, the consecutive particles are invariant only under the shifts in one direction as in (\ref{eq: symmetry})-(\ref{eq: symmetry2}), and the finite particle system is not invariant under permutations, for example,  
\[
\text{Law} (X_{\cdot, 1}^{(u)}, X_{\cdot, 2}^{(u)}, \ldots , X_{\cdot, n}^{(u)}) \neq \text{Law} (X_{\cdot, n}^{(u)}, X_{\cdot, n-1}^{(u)}, \ldots , X_{\cdot, 1}^{(u)}) \,  
\]
unless $\, \widetilde{b}(\cdot, \cdot, \cdot) \equiv 0\,$. To our knowledge, our approach of breaking the invariance under permutations provides the first such instance of describing the dependence of the limiting system in the context of a particle system approximation to the solution of a nonlinear stochastic McKean-Vlasov equation. The simple case of a directed chain with its recursive structure sets itself apart from other network structures by allowing for a description by representative particles which solve a nonlinear McKean-Vlasov equation with distributional constraint. The analysis of this kind of ``matryoshka'' McKean-Vlasov equations might be of independent interest. 
\hfill $\,\square\,$
\end{remark}

\begin{prop}\label{prop: 2.0} In addition to the same assumptions for the functional $\,b\,$ as in Proposition \ref{prop: 1.2}, we assume that the marginal distribution $\, \mathrm m_{t}( {\mathrm d} y ) \, =\, \mathrm m^{\ast}_{t} ( {\mathrm d} y) \,$ of $\, ( X^{(u)}_{t},  t \ge 0) $ has the density $\, m_{t} (\cdot)\,$ $($i.e., $\, \mathrm m_{t}( {\mathrm d} y ) \, =\, m_{t}(y) {\mathrm d} y \,$, $\, y \in \mathbb R\,$ $)$  with $\, \int_{\mathbb R} \lvert y \rvert^{2} \mathrm m_{0}( {\mathrm d} y )  < \infty \,$ and assume there exists a constant $\,C_{T}\,$ such that 
\begin{equation}  \label{ineq: bm} 
\Big \lvert \widetilde{b}(t, x_{1}, y_{1}) \cdot \frac{\,m_{t}(x_{1})\,}{\,m_{t} (y_{1}) \,} - \widetilde{b} (t, x_{2}, y_{2}) \cdot \frac{\, m_{t}(x_{2}) \,}{\,m_{t}(y_{2})\,} \Big \rvert \le C_{T} ( \lvert x_{1} - x_{2}\rvert + \lvert y_{1} - y_{2}\rvert )  \, 
\end{equation} 
for every $\, (x_{i}, y_{i}) \in \mathbb R^{2} \,$, $\,i \, =\, 1, 2\,$, $\,0 \le t \le T\,$ and 
\begin{equation} \label{ineq: gc} 
\Big \lvert \widetilde{b}(t, x, y) \cdot \frac{\,m_{t}(x)\,}{\,m_{t} (y) \,}  \Big \rvert \le C_{T}( 1+ \lvert x \rvert + \lvert y \rvert ) \, 
\end{equation}
for every $\, (x, y) \in \mathbb R^{2}\,$, $\, 0 \le t \le T \,$. 
Then for the difference between $(\ref{eq: PAu})-(\ref{eq: PAub})$ and $(\ref{eq: 3.4})$-$(\ref{eq: 3.5})$ we have the estimate 
\begin{equation} \label{eq: prop2.0}
\sup_{n \ge 1}\frac{\,1\,}{\,\sqrt{ n}\,} \sum_{i=1}^{n} \mathbb E [ \sup_{0 \le s \le t} \lvert X_{s,i}^{(u)} - \overline{X}_{s,i} \rvert] < \infty \, . 
\end{equation}
\end{prop}

\begin{proof}
Substituting 
\begin{equation*}
\begin{split}
\widetilde{b}(s, X_{s,i}^{(u)}, X_{s,j}^{(u)}) \, & =\,   (\widetilde{b}(s, X_{s,i}^{(u)}, X_{s,j}^{(u)}) - \widetilde{b}(s, \overline{X}_{s,i}, X_{s,j}^{(u)})) 
\\
& {} \quad + (\widetilde{b}(s, \overline{X}_{s,i}, X_{s,j}^{(u)}) - \widetilde{b}(s, \overline{X}_{s,i}, \overline{X}_{s,j})) + 
\widetilde{b}(s,  \overline{X}_{s,i},  \overline{X}_{s,j}) 
\end{split}
\end{equation*}
into the differences 
\begin{equation}
\begin{split}
X_{t,i}^{(u)} - \overline{X}_{t,i} \, &=\,  u \int^{t}_{0} ( \widetilde{b} (s, {X}_{s,i}^{(u)}, {X}_{s,i+1}^{(u)}) - \widetilde{b} (s, \overline{X}_{s,i}, \overline{X}_{s,i+1})) {\mathrm d} s 
\\
& {} \quad + (1-u) \int^{t}_{0} \Big( \frac{\,1\,}{\,n\,}\sum_{j=1}^{n} \widetilde{b}(s, X_{s,i}^{(u)} , X_{s,j}^{(u)}) - \int_{\mathbb R} \widetilde{b}(s, \overline{X}_{s,i}, y) {\mathrm m}^{\ast}( {\mathrm d} y) \Big) {\mathrm d} s 
\end{split}
\end{equation}
for $\, i \, =\, 1, \ldots , n-1\,$, and the difference  
\begin{equation}
\begin{split}
X_{t,n}^{(u)} - \overline{X}_{t,n} \, &=\,  u \int^{t}_{0} ( \widetilde{b} (s, {X}_{s,n}^{(u)}, {X}_{s,1}^{(u)}) - \widetilde{b} (s, \overline{X}_{s,n}, \overline{X}_{s,n+1})) {\mathrm d} s 
\\
& {} \quad + (1-u) \int^{t}_{0} \Big( \frac{\,1\,}{\,n\,}\sum_{j=1}^{n} \widetilde{b}(s, X_{s,n}^{(u)} , X_{s,j}^{(u)}) - \int_{\mathbb R} \widetilde{b}(s, \overline{X}_{s,n}, y) {\mathrm m}^{\ast}( {\mathrm d} y) \Big) {\mathrm d} s 
\end{split} 
\end{equation}
at the boundary for $\, 0 \le t \le T\,$, applying the triangle inequality and (\ref{eq: Lip}), and then taking the supremum, we obtain   
\begin{equation} \label{eq: 3.17a} 
\begin{split} 
& \sum_{i=1}^{n} \sup_{0 \le t \le T} \lvert X_{t, i}^{(u)} - \overline{X}_{t,i}\rvert \\
  \le \, & 2 C_{T} \int^{T}_{0}  \sum_{i=1}^{n} \sup_{0 \le t \le s} \lvert X_{t, i}^{(u)} - \overline{X}_{t,i}\rvert {\mathrm d} s + 2 C_{T} u \int^{T}_{0} ( \lvert X_{s,1}^{(u)} - \overline{X}_{s, n+1} \rvert - \lvert X_{s,1}^{(u)} - \overline{X}_{s, 1}\rvert) {\mathrm d} s \\
& {} \quad + (1 - u) \int^{T}_{0} \frac{1}{\, n \, } \sum_{i=1}^{n} \Big \lvert  \sum_{j=1}^{n} \overline{b} ( s, \overline{X}_{s, i}, \overline{X}_{s, j}) \Big \rvert {\mathrm d} s 
\end{split}
\end{equation}
\[
 \le \,  2 C_{T} \int^{T}_{0} \sum_{i=1}^{n} \sup_{0 \le t \le s} \lvert X_{t, i}^{(u)} - \overline{X}_{t,i}\rvert {\mathrm d} s + 2 C_{T} u \int^{T}_{0}  \lvert  \overline{X}_{s, n+1}  - \overline{X}_{s, 1} \rvert {\mathrm d} s 
\]
\[
{} \quad + (1 - u) \int^{T}_{0} \frac{1}{\, n \, } \sum_{i=1}^{n} \Big \lvert  \sum_{j=1}^{n} \overline{b} ( s, \overline{X}_{s, i}, \overline{X}_{s, j}) \Big \rvert {\mathrm d} s\, , 
\]
where we set $\, \overline{b}(s, x, z) \, :=\,  \widetilde{b}(s, x, z) - \int_{\mathbb R} \widetilde{b}(s, x, y) {\mathrm m}^{\ast} ( {\mathrm d} y) \, $ for $\, x, z \in \mathbb R\,$, $\, 0 \le s \le T\,$. Here we used $\, \lvert x \rvert - \lvert y \rvert \le \lvert x - y \rvert\,$, $\, x, y \in \mathbb R\,$ in the last inequality, and this way we take care of the boundary particle. Note that $\, X_{\cdot, 1}^{(u)} \equiv X_{\cdot, n+1}^{(u)}\,$ but $\, \overline{X}_{\cdot,1} \not \equiv \overline{X}_{\cdot, n+1}\,$. 

After using {Gronwall}'s lemma, taking  expectation, we obtain 
\[
\sum_{i=1}^{n} \mathbb E [ \sup_{0 \le t \le T} \lvert X_{t, i}^{(u)} - \overline{X}_{t,i}\rvert] \le 2C_{T} e^{2C_{T} T} \mathbb E \Big[ \int^{T}_{0} \lvert  \overline{X}_{s, n+1}  - \overline{X}_{s, 1} \rvert {\mathrm d} s + \int^{T}_{0} \frac{1}{\, n \, } \sum_{i=1}^{n} \Big \lvert  \sum_{j=1}^{n} \overline{b} ( s, \overline{X}_{s, i}, \overline{X}_{s, j}) \Big \rvert {\mathrm d} s\Big]  \, , 
\]
where there exists some constant $\, c > 0\,$ such that we evaluate the first term 
\begin{equation}
\mathbb E \Big [ \int^{T}_{0}  \lvert  \overline{X}_{s, n+1}  - \overline{X}_{s, 1} \rvert {\mathrm d} s \Big]  \le \mathbb E \Big [ \int^{T}_{0} (  \sup_{0 \le u \le T} \lvert  \overline{X}_{u, n+1} \rvert  + \sup_{ 0 \le u \le T} \lvert \overline{X}_{u, 1} \rvert ) {\mathrm d} s \Big] \le 2 T( \mathbb E [  \lvert \overline{X}_{0,1}\rvert]  + c) e^{c T} \, , 
\end{equation}
by (\ref{eq: 1.2b}) in Proposition \ref{prop: 1.2} and then with (\ref{ineq: bm})-(\ref{ineq: gc}) we evaluate the second term 
\begin{equation} \label{eq: 3.17}
\sum_{i=1}^{n}\mathbb E \Big [ \int^{T}_{0}\frac{\,1\,}{\,n\,}  \Big \lvert \sum_{j=1}^{n} \overline{b} ( s, \overline{X}_{s, i}, \overline{X}_{s, j}) \Big \rvert {\mathrm d} s  \Big] \le \sum_{i=1} ^{n} \int^{T}_{0} \!\! \Big( \mathbb E \Big [\frac{\,1\,}{\,n^{2}\,}  \Big \lvert \sum_{j=1}^{n} \overline{b} ( s, \overline{X}_{s, i}, \overline{X}_{s, j}) \Big \rvert^{2} \Big] \Big)^{1/2}   {\mathrm d} s  \le {\,c \sqrt{n}\,} 
\end{equation}
by the {Cauchy-Schwarz} inequality and the (Markov) chain structure of the particle system $\, \overline{X}_{\cdot, i}\,$, $i \, =\, 1, \ldots , n$, that is, by the map $\, {\bm \Phi}\,$ in (\ref{eq: Phi0}), $\, \overline{X}_{\cdot, i} \, =\, {\bm \Phi} ( \cdot , ({\mathrm m}^{\ast}_{s})_{0 \le s \le \cdot}, ( \overline{X}_{s, i+1})_{0 \le s \le \cdot} , (W_{s,i})_{0 \le s \le \cdot}) \,$ for $\, i \, =\, n-1, \ldots , 1\,$. Note that when $\,u \in (0, 1] \,$, $\, \overline{X}_{\cdot, i}\,$ and $\, \overline{X}_{\cdot, j}\,$ are dependent for $\, i \neq j\,$, while $\, \overline{X}_{\cdot, i+1}\,$ and $\, W_{\cdot, i}\,$ are independent for $\, i \, =\, n-1, \ldots , 1\,$. An intuitive  interpretation of the last inequality in (\ref{eq: 3.17}) is that the dependence between $\, \overline{X}_{s , i}\,$ and $\, \overline{X}_{s, j}\,$ decays sufficiently fast, as $\, \lvert i - j \rvert \to \infty\,$. Its precise statement and some technical details are given in Appendix \ref{sec: proof1}. 

Finally, combining these inequalities, we conclude the proof of (\ref{eq: prop2.0}) by 
\[
\sup_{n \ge 1}\frac{\,1\,}{\,\sqrt{n}\,} \sum_{i=1}^{n} \mathbb E [ \sup_{0 \le t \le T} \lvert X_{t, i}^{(u)} - \overline{X}_{t,i}\rvert] \le2C_{T} e^{2C_{T} T}  \sup_{n \ge 1} \big( \frac{\, 2 T\, }{\sqrt{n}}( \mathbb E [  \lvert \overline{X}_{0,1}\rvert]  + c) e^{c T} + c \big) < \infty \, . 
\]
\end{proof}

\begin{remark} The fluctuation results (central limit theorem and large deviations) suggested from Proposition \ref{prop: 2.0} are ongoing research topics. We conjecture that Propositions \ref{prop: 1.4}-\ref{prop: 1.5} still hold if we replace (\ref{eq: PAub}) by another process,  e.g., a standard Brownian motion, as long as the effect of the boundary process on the first two (or $\,k\,$) components in (\ref{eq: PAu}) diminishes sufficiently fast in the limit. The additional conditions (\ref{ineq: bm})-(\ref{ineq: gc}) are used to evaluate the decay of asymptotic covariance between $\, \overline{X}_{\cdot, i}\,$ and $\,\overline{X}_{\cdot, i + n}\,$ as $\, n \to \infty\,$ (see Appendix \ref{sec: proof1}). In particular, the dependence between the first particle $\, \overline{X}_{\cdot, 1}\,$ and the last particle $\, \overline{X}_{\cdot, n}\,$ of the directed chain diminishes in the limit. It is an ongoing project to see whether one may relax these conditions (\ref{ineq: bm})-(\ref{ineq: gc}). 
\hfill $\,\square\,$
\end{remark}


\section{Detecting mean-field in the presence of directed chain interaction} \label{sec: CaseStudy}

In the weak solution $\,(\Omega, \mathcal F, (\mathcal F_{t}), \mathbb P) \,$, $\, (X_{\cdot}, \widetilde{X}_{\cdot}) \, :=\, (X_{\cdot}^{(u)}, \widetilde{X}_{\cdot}^{(u)}) \,$, $\,B_{\cdot}\,$ from Proposition \ref{prop: 1}, the parametric value $\,u\,$ in (\ref{eq: NDu}) indicates how much the particle $\,X_{\cdot}\,$ depends on the neighborhood particle $\, \widetilde{X}_{\cdot}\,$ in the directed chain, and $\, (1-u)\,$ indicates how much it depends on its law $\, \text{Law} ( X_{t}) \,$ for every $\,t \ge 0\,$. Let us consider the following detection problem of a single observer. 

\medskip 

\noindent {\bf Detection Problem.} Suppose that an observer only observes the single path $\,X_{t}\,$, $\,t \ge 0\,$ but does neither  know the values $\,u \in [0, 1]\,$ nor $\, \widetilde{X}_{t} \,$, $\,t \ge 0\,$ in (\ref{eq: NDu})-(\ref{eq: INITIAL}) under the same assumptions as Proposition \ref{prop: 1}. Only given the filtration $\, \mathcal F^{X}_{t} \, :=\, {\bm \sigma} ( X_{s}, 0 \le s \le t ) \vee \mathcal N\,$, $\, t\ge 0\,$, augmented by the null sets $\,\mathcal N\,$, can the observer detect the value $\,u \in [0, 1]\,$ (and hence, the effect $\, 1 - u \,$ of mean-field)? 

\medskip 

In order to discuss this problem, it is natural to extend our consideration to the solution $\, ( \overline{X}_{t, 1},  \ldots , \overline{ X}_{t, n+1}) \,$, $\, t \ge 0 \,$ of the system of the directed chain stochastic differential equations  
\begin{equation} \label{eq: 4.1}
\begin{split}	
{\mathrm d} \overline{X}_{t, i} \, =\, b(t, \overline{X}_{t, i}, \, F_{t, i} ) \, {\mathrm d} t + {\mathrm d} B_{t, i} \,  ; \quad i \, =\,  1, \ldots , n\, , \, \, t \ge 0  \, , 
\end{split}
\end{equation} 
as the solution to the system of the directed chain stochastic equations in (\ref{eq: 3.4})-(\ref{eq: 3.5}) in section \ref{sec: InfParticle}, for arbitrary $\,n \in \mathbb N\,$, where $\, F_{s, i} \,$ is the random measure similar to (\ref{eq: Fu}), i.e., 
\begin{equation} \label{eq: Fu4} 
{F}_{s, i} \, :=\, u \cdot \delta_{ \overline{X}_{s, i+1} } + (1-u) \cdot \mathcal L_{ \overline{X}_{s, i}} \, , \quad i \, =\,  1, \ldots , n \, , \, \, s \ge 0 \, 
\end{equation} 
with the distributional constraints, such that the initial values $\, ( \overline{X}_{0, 1},  \ldots  , \overline{X}_{0, n+1})\,$ are independently, identically distributed with finite second moments; The marginal law is identical  
\begin{equation} \label{eq: NDlaw4}
\text{Law} ( \{ \overline{X}_{t, i}, t \ge 0 \}) \, =\, \text{Law} (\{ \overline{X}_{t, 1}, t \ge 0 \})  \, ; \quad i \, =\,   1, 2, \ldots , n+1 \, , 
\end{equation}
and the following independence relationships hold for independent standard Brownian motions $\, (B_{t, 1}, \ldots , B_{t, n})\,$, $\, t \ge 0\,$
\begin{equation} \label{eq: NDlaw5}
{\bm \sigma} (\{ ( \overline{X}_{t, n+1}, \ldots , \overline{X}_{t, i+1}), t \ge 0 \}, \overline{X}_{0, i}) \perp\!\!\!\perp {\bm \sigma} ( \{B_{t, i}, t \ge 0 \}) \, ; \quad i \, =\,   n, \ldots , 1 \, . 
\end{equation}
The weak solution $\, ( \overline{X}_{\cdot, 1}, \ldots , \overline{ X}_{\cdot, n+1}) \,$ can be constructed as we considered in section \ref{sec: InfParticle}.  Namely, we solve for $\, ( \overline{X}_{\cdot, n}, \overline{X}_{\cdot, n+1})\,$ in (\ref{eq: 4.1}) first as in Proposition \ref{prop: 1} and then solve recursively for the directed chain system $\, ( \overline{X}_{\cdot, k}, \overline{X}_{\cdot, k+1}, \ldots , \overline{X}_{\cdot, n+1}) \,$ for $\, k \, =\,  n, \ldots , 1\,$. We redefine 
\begin{equation} \label{eq: 4.1pair}
\, (X_{\cdot}, \widetilde{X}_{\cdot})  \, =\, ( X_{\cdot}^{(u)}, \widetilde{X}_{\cdot}^{(u)})\, :=\, (\overline{X}_{\cdot, 1}, \overline{X}_{\cdot , 2}) \,
\end{equation}
from the first two elements of $\,( \overline{X}_{\cdot, 1}, \overline{X}_{\cdot,2 },  \ldots , \overline{ X}_{\cdot, n+1}) \,$, and the observer only observes $\,   X_{\cdot} \, =\, \overline{X}_{\cdot, 1}\,$. 

\begin{remark}[Variations of the detection problem] The setup of the detection problem would be different, if the observer observes the whole $\,(n+1)\,$ particles $\, (\overline{X}_{\cdot, 1}, \ldots , \overline{X}_{\cdot, n+1} )\,$ in (\ref{eq: 4.1}) or if the observer observes the pre-limit system (\ref{eq: NDMF}) of $\,n\,$ particles. It would be interesting, yet out of scope of this current paper, to study the detection methods and to quantify the information gain/loss for different setups, and moreover to detect the presence of the propagation of chaos. These interesting open problems were  suggested by the  editors and the reviewers, while this paper was revised. \hfill $\,\square\,$
\end{remark}

Let us define the stochastic exponential 
\begin{equation} \label{eq: SEZ}
Z_{t} \, :=\,  \exp \Big ( - \int^{t}_{0} b(s, X_{s}, F_{s}) {\mathrm d} B_{s} - \frac{1}{\, 2\, }\int^{t}_{0} \lvert b(s, X_{s}, F_{s})\rvert^{2} {\mathrm d} s \Big) \, ; \quad t \ge 0 \, , 
\end{equation}
where $\, F_{s} \, :=\, F_{s, 1} \, =\, F_{s}^{(u)}\,$ from (\ref{eq: Fu}) and $\, B_{\cdot}\, =\,  B_{\cdot, 1}\,$, which satisfies $\, {\mathrm d} Z_{t} \, =\,  - Z_{t} b(t, X_{t}, F_{t}) {\mathrm d} B_{t}$, $\,t \ge 0\,$. 
Here $\,\{Z_{\cdot}; \mathcal F_{\cdot}\}\,$ is a nonnegative, local martingale and hence, is a supermartingale with $\, \mathbb E [ Z_{t} ] \le 1\,$, $\, t\ge 0 \,$. 

  If the Novikov condition (e.g., Corollary 3.5.13 of Karatzas \& Shreve (1991)) for $\, Z_{\cdot}\,$ holds, i.e., 
\begin{equation} \label{eq: Novikov} 
\mathbb E \Big[ \exp \Big( \frac{1}{\, 2 \, } \int^{t}_{0} \lvert b(s, X_{s}, F_{s})\rvert^{2} {\mathrm d} s \Big) \Big] < \infty \, ; \quad t \ge 0 \, , 
\end{equation}
then $\,Z_{\cdot}\,$ is a martingale. Since it is not always easy to verify the Novikov condition directly except for the Gaussian case (e.g., see section \ref{sec: iOU} below) or for the bounded functional case (i.e., the functional $\, \widetilde{b}\,$ in (\ref{eq: Lin}) is bounded), we shall discuss the martingale property of $\,Z_{\cdot}\,$. 

Let us assume the finite moment condition $\, \mathbb E [ \lvert X_{0}\rvert^{2}] < \infty\,$ for the initial distribution $\, \theta\,$ in (\ref{eq: INITIAL}). Then under the linear growth condition (\ref{eq: 1.2a}) and this finite second moment condition, 
as in Proposition \ref{prop: 1.2}, we have $\, 
\mathbb E [ \sup_{0 \le t \le T} \lvert X_{t}\rvert^{2}] < + \infty \, $ (see Remark \ref{rem: 2.4}), and hence, combining with the inequalities $\, \lvert b(s, X_{s}, F_{s}) \rvert \le C_{T} ( 1 + \lvert X_{s}\rvert + u \lvert \widetilde{ X}_{s}\rvert + ( 1- u ) \mathbb E [ \lvert X_{s}\rvert]) \,$, $\, 0 \le s \le T \,$ and $\, (a_{1}+a_{2}+a_{3}+a_{4})^{2} \le 4 (a_{1}^{2} + a_{2}^{2} + a_{3}^{2}+a_{4}^{2})\,$ for nonnegative reals $\,a_{i} \ge 0\,$, we obtain 
\begin{equation} \label{eq: Eb2}
\mathbb E \Big[ \int^{T}_{0} \lvert b(s, X_{s}, F_{s}) \rvert^{2} {\mathrm d} s \Big] \le 4 C_{T}^{2} T ( 1 + 3 \mathbb E [ \sup_{0 \le s \le T} \lvert X_{s}\rvert^{2} ] ) < \infty \, . 
\end{equation}

Following the proof of Lemma 3.9 (and see also Exercise 3.11) of Bain \& Crisan (2009), in order to show $\, \mathbb E [ Z_{t} ] \, =\, 1	\,$, $\, t \ge 0 \,$, we consider for $\,\varepsilon > 0\,$, 
\begin{equation}  \label{eq: 4.2}
\frac{Z_{t}}{\, 1 + \varepsilon Z_{t}\, } \, =\,  \frac{1}{\, 1  + \varepsilon \, } + \int^{t}_{0} \frac{Z_{s} b(s, X_{s}, F_{s}) }{(1+\varepsilon Z_{s})^{2}}  {\mathrm d} B_{s} - \int^{t}_{0} \frac{ \varepsilon Z_{s}^{2} \lvert b(s, X_{s}, F_{s}) \rvert ^{2}}{\, ( 1 + \varepsilon Z_{s})^{3}\, }  {\mathrm d} s  \, ; \quad t \ge 0 \, , 
\end{equation}
and its expectation for $\, 0 \le t \le T\,$
\begin{equation} \label{eq: EZ1}
\begin{split}
1 \ge \mathbb E [ Z_{t}] \ge \mathbb E \Big[ \frac{Z_{t}}{\, 1 + \varepsilon Z_{t}\, } \Big] \, & =\,  \frac{1}{\, 1 + \varepsilon\, } - \mathbb E \Big[ \int^{t}_{0} \frac{ \varepsilon Z_{s}^{2}\lvert b(s, X_{s}, F_{s}) \rvert  ^{2}  }{\, ( 1 + \varepsilon Z_{s})^{3}\, } {\mathrm d} s \Big] \, \\
\end{split}
\end{equation}
where we used (\ref{eq: Eb2}) to show that the stochastic integral in (\ref{eq: 4.2}) is indeed a martingale and hence its expectation is zero. Thus, in order to verify $\,\mathbb E [ Z_{t} ] \, =\,  1\,$, $\, t\ge 0 \,$, by letting $\,\varepsilon \downarrow 0\,$ in (\ref{eq: EZ1}) and by the dominated convergence theorem, it suffices to check 
\begin{equation} \label{eq: EZb2}
\mathbb E \Big[ \int^{T}_{0}Z_{s} \lvert b(s, X_{s}, F_{s}) \rvert^{2}  {\mathrm d} s \Big]  < \infty \, ; \quad T > 0 \, . 
\end{equation}

Note that since $\,F_{s} \, =\, F_{s}^{(u)}\,$ in (\ref{eq: Fu}) depends on $\, \widetilde{X}_{\cdot}\,$, under the linear growth condition (\ref{eq: 1.2a}) on the functional $\,b\,$, the condition (\ref{eq: EZb2}) is reduced to estimates for both 
\begin{equation} \label{eq: EZX} 
\, \mathbb E \Big[ \int^{T}_{0} Z_{s} \lvert X_{s}\rvert ^{2} {\mathrm d}s  \Big] \, < \infty \quad  \text{ and } \quad \, \mathbb E \Big[ \int^{T}_{0} Z_{s} \lvert \widetilde{X}_{s}\rvert ^{2} {\mathrm d} s \Big] \, < \infty \, ; \quad T > 0 \, ,  
\end{equation} 
where the joint distribution of $\, (X_{\cdot}, Z_{\cdot}) \,$ is not the same as that of $\, (\widetilde{X}_{\cdot}, Z_{\cdot}) \,$.

\begin{prop} \label{eq: Girsanov}
In addition to the assumptions in Proposition \ref{prop: 1}, let us assume $\,\mathbb E [ \lvert X_{0} \rvert^{2}] < + \infty \,$. Then the first inequality in $(\ref{eq: EZX})$ holds. Moreover, for $\, i \, =\,  2, \ldots , n \,$ and for every $\, T > 0\,$, 
\[
\mathbb E \Big[ \int^{T}_{0} Z_{s} \lvert X_{s, i+1}\rvert^{2} {\mathrm d} s \Big] < \infty 
\quad \text{ implies }  \quad 
\mathbb E \Big[ \int^{T}_{0} Z_{s} \lvert X_{s, i}\rvert^{2} {\mathrm d} s \Big ] < \infty \, , 
\]
where $\, (X_{\cdot, 1}, \ldots , X_{\cdot, n+1}) \,$ is defined from $(\ref{eq: 4.1})$-$(\ref{eq: NDlaw5})$ with $(\ref{eq: 4.1pair})$. In particular, if for every $\, T > 0\,$, 
$\, \mathbb E [ \int^{T}_{0} Z_{s} \lvert X_{s, n+1}\rvert^{2} {\mathrm d} s ] < \infty\,$, 
then the second inequality in $(\ref{eq: EZX})$ holds. 
\end{prop}

\begin{proof} Under the assumptions in Proposition \ref{prop: 1} with $\, \mathbb E [ \lvert X_{0}\rvert^{2}] < \infty\,$, we consider for $\,\varepsilon > 0\,$ 
\begin{equation} \label{eq: Girsanov.1}
{\mathrm d} \Big( \frac{ Z_{t} \lvert X_{t}\rvert^{2} }{\,  1+ \varepsilon Z_{t}\lvert X_{t}\rvert^{2}\, } \Big) \, =\,  \frac{ {\mathrm d} (Z_{t}\lvert X_{t}\rvert^{2})}{\, (1+ \varepsilon Z_{t}\lvert X_{t}\rvert^{2})^{2}\, } - \frac{\varepsilon \, {\mathrm d} \langle Z_{\cdot}\lvert X_{\cdot}\rvert^{2}\rangle_{t} 	}{\, (1+ \varepsilon Z_{t}\lvert X_{t}\rvert^{2})^{3}\, } \, , 
\end{equation}
where $\, 
{\mathrm d} ( Z_{t} \lvert X_{t}\rvert^{2} ) \, =\, Z_{t} {\mathrm d} t + (2 Z_{t} X_{t}  - Z_{t}\lvert X_{t}\rvert^{2}) {\mathrm d} B_{t} \, $, $\, t \ge 0 \,$. Taking the expectations, we claim 
\[
\mathbb E \Big[\frac{ Z_{t} \lvert X_{t}\rvert^{2} }{ \, 1+ \varepsilon Z_{t}\lvert X_{t}\rvert^{2}\, }  \Big] \,  \le  \,  \mathbb E \Big[ \frac{ \lvert X_{0}\rvert^{2}  }{1+ \varepsilon \lvert X_{0}\rvert^{2}}\Big] + \mathbb E \Big[ \int^{t}_{0}  \frac{ Z_{s}}{\, (1+ \varepsilon Z_{s}\lvert X_{s}\rvert^{2})^{2}\, } {\mathrm d} s \Big] \, \le \, \mathbb E [ \lvert X_{0}\rvert^{2} ] + \int^{t}_{0} \mathbb E [ Z_{s}] {\mathrm d} s \, ; \quad t \ge 0 \, . 
\] 
Here the stochastic integrals with respect to the Brownian motion in (\ref{eq: Girsanov.1}) are indeed martingales, as in Exercise 3.11 of Bain \& Crisan (2009). Note also that $\,b\,$ disappears in the evaluation. Since $\, \mathbb E [ Z_{\cdot}] \le 1\,$, by letting $\,\varepsilon \downarrow 0\,$, we obtain the first inequality in (\ref{eq: EZX}) from $\, \mathbb E [ Z_{t} \lvert X_{t}\rvert^{2} ] \le \mathbb E [ \lvert X_{0}\rvert^{2}] + t\,$. 

For the second assertions, we replace $\, Z_{t} \lvert X_{t}\rvert^{2}/ ( 1 + \varepsilon \lvert X_{t}\rvert^{2} )\,$ in (\ref{eq: Girsanov.1})  by $\, Z_{t} \lvert X_{t, i}\rvert^{2}/ ( 1 + \varepsilon \lvert X_{t, i}\rvert^{2} )\,$, $\, t \ge 0\,$ for $\,i \, =\,  2, 3, \ldots , n\,$. Thanks to (\ref{eq: 1.2a}) and $\, \mathbb E [ Z_{\cdot}] \le 1\,$, we have 
\[
\frac{{\mathrm d}}{ {\mathrm d} t } \mathbb E \Big[\frac{ Z_{t} \lvert X_{t, i}\rvert^{2} }{ \, 1+ \varepsilon Z_{t}\lvert X_{t, i}\rvert^{2}\, }  \Big] \,  \le 1 + 4 C_{T} \Big(  1 + \mathbb E [ \sup_{0 \le s \le t} \lvert X_{s, i}\rvert^{2} ]  + 
 \mathbb E \Big[\frac{ Z_{t} \lvert X_{t, i}\rvert^{2} }{ \, 1+ \varepsilon Z_{t}\lvert X_{t, i}\rvert^{2}\, }  \Big] + \mathbb E \Big[\frac{ Z_{t} \lvert X_{t, i+1}\rvert^{2} }{ \, 1+ \varepsilon Z_{t}\lvert X_{t, i+1}\rvert^{2}\, }  \Big] \Big) \, . 
\]
As in remark \ref{rem: 2.4}, we may derive the estimate $\, \mathbb E [ \sup_{0 \le s \le T} \lvert X_{s}\rvert^{2}] < \infty\,$. Then applying the Gronwall inequality, we obtain the estimate that there exists a constant $\, c > 0 \,$  such that 
\[
\mathbb E \Big[\frac{ Z_{t} \lvert X_{t, i}\rvert^{2} }{ \, 1+ \varepsilon Z_{t}\lvert X_{t, i}\rvert^{2}\, }  \Big] \le c + 4 C_{T} \mathbb E \Big[\frac{ Z_{t} \lvert X_{t, i+1}\rvert^{2} }{ \, 1+ \varepsilon Z_{t}\lvert X_{t, i+1}\rvert^{2}\, }  \Big] + \int^{t}_{0} \Big( c +  4 C_{T} \mathbb E \Big[\frac{ Z_{s} \lvert X_{s, i+1}\rvert^{2} }{ \, 1+ \varepsilon Z_{s}\lvert X_{s, i+1}\rvert^{2}\, }  \Big] \Big) e^{4 C_{T} T} {\mathrm d} s \, . 
\]
Integrating over $\, [0, T]\,$ with respect to $\,t \,$ and letting $\,\varepsilon \downarrow 0\,$, we claim the conclusions. 
\end{proof}

Let us assume that (\ref{eq: Novikov}) or 
$(\ref{eq: EZX})$ holds. 
Then the stochastic exponential $\,(Z_{t}, \mathcal F_{t})\,$, $\,t \ge 0\,$ in $(\ref{eq: SEZ})$ is martingale. By Girsanov theorem, under a new probability measure $\, {\mathbb P}_{0}\,$ with expectation $\, {\mathbb E}_{0}\,$, defined 
by 
\begin{equation} \label{eq: GirsanovRN}
\, 
({ {\mathrm d} {\mathbb P}_{0}}/{ {\mathrm d} \mathbb P }) \vert_{\mathcal F_{T}} \, :=\,  Z_{T} \, 
\end{equation} 
for every $\,T > 0\,$,  
we have the Kallianpur-Striebel formula:  $\,{\mathbb P}_{0} \,$ $(\, \mathbb P \,)$-a.s.
\begin{equation} \label{eq: KallStr}
\pi_{t} (\varphi ) \, :=\, \mathbb E [ \varphi ( \widetilde{X}_{t} ) \vert \mathcal F^{X}_{t}] \, =\, \frac{\rho_{t}(\varphi)}{\rho_{t}( {\bf 1}) }  \, , \quad \text{ where } \quad 
 \rho_{t}(\varphi) \, :=\, {\mathbb E}_{0} [ Z_{t}^{-1} \varphi ( \widetilde{X}_{t}) \vert \mathcal F^{X}_{T} ] \, ; \quad 0 \le t \le T \,  
\end{equation} 
for measurable function $\, \varphi:\mathbb R \to \mathbb R\,$ with $\, \mathbb E  [ \lvert \varphi ( \widetilde{X}_{t}) \rvert ] < \infty \,$. Its proof is a direct consequence of Proposition 3.16 and Exercise 5.1 of Bain \& Crisan (2009). 

Given the observation $\, \mathcal F_{T}^{X}\,$, the conditional log-likelihood function $\, \mathbb E [ \log ({ {\mathrm d} \mathbb P }/{ {\mathrm d}  {\mathbb P}_{0}  }) \vert_{\mathcal F_{T}} \vert \mathcal F_{T}^{X} ] \, $ 
is 
\[
\mathbb E \big[ {} - \log Z_{T} \vert \mathcal F_{T}^{X} \big] \, =\,  \mathbb E \Big[ \int^{T}_{0} b(s, X_{s}, F_{s}) {\mathrm d} X_{s} - \frac{1}{2} \int^{T}_{0} \lvert b(s, X_{s}, F_{s}) \rvert^{2} {\mathrm d} s \Big \vert \mathcal F_{T}^{X}\Big]  \, . 
\] 
Substituting the expression $\, b(s, X_{s}, F_{s}) \, =\, u \widetilde{b}(s, X_{s}, \widetilde{X}_{s}) + (1-u) \int_{\mathbb R} \widetilde{b}(s, X_{s}, y) {\rm m}_{s}({\mathrm d} y ) \, $, $\, s \ge 0 \, $, we see it is a quadratic function of $\,u\,$. Thus the conditional log-likelihood is maximized at  the conditional maximum likelihood estimator $\, \widehat{u}_{T}\,$ defined by 
\begin{equation} \label{eq: MLE0} 
\begin{split}
& \widehat{u}_{T} \, :=\, \Big[ \mathbb E \Big[  \int^{T}_{0} \lvert \overline{b}(s, X_{s}, \widetilde{X}_{s}) \rvert^{2}  {\mathrm d} s \Big \vert \mathcal F_{T}^{X} \Big]  \Big]^{-1} \\
& ~~~~~~~~~~~~~~~~~~~~~ \times 
 \mathbb E \Big[  \int^{T}_{0} \overline{b} (s, X_{s}, \widetilde{X}_{s})  {\mathrm d}_{s} \Big( X_{s} - \int^{s}_{0} \int_{\mathbb R}\widetilde{b}(u, X_{u}, y) {\rm m}_{u} ({\mathrm d} y) {\mathrm d} u \Big) \Big \vert   \mathcal F_{T}^{X} \Big]  \, ,  
\end{split}
\end{equation}
where $\,\overline{b}(s, x, z) \, :=\, \widetilde{b}(s, x, z) - \int_{\mathbb R} \widetilde{b}(s, x, y) {\rm m} ( {\mathrm d} y ) \,$ is defined as in the proof of Proposition \ref{prop: 2.0} for $\,x, z \in \mathbb R\,$, $\, 0 \le s \le T\,$. The maximum likelihood estimator $\, \widehat{u}_{T}\,$ in (\ref{eq: MLE0}) is well defined if the denominator is not zero, e.g., $\, \overline{b}(\cdot, \cdot, \cdot) \neq 0\,$. 

The analysis of (\ref{eq: MLE0}) is not straightforward due to the conditional expectation and the filtering feature. 
We shall discuss the filtering equations in the following section \ref{sec: Filter} and then see  the consistent estimators under the special linear case in section \ref{sec: iOU}. 
For the theory of parameter estimation in Stochastic Filtering, see e.g., chapter 17 of Liptser \& Shiryayev (2001). 

\subsection{Filtering equations} \label{sec: Filter} 
In the following let us assume under $\,\mathbb P_{0}\,$ defined in (\ref{eq: GirsanovRN})
\begin{equation} \label{eq: cZakai} 
{\mathbb P }_{0} \Big( \int^{t}_{0} \big \lvert {\mathbb E}_{0} \big[  \lvert b ( s, X_{s}, F_{s})\rvert \, \big \vert \, \mathcal F_{T}^{X} \big] \big \rvert ^{2} {\mathrm d} s  < \infty \Big) \, =\,  1 \, ; \quad 0 \le t \le T \, . 
\end{equation}

\begin{prop} \label{prop: Zakai} Let us recall $(\ref{eq: 4.1pair})$ and assume $(\ref{eq: EZX})$ and $(\ref{eq: cZakai})$. For every $\,\varphi \in C^{2}_{0}(\mathbb R) \,$, the conditional expectations $\, \rho_{\cdot}(\varphi) \, =\, \mathbb E_{0} [ Z_{\cdot}^{-1} \varphi ( \widetilde{X}_{\cdot}) \vert \mathcal F_{T}^{X} ] \,$ in $(\ref{eq: KallStr})$ satisfy 
\begin{equation} \label{eq: Zakai} 
\rho_{t} (\varphi) \, =\,  \pi_{0} (\varphi)  + \int^{t}_{0} \rho_{s,2} ( \varphi b) {\mathrm d} X_{s} + \int^{t}_{0} \rho_{s,3}( \widetilde{\mathcal A}_{s} \varphi) {\mathrm d} s\, , \quad 0 \le t \le T \, , 
\end{equation}
where $\,\pi_{0}(\varphi) \, =\,  \mathbb E [ \varphi (\widetilde{X}_{0}) \vert \mathcal F_{0}^{X}] \, =\,  \mathbb E [ \varphi ( \widetilde{X}_{0}) ] \, =\,  \mathbb E [ \varphi (X_{0}) ] \,$ in $(\ref{eq: KallStr})$, $\, \rho_{s,2}(\varphi b) \,$ and $\, \rho_{s,3}( \widetilde{\mathcal A}_{s} \varphi ) \,$ are defined by 
\begin{equation} \label{eq: Zakai1}
\rho_{s, 2}(\varphi b) \, :=\,  {\mathbb E}_{0} \big [ Z_{s}^{-1}\varphi ( \widetilde{X}_{s}) b(s, X_{s}, F_{s}) \vert \mathcal F_{T}^{X}\big]  \, , 
\end{equation}
\begin{equation} \label{eq: Zakai2} 
\rho_{s,3}(\widetilde{\mathcal A}_{s} \varphi) \, :=\,  {\mathbb E}_{0} \Big[ Z_{s}^{-1} \Big( \varphi^{\prime} ( \widetilde{X}_{s}) b ( s, \widetilde{X}_{s} , {F}_{s, 2}) + \frac{1}{\, 2\, } \varphi^{\prime\prime} ( \widetilde{X}_{s}) \Big) \Big \vert  \mathcal F_{T}^{X} \Big] \, , \quad 0 \le s \le T \, . 
\end{equation}
Here $\,F_{\cdot} \, =\,  F_{\cdot, 1}\,$ and $\, {F}_{\cdot, 2}\,$ are the random measures defined as in $(\ref{eq: Fu4})$ from the law of $\,X_{\cdot} \, =\, \overline{X}_{\cdot, 1}\,$, $\, \widetilde{X}_{\cdot} \, =\, \overline{X}_{\cdot, 2}\,$ and $\, \overline{X}_{\cdot, 3} \,$,  in the solution $\, ( \overline{X}_{\cdot, 1}, \overline{X}_{\cdot, 2},  \overline{ X}_{\cdot, 3}) \,$ to the system $(\ref{eq: 4.1})$-$(\ref{eq: Fu4})$  of the directed chain stochastic differential equation with the distributional constraints $(\ref{eq: NDlaw4})$-$(\ref{eq: NDlaw5})$. 
\end{prop}

\begin{proof} The proof idea is a slight modification of Theorem 3.24 of Bain \& Crisan (2009). For $\,\varphi \in C^{2}_{0}(\mathbb R) \,$ let us take the semimartingale decomposition $\,\varphi (\widetilde{X}_{\cdot}) \, =\,  \varphi ( \widetilde{X}_{0}) + M_{\cdot}^{\varphi} + A_{\cdot}^{\varphi} \,$ of $\, \varphi (\widetilde{X}_{\cdot})\,$, where $\, M_{\cdot}^{\varphi}\,$ and $\, A_{\cdot}^{\varphi}\,$ are the martingale and the finite variation terms, respectively, 
\[
{\mathrm d} M^{\varphi}_{t} \, :=\, \varphi^{\prime} ( \widetilde{X}_{t}) {\mathrm d} B_{t, 2} \, , \quad 
{\mathrm d} A_{t}^{\varphi} \, :=\, \varphi^{\prime} ( \widetilde{X}_{t}) b(t, \widetilde{X}_{t}, F_{t, 2}) {\mathrm d} t + \frac{1}{\, 2\, } \varphi^{\prime\prime} ( \widetilde{X}_{t}) {\mathrm d}t \, , 
\]
and then consider $\, \widetilde{Z}^{\varepsilon}_{t} \cdot \varphi ( \widetilde{X}_{t})  \,$, $\, t \ge 0\,$ for $\, \varepsilon > 0\,$, and its conditional expectation with respect to $\, \mathcal F_{T}^{X}\,$, where $\, \widetilde{Z}^{\varepsilon}_{t} \, :=\, Z_{t}^{-1} / ( 1 + \varepsilon Z_{t}^{-1}) \,$, $\, t \ge 0\,$. 
Since $\, {\mathrm d} Z_{t}^{-1} \, =\,  Z_{t}^{-1} b(t, X_{t}, F_{t}) {\mathrm d} X_{t}\,$, $\, t \ge 0\,$, we have 
\[
{\mathrm d} \widetilde{Z}_{t}^{\varepsilon} \, =\,  \frac{Z_{t}^{-1} b(t, X_{t}, F_{t}) }{\, (1 + \varepsilon Z_{t}^{-1})^{2}\, } {\mathrm d} X_{t} - \frac{\varepsilon Z_{t}^{-2} \lvert b(t, X_{t}, F_{t}) \rvert^{2}\, }{ ( 1 + \varepsilon Z_{t}^{-1})^{3}\, } {\mathrm d} t \, ; \quad t \ge 0 \, . 
\] 

Substituting these expressions for $\, M^{\varphi}_{\cdot}\,$, $\, A^{\varphi}_{\cdot}\,$ and $\, \widetilde{Z}^{\varepsilon}_{\cdot}\,$ into 
\begin{equation} \label{eq: cZakai3} 
\mathbb E_{0} [ \widetilde{Z}^{\varepsilon}_{t}\varphi ( \widetilde{X}_{t}) \vert \mathcal F_{T}^{X} ]  \, =\,  \mathbb E_{0} [ \widetilde{Z}^{\varepsilon}_{0}\varphi ( \widetilde{X}_{0}) \vert \mathcal F_{T}^{X} ] + \mathbb E_{0} \Big[ \int^{t}_{0} \widetilde{Z}_{s}^{\varepsilon} {\mathrm d} A_{s}^{\varphi} 
+ \int^{t}_{0} \widetilde{Z}_{s}^{\varepsilon} {\mathrm d} M_{s}^{\varphi}  + \int^{t}_{0} \varphi ( \widetilde{X}_{s}) {\mathrm d} \widetilde{Z}^{\varepsilon}_{s} \Big \vert \mathcal F_{T}^{X} \Big] \, ,  
\end{equation}
and taking the limits as $\,\varepsilon \downarrow 0\,$ under  $(\ref{eq: EZX})$ and $(\ref{eq: cZakai})$, we obtain (\ref{eq: Zakai}) with (\ref{eq: Zakai1})-(\ref{eq: Zakai2}). Indeed, we need (\ref{eq: EZX}) to show (\ref{eq: EZb2}) and then  with $\, \lVert \varphi \rVert_{\infty} \, :=\, \sup_{x\in \mathbb R} \lvert \varphi (x) \rvert \,$ for $\, \varphi \in C^{2}_{0} ( \mathbb R) \,$, 
\[
\mathbb E_{0} \Big[ \int^{t}_{0} \Big \lvert \widetilde{Z}_{s}^{\varepsilon} \cdot \frac{\varphi ( \widetilde{X}_{s}) b(s, {X}_{s}, F_{s}) }{ ( 1 + \varepsilon Z_{s}^{-1}) } \Big \rvert^{2} {\mathrm d}s \Big] \le \lVert \varphi \rVert_{\infty}^{2} \mathbb E \Big[ \int^{t}_{0} Z_{s} \lvert b(s, X_{s}, F_{s})  \rvert^{2} {\mathrm d} s \Big] < \infty \, , 
\]
\[
\mathbb E_{0} \Big[ \int^{t}_{0} \mathbb E_{0} \Big[ \Big \lvert \widetilde{Z}_{s}^{\varepsilon} \cdot \frac{\varphi ( \widetilde{X}_{s}) b(s, {X}_{s}, F_{s}) }{ ( 1 + \varepsilon Z_{s}^{-1}) } \Big \rvert^{2}   \Big \vert \mathcal F_{T}^{X}\Big] {\mathrm d} s \Big] \le  \lVert \varphi \rVert_{\infty}^{2} \mathbb E \Big[ \int^{t}_{0} Z_{s} \lvert b(s, X_{s}, F_{s})  \rvert^{2} {\mathrm d} s \Big] < \infty \, , 
\]
and hence 
\begin{equation} \label{eq: cZakai2}
\mathbb E_{0} \Big[ \int^{t}_{0} \widetilde{Z}_{s}^{\varepsilon} \cdot \frac{\varphi ( \widetilde{X}_{s}) b(s, {X}_{s}, F_{s}) }{ ( 1 + \varepsilon Z_{s}^{-1}) } {\mathrm d} X_{s} \Big \vert \mathcal F_{T}^{X}\Big] \, =\, \int^{t}_{0}  \mathbb E_{0} \Big[ \widetilde{Z}_{s}^{\varepsilon} \cdot \frac{\varphi ( \widetilde{X}_{s}) b(s, {X}_{s}, F_{s}) }{ ( 1 + \varepsilon Z_{s}^{-1}) }  \Big \vert \mathcal F_{T}^{X}\Big] {\mathrm d} X_{s} \, . 
\end{equation}
is a martingale under $\, \mathbb P_{0}\,$ for $\, 0 \le t \le T\,$. We need (\ref{eq: cZakai}) to verify that $\, \mathbb P_{0} \,$ a.s. 
\[
\int^{t}_{0} \Big \lvert \mathbb E_{0} \Big[ \widetilde{Z}_{s}^{\varepsilon} \cdot \frac{\varphi ( \widetilde{X}_{s}) b(s, {X}_{s}, F_{s}) }{ ( 1 + \varepsilon Z_{s}^{-1}) } \Big \vert \mathcal F_{T}^{X} \Big]  - \rho_{s, 2} ( \varphi b) \Big \rvert^{2} {\mathrm d} s \le 4 \lVert \varphi \rVert^{2}_{\infty} \int^{t}_{0} \lvert \mathbb E_{0} [ \lvert b(s, X_{s}, F_{s})\rvert \big \vert \mathcal F_{T}^{X}]\rvert^{2} {\mathrm d} s < \infty  \,  
\]
for $\, 0 \le t \le T\,$ and then by the dominated convergence theorem to show that as $\,\varepsilon \downarrow 0\,$, for a suitably chosen subsequence $\,\varepsilon_{n} \downarrow 0\,$, (\ref{eq: cZakai2}) converges $\,\mathbb P_{0}\,$-a.s. to the $\, \mathbb P_{0}\,$-local martingale $\, \int^{\cdot}_{0} \rho_{s,2}(\varphi b) {\mathrm d} X_{s}\,$. The convergence of the other terms in (\ref{eq: cZakai3}) along $\,\varepsilon_{n}\,$ is relatively straightforward. 

\end{proof}

More generally, for every $\,n \ge 2\,$ and every $\,k \, =\, 1, 2, \ldots , n\,$, given $\,\varphi (t, x) \in C^{1,2}([0, T] \times \mathbb R^{k}) \,$ with bounded support in $\, \mathbb R^{k}\,$ and bounded in time $\, 0 \le t \le T\,$ let us recall (\ref{eq: 4.1pair}), i.e., $\, \overline{X}_{\cdot, 1} \equiv X_{\cdot}\,$, and define
\begin{equation} \label{eq: rhotk}
\rho_{t, k} (\varphi) \, :=\,  \mathbb E_{0} [ Z_{t}^{-1} \varphi (t,  \overline{X}_{t, 1}, \ldots , \overline{X}_{t,k}) \vert \mathcal F_{T}^{X}] \, ; \quad k \, =\,   2, \ldots , n \, , \, \, 0 \le t \le T \, ,  
\end{equation}
and similarly, let us define the normalized version 
\begin{equation} \label{eq: pitk}
\pi_{t,k} (\varphi) \, :=\,  \mathbb E [ \varphi (t, \overline{X}_{t, 1}, \ldots , \overline{X}_{t,k}) \, \vert \, \mathcal F_{t}^{X} ] \, ; \quad k \, =\, 2, \ldots , n , \, \, 0 \le t \le T\, 
\end{equation}
for $\, \varphi \in C^{2}_{0} ([0, T] \times \mathbb R^{k}) \,$, and for $\, (s, x) \in [0, T] \times \mathbb R^{k} \,\,$, $\, i \, =\,  1, 2, \ldots , n\,$
\[
D_{s} \varphi (s, x) \, :=\, \frac{\,\partial \varphi \,}{\, \partial s\,} (s, x) \, , \quad D_{i} \varphi (s, x)\, :=\,  \frac{\partial \varphi }{\partial x_{i}} (s, x) \, , \quad D_{i}^{2} \varphi (s, x) \, :=\, \frac{\,\partial^{2} \varphi \,}{\,\partial x_{i}^{2}\,}(s, x) \, . 
\]
Then with a similar reasoning as in the proof of Proposition \ref{prop: Zakai}, we obtain the following system (\ref{eq: Zakai3}) of Zakai equations for the (unnormalized) conditional expectations $\, \rho_{\cdot,k}\,$ of function of $\, ( \overline{X}_{\cdot, 1}, \ldots , \overline{X}_{\cdot, k}) \,$, $\, k \, =\,  2, \ldots , n\,$ with respect to $\, \mathcal F_{T}^{X} \,$ and for arbitrary $\,n \ge 2\,$. 

\begin{prop} \label{prop: Zakai2} Under the same assumption as in Proposition \ref{prop: Zakai}, $\, \rho_{\cdot,k} (\varphi) \,$ in $(\ref{eq: rhotk})$ satisfies 
\begin{equation} \label{eq: Zakai3} 
\rho_{t, k}(\varphi) \, =\,  \pi_{0,k}(\varphi) +  \int^{t}_{0} \rho_{s, k} (\varphi b) {\mathrm d} X_{s} + \int^{t}_{0} \rho_{s, k+1}( \widetilde{A}_{s} \varphi) {\mathrm d} s \, , 
\end{equation} 
where the integrands are defined by 
\begin{equation*}
\begin{split}
\rho_{s,k+1}( \widetilde{\mathcal A}_{s} \varphi) \, :=\,  {\mathbb E}_{0} \Big[ Z_{s}^{-1} & \Big( \sum_{i=1}^{k} D_{i} \varphi ( s, \overline{X}_{s, 1}, \ldots , \overline{X}_{s, k} ) \cdot b ( s, \overline{X}_{s, i} , {F}_{s, i}) + \frac{1}{\, 2\, } \sum_{i=1}^{k} D_{i}^{2} \varphi ( s, \overline{X}_{s, 1}, \ldots , \overline{X}_{s, k} ) \\
& {} +  D_{s} \varphi ( s, \overline{X}_{s, 1}, \ldots , \overline{X}_{s, k} ) + D_{1}   \varphi ( s, \overline{X}_{s, 1} , \ldots , \overline{X}_{s, k})   \cdot  b ( s, \overline{X}_{s, 1} , {F}_{s, 1})    \Big) \Big \vert  \mathcal F_{T}^{X} \Big] \, , \end{split}
\end{equation*}
and 
\begin{equation} \label{eq: rhoskphib}
\rho_{s, k}(\varphi b) \, :=\,  {\mathbb E}_{0} \big [ Z_{s}^{-1}\varphi (s, \overline{X}_{s, 1}, \ldots , \overline{X}_{s, k}) b(s, X_{s, 1}, F_{s, 1}) \, \vert \, \mathcal F_{T}^{X}\, \big]  \, 
\end{equation} 
for $\, 0 \le s \le  t \le T\,$, $\, k \, =\,  2, \ldots , n\,$ and for arbitrary $\,n \ge 2\,$. 
\end{prop} 

Now under the assumption (\ref{eq: EZX}), we have the Kallianpur-Striebel formula: $\,\mathbb P_{0}\,$ ($\, \mathbb P \,$)-a.s. 
\[
\pi_{t, k} (\varphi) \, =\,  \frac{\, \rho_{t, k}(\varphi) \,}{\, \rho_{t}( {\bf 1} ) \,} \, ; \quad k \, =\, 2, \ldots , n \, , \, \, 0 \le t \le T \, . 
\]
Then it follows from Proposition \ref{prop: Zakai} that 
\begin{equation} \label{eq: rhot1}
\rho_{t}({\bf 1}) \, =\,  1 + \int^{t}_{0} \rho_{s, 2} (b) {\mathrm d} X_{s} \, =\,  1 + \int^{t}_{0} \rho_{s} ({\bf 1}) \pi_{s, 2} (b) {\mathrm d} X_{s} \, ; \quad 0 \le t \le T \, . 
\end{equation}

For fixed $\,\varepsilon > 0\,$, applying Ito's formula to $\, (1/2) \log ( \varepsilon + \lvert \rho_{t} ( {\bf 1}) \rvert^{2} ) \,$ with (\ref{eq: rhot1}), we obtain 
\[
{\mathrm d} \Big( \frac{\,1\,}{\,2\,} \log ( \varepsilon + \lvert \rho_{t} ( {\bf 1}) \rvert^{2} ) \Big) \, =\,  \frac{\,\lvert \rho_{t} ( {\bf 1}) \rvert^{2} \pi_{t, 2} (b) \,}{\, \varepsilon + \lvert \rho_{t} ( {\bf 1}) \rvert^{2}\,} {\mathrm d} X_{t} + \frac{\,(\varepsilon - \lvert \rho_{t} ( {\bf 1}) \rvert^{2}) \lvert \rho_{t} ( {\bf 1}) \rvert^{2} \lvert \pi_{t} ( {\bf 1}) \rvert^{2}\,}{\,2(\varepsilon + \lvert \rho_{t} ( {\bf 1}) \rvert^{2}\,} {\mathrm d} t 
\]
for $\, 0 \le t \le T\,$. Under the assumption of Proposition \ref{prop: Zakai2}, letting $\, \varepsilon \downarrow 0 \, $, by the dominated convergence theorem, we have $\, 
\log \rho_{t}( {\bf 1}) \, =\,  \int^{t}_{0} \pi_{s, 2}( b) {\mathrm d} X_{s} - \frac{\,1\,}{\,2\,} \int^{t}_{0} \lvert \pi_{s, 2}(b) \rvert^{2} {\mathrm d} s \,$, 
and hence   
\[
\mathbb E_{0} [ Z_{t}^{-1} \vert \mathcal F^{X}_{T} ] \, =\, \rho_{t}( {\bf 1}) \, =\,  \exp \Big( \int^{t}_{0} \pi_{s, 2} (b) {\mathrm d} X_{s} - \frac{\,1\,}{\,2\,} \int^{t}_{0} \lvert \pi_{s, 2}(b) \rvert^{2} {\mathrm d} s \Big) \, ; \quad  0 \le t \le T\, . 
\]

\begin{prop}\label{prop: KushnerStrat}  In addition to the assumption in Proposition \ref{prop: Zakai}, let us assume 
\begin{equation} \label{eq: KScondition}
\mathbb P \Big( \int^{t}_{0} \lvert \pi_{s,k} (b) \rvert ^{2} {\mathrm d} s < \infty \Big) \, =\,  1 \, ; \quad k \, =\, 2, \ldots , n \, , \, \, 0 \le t \le T \, . 
\end{equation}
Then the conditional expectation $\, \pi_{\cdot,k} (\varphi) \,$ in $(\ref{eq: pitk})$ with respect to $\, \mathcal F_{\cdot}^{X}\,$ satisfies the following system of Kushner-Stratonovich equations  
\begin{equation} \label{eq: KushnerStratonovich}
\pi_{t,k}(\varphi) \, =\, \pi_{0, k}(\varphi) + \int^{t}_{0} (\pi_{s, k} (\varphi b) - \pi_{s,k}(b) \pi_{s,k}(\varphi) ) \big( {\mathrm d} X_{s}  - \pi_{s,k}(b) {\mathrm d} s \big) + \int^{t}_{0} \pi_{s,k+1}( \widetilde{\mathcal A}_{s} \varphi) {\mathrm d} s  
\end{equation}
where we define $\, 
\pi_{s,k}(\varphi b) \, :=\, \rho_{s,k}(\varphi b) / \rho_{s} ( {\bf 1}) \, $, $\,  \pi_{s,k+1}( \widetilde{ \mathcal A}_{s} \varphi) \, :=\, \rho_{s,k+1} ( \widetilde{\mathcal A}_{s} \varphi ) \, / \, \rho_{s} ( {\bf 1} ) \, $ from $(\ref{eq: rhoskphib})$ for $\, 0 \le t \le T\,$, $\, k \, =\,  2, \ldots , n\,$ and for arbitrary $\,n \ge 2\,$. 
\end{prop}

\begin{proof} The proof is now straightforward by It\^o's formula under the condition (\ref{eq: KScondition}) thanks to the representations of $\, \rho_{\cdot} ({\bf 1}) \,$ and $\,\rho_{\cdot, k} (\varphi) \,$ in (\ref{eq: Zakai3}). 
\end{proof}

\begin{remark}
As in Remarks \ref{rem: 1.2} and \ref{rem: 3.1}, when $\, u \in (0, 1]\,$, as a description of the conditional expectations $\,\pi_{\cdot, k}\,$, $\, \rho_{\cdot, k}\,$, $\, k \, =\,  2, \ldots , n\,$ with respect to $\, \mathcal F_{T}^{X}\,$, the system (\ref{eq: Zakai3}) or the system (\ref{eq: KushnerStratonovich}) has an infinite-dimensional aspect. This is because $\,\rho_{\cdot, k}\,$ in (\ref{eq: Zakai3}) is represented by the integral of $\, \rho_{\cdot, k}\,$ with respect to $\, {\mathrm d} X_{s}\,$ and the integral of $\, \rho_{\cdot, k+1}\,$ with respect to $\,  {\mathrm d} s\,$ and because $\,\rho_{\cdot, k+1}\,$ is the conditional expectation of function of $\, (\overline{X}_{\cdot, 1}, \ldots , \overline{X}_{\cdot, k+1}) \,$ if $\, u \in (0, 1]\,$ for every $\, k \, =\,  2, \ldots , n\,$. The system (\ref{eq: KushnerStratonovich}) for $\,\pi_{\cdot,k} \,$ has the same aspect, inherited from $\,\rho_{\cdot, k}\,$ in (\ref{eq: Zakai3}). Since $\, n \,$ is arbitrary, the chain of such descriptions continues. Unique characterizations of conditional expectations as a solution to this infinite system of Zakai equations (\ref{eq: Zakai3}) and  Kushner-Stratonovich equations (\ref{eq: KushnerStratonovich}) is out of scope of the current paper. Numerical methods, particle methods and their comparisons for solving such systems are also interesting ongoing projects. \hfill $\,\square\,$
\end{remark}

\subsection{Connection to the infinite-dimensional Ornstein-Uhlenbeck process} \label{sec: iOU}
Let us take a time-homogeneous linear functional $\, b(t, x, \mu) \, :=\, - \int_{\mathbb R} (x- y) \mu ({\mathrm d} y) \,$ for $\, t \ge 0\,$, $\, x \in \mathbb R\,$, $\, \mu \in \mathcal M (\mathbb R)\,$ of mean-reverting type. Then, (\ref{eq: NDu}) is reduced to the stochastic differential equation 
\begin{equation} \label{eq: E1}
{\mathrm d} X_{t}^{(u)} \, =\,  - \big(u\,  (X_{t}^{(u)} - \widetilde{X}_{t}^{(u)})  + (1-u) ( X_{t}^{(u)} - \mathbb E [ X_{t}^{(u)}] ) \big )\,  {\mathrm d} t + {\mathrm d} B_{t} \, ; \quad t \ge 0 \, 
\end{equation}
for each $\, u \in [0, 1]\,$, 
where we recall that $\widetilde{X}^{(u)}$ has the same law as ${X}^{(u)}$ and is independent of the Brownian motion $B$.

In the case $\,u \, =\, 0\,$, we have $X_{\cdot}^{(0)}=\,X_{\cdot}^{\bullet}$, where $X_{\cdot}^{\bullet}$ is the solution of the pure 
McKean-Vlasov stochastic differential equation
\begin{equation} \label{eq: E2}
{\mathrm d} X_{t}^{\bullet} \, =\,  - (X_{t}^{\bullet} - \mathbb E[ X_{t}^{\bullet}] ) {\mathrm d} t + {\mathrm d} B_{t} \, ,\quad t \ge 0 \, .
\end{equation} 

In the case $\,u \, =\, 1\,$, we have $X^{(1)}_{\cdot}=X_{\cdot}^{\dagger}$, where $X_{\cdot}^{\dagger}$ is given by
\begin{equation} \label{eq: E3}
{\mathrm d} X_{t}^{\dagger} \, =\, - (X_{t}^{\dagger} - \widetilde{X}_{t}^{\dagger}) {\mathrm d} t + {\mathrm d} B_{t} \, , \quad t \ge 0 \, ,
\end{equation}
with $\widetilde{X}^{\dagger}$ having the same law as ${X}^{\dagger}$ and  being independent of $B$.

Coming back to the general case with $u\in [0,1]$ and 
setting a fixed initial value $\, X_{0}^{(u)} \, =\,  0 \,$, we see that the expectations are constant in time 
\begin{equation} \label{eq: E4}
\mathbb E [ X_{t}^{(u)} ] \, =\, \mathbb E [ \widetilde{X}_{t}^{(u)} ] \, =\,  \mathbb E [ X_{t}^{\bullet} ] \, =\,  \mathbb E [ X_{t}^{\dagger}] \, =\, 0  \, , \quad t \ge 0 \, , \, \, u \in [0, 1] \, ,  
\end{equation}
with an explicitly solvable Gaussian pair $\, ( X^{(u)}(t), \widetilde{X}^{(u)}(t))\,$ for $\, t \ge 0 \, $, $\,  u \in [0, 1] \,$ 
\begin{equation} \label{eq: E0} 
\begin{split}
X_{t}^{(u)}\, &=\, \int^{t}_{0} e^{- (t-s)} u \widetilde{X}^{(u)}_{s} {\mathrm d} s + \int^{t}_{0} e^{-(t-s)} {\mathrm d} B_{s} \, , \\
\widetilde{X}_{t}^{(u)} \, &=\, \int^{t}_{0} \sum_{k=0}^{\infty}  \mathfrak p_{0, k} (t-s; u) \, {\mathrm d} W_{s, k} \,, \quad \mathfrak p_{0, k} (t-s; u)\, :=\,  \frac{u^{k}(t-s)^{k}}{k!} e^{-(t-s)} \, , 
\end{split}
\end{equation}
where $\, (W^{k}_{\cdot}, k \ge 0)\,$ is a sequence of independent, one-dimensional standard Brownian motions, independent of the Brownian motion $\, B(\cdot)\,$. 
Note that the integrand $\, \mathfrak p_{0, k} (t-s; u)\,$, $\,k \in \mathbb N_{0}\,$ in (\ref{eq: E0}) is a (taboo) transition probability $\, \mathbb P ( M(t-s) = 	k \vert M(0) \, =\, 0) \,$ of a continuous-time Markov chain $\,M(\cdot) \,$ in the state space $\, \mathbb N_{0}\,$ with generator matrix $\, {\bm Q} \, =\, (q_{i,j})_{i,j \in \mathbb N_{0}} \,$ with $\, q_{i,i+1} \, =\,  u \in [0, 1]\,$, $\, q_{i,i} \, =\, -1\,$ and $\,q_{i,j} \, =\, 0\,$ for the other entries $\,j \, \neq\, i, i+1\,$. When $\, u \, =\, 0 \,$, $\,{\bm Q}\,$ is the generator of a Markov chain with jump rate $\, 1\,$ from state $\,i\,$ and is killed immediately. When $\,u \, =\, 1\,$, $\, {\bm Q} \,$ is the generator of a Poisson process with rate $\,1\,$. When $\,u \, \in (0, 1)\, \,$, it jumps from $\,i\,$ to $\, i+1\,$ with probability $\,u\,$ and killed with probability $\, (1-u) \,$. Thus we interpret $\,{\mathfrak p}_{0, k}(t-s; u) \,$ as $\,(0, k) \,$-element of the $\,\mathbb N_{0} \times \mathbb N_{0}\,$-dimensional matrix exponential $\, e^{(t-s)Q}\,$, i.e., 
\[
\, (\, \mathfrak p_{i, j}(t-s; u) \, :=\,  \mathbb P ( M(t-s) \, =\, j \vert M(0) \, =\, i ) \, , i,j \in \mathbb N_{0}\, ) \, \equiv\, ((e^{(t-s) {\bm Q}})_{i, j}, i,j \in \mathbb N_{0})\,; \quad t \ge s \ge 0 \, .  
\]
For the matrix exponential $\, e^{t {\bm Q}}\,$, $\, t \ge 0\,$ of such $\, {\bm Q}\,$, see for example, {Friedman} (1971). Then we have a {Feynman-Kac} representation formula 
\begin{equation} \label{eq: E-FK}
\widetilde{X}^{(u)}_{t} \, =\,  \mathbb E^{M} \Big[ \int^{t}_{0} \sum_{k=0}^{\infty} {\bf 1}_{\{M(t-s) \, =\, k\}} {\mathrm d} W_{s, k} \vert M(0) \, =\, 0\Big]\, ; \quad t \ge 0 \, , 
\end{equation}
where the expectation is taken with respect to the probability induced by the Markov chain $\, M(\cdot)\,$, independent of the Brownian motions $\, (W_{\cdot, k}, k \in \mathbb N_{0} )\,$. 

\medskip 

Indeed, by Proposition \ref{prop: 1.4}, the solution (\ref{eq: E0}) is obtained by an infinite particle approximation 
\begin{equation} \label{eq: E-1}
{\mathrm d} X_{t,k}^{(u)}\, =\,  - ( X_{t,k}^{(u)} - u X_{t,k+1}^{(u)}) {\mathrm d} t + {\mathrm d} W_{t, k} \, ; \quad t \ge 0 \, , \quad k \in \mathbb N_{0} 
\end{equation}
of the simplified form of (\ref{eq: E1}), that is,  
\[
{\mathrm d} X_{t}^{(u)}\, =\,  - ( X_{t}^{(u)} - u \widetilde{X}^{(u)}_{t}) {\mathrm d} t + {\mathrm d} B_{t} \, ; \quad t \ge 0 \,.   
\]
Here we assume $\, {\bm \sigma} ( X_{t, k+1}^{(u)}, t \ge 0) \,$ and $\, {\bm \sigma} ( W_{t,k}, t \ge 0) \,$ are independent for every $\,k \in \mathbb N_{0}\,$. The infinite particle system (\ref{eq: E-1}) can be represented as an infinite-dimensional Ornstein-Uhlenback  stochastic differential equation or more generally, stochastic evolution equation (see e.g., {Dawson} (1972), {Da Prato \& Zabczyk} (1992), {Kallianpur \& Xiong} (1995), {Batt, Kallianpur, Karandikar, \& Xiong} (1998), {Athreya, Bass \& Perkins} (2005) for more general results in Hilbert spaces)
\begin{equation} \label{eq: E-2}
{\mathrm d} {\bm X}_{t} \, =\, {\bm Q} {\bm X}_{t}\,  {\mathrm d} t + {\mathrm d} {\bm W}_{t} \, ,  
\end{equation}
where $\, {\bm X}_{\cdot}\, :=\, (X_{\cdot, k}^{(u)}, k \in \mathbb N_{0} ) \,$ with $\, {\bm X}_{0} \, =\, {\bm 0}\,$, and $\, {\bm W}_{\cdot} \, :=\, (W_{\cdot, k},k \in \mathbb N_{0} ) \,$. Note that the transition probabilities $\, \mathbb P ( M(t) \, =\, k \vert M(0) \, =\, i ) \, =\, (e^{t {\bm Q}})_{i,k}\,$, $\,i, k \in \mathbb N_{0}\,$ of the continuous-time Markov chain $\, M(\cdot) \,$ defined in the previous paragraph satisfies the backward {Kolmogorov} equation 
\[
\frac{\, {\mathrm d} \,}{\, {\mathrm d} t\,} e^{t {\bm Q}}\, =\,  {\bm Q} \, e^{t {\bm Q}} \, ; \quad t \ge 0 \, . 
\]
Thus, by {It\^o}'s formula we directly verify 
\[
{\mathrm d} \Big( \int^{t}_{0} e^{(t-s) {\bm Q}} {\mathrm d} {\bm W}_{s} \Big) \, =\,  \Big( {\bm Q}  \int^{t}_{0} e^{(t-s) {\bm Q}} {\mathrm d} {\bm W}_{s} \Big)\, {\mathrm d} t + {\mathrm d} {\bm  W}_{t} \, ; \quad t \ge 0 \, , 
\]
and hence 
\[
{\bm X}_{t} \, =\, \int^{t}_{0} e^{(t-s) {\bm Q}} \, {\mathrm d} {\bm W}_{s} \, ; \quad t \ge 0 \, , 
\]
is a solution to (\ref{eq: E-2}). Therefore,  (\ref{eq: E0}) is the solution to (\ref{eq: E1}).  Although $\, {\bm Q} \,$ has the specific form here, it is easy to see that in general, the {Feynman-Kac} formula (\ref{eq: E-FK}) still holds for the infinite-dimensional Ornsten-Uhlenbeck process with a class of generators  $\, {\bm Q} \,$ which form a {Banach} algebra (e.g., the generator of the discrete-state, compound Poisson processes, see {Friedman} (1971)).  

\bigskip 

\subsubsection{Asymptotic Dichotomy}

\begin{table}
\begin{center}
\begin{tabular}{|c|l|c|c|}
\hline 
$\,u\,$ & Interaction Type & Asymptotic Variance & Asymptotic Independence \\ 
&&& (Propagation of chaos)\\
\hline \hline 
$\,u=0 \,$ & Purely mean-field  \eqref{eq: E2}&  & Independent \\ 
\cline{1-2} \cline{4-4} 
$u\in (0,1)$& Mixed interaction  & {Stabilized}& \\
&  &  & Dependent\\
\cline{1-3} 
$\,u=1\,$ & Purely directed chain \eqref{eq: E3}& Explosive &  \\ 
\hline
\end{tabular}
\caption{Different behaviors for different values of $\,u\,$ in the linear, Gaussian case (\ref{eq: E1}).} \label{table: 1} 

Asymptotic variances are given in Remark \ref{remark 4.2}.  Dependence is described in (\ref{eq: variance-covariance2}). 
\end{center}
\end{table}

With $\, X_{0}^{(u)} \, =\, 0\,$ still, it follows from (\ref{eq: E0}) that the auto covariance and cross covariance are 
\begin{equation} \label{eq: variance-covariance2}
\begin{split}
\mathbb E [ X_{s}^{(u)} X_{t}^{(u)} ] \, & =\, \mathbb E [ \widetilde{X}_{s}^{(u)} \widetilde{X}_{t}^{(u)} ]\, =\,  e^{-(t-s)} \int^{s}_{0} e^{-2v} I_{0} ( 2 u \sqrt{(t-s+v)v} ) {\mathrm d} v \, ; \quad 0 \le s \le t \, , 
\\
\mathbb E [ X_{s}^{(u)} \widetilde{X}^{(u)}_{t}] \, & =\,  u \int^{s}_{0} e^{-(s-v)} \mathbb E [ \widetilde{X}^{(u)}_{v} \widetilde{X}^{(u)}_{t} ] {\mathrm d} v \, =\,  u \int^{s}_{0} e^{-(s-v)} \mathbb E [ {X}^{(u)}_{v} {X}^{(u)}_{t} ] {\mathrm d} v \, ; \quad t , s \ge  0 \, 
\end{split}
\end{equation}
for $\, u \in [0,1 ]\,$. Here $\,I_{\nu}(\cdot) \,$ is the modified Bessel function of the first kind with index $\, \nu \,$, defined by 
\[
I_{\nu}(x) \, :=\,  \sum_{k=0}^{\infty} \frac{ ( x/ 2)^{2k+\nu}}{\, \Gamma (k+1) \cdot \Gamma (\nu + k + 1)\, }   \, ; \quad x > 0 \, , \nu \ge - 1 \, . 
\]
Note that the Bessel functions $\, I_{0}(x) \,$ and $\, I_{1}(x)\,$ grow with the order of $\, O ( e^{x} \, / \, \sqrt{ 2 \pi x} ) \,$ as $\, x \to \infty\,$. 

\begin{remark}[Asymptotic dichotomy of (\ref{eq: E1})] \label{remark 4.2} The asymptotic behaviors of their variances as $\, t \to \infty \,$ are dichotomous as in Table \ref{table: 1}: 
\begin{equation} \label{eq: E5}
\, \text{Var}( X_{t}^{(u)}) \, =\, \int^{t}_{0} e^{-2v} I_{0} ( 2uv) {\mathrm d} v  \, =\, \left \{ \begin{array}{cc} 
O(1) \,, & \, u \in [0, 1) \, , \\ 
O ( \sqrt{ t } ) \, , & \, u \, =\, 1 \,,  \end{array} \right . 
\end{equation}
\[
\text{ with } \quad \text{Var}( X_{t}^{(0)}) \, =\, \text{Var}( X_{t}^{\bullet}) \, =\,  \frac{\, 1 - e^{-2t}\, }{2}  \, , \quad 
\quad \text{Var}( X_{t}^{(1)}) =\text{Var}( X_{t}^{\dagger}) \, =\,  t \,e^{-2t} ( I_{0}(2t) + I_{1}(2t) ) \,.     
\]
\begin{enumerate}
\item When $\,u \, \in  [0, 1) \,$, the process $\, X^{(u)}_{\cdot}  \,$ is positive recurrent and its stationary distribution is Gaussian with mean $\,0\,$ and variance 
\begin{equation} \label{eq: varOUu}
\lim_{t\to \infty} \text{Var}( X_{t}^{(u)}) \, =\,  \int^{\infty}_{0} e^{-2 v} I_{0} (2 u v) {\mathrm d} v \, =\, \frac{\,1\,}{\,2 \sqrt{ 1 - u^{2}}  }  \, < \infty\,. 
\end{equation}
In particular, when $\,u \, =\, 0\,$, $\, X^{(0)}_{\cdot} \, =\, X^{\bullet}_{\cdot}\,$ is an  Ornstein-Uhlenbeck process with a stationary Gaussian distribution of mean $\,0\,$ and variance $\, 1\, / \, 2\,$, independent of $\, \widetilde{X}^{(0)}_{\cdot} \, =\,  \widetilde{X}^{\bullet}_{\cdot}\,$. 

\item When $\,u \, =\, 1\,$, the process $\, X^{(1)}_{\cdot}\, =\, X^{\dagger}_{\cdot}\,$ is a mean zero Gaussian process with growing variance of the order $\,O (\sqrt{t})\,$ with $\, \lim_{t\to \infty} \text{Var} ( X_{t}^{\dagger}) \, =\,  \infty\,$, given by (\ref{eq: E5}) and covariances 
\[
\mathbb E [ X_{s}^{\dagger} X_{t}^{\dagger}] \, =\, \mathbb E [ \widetilde{X}_{s}^{\dagger} \widetilde{X}_{t}^{\dagger}]\, =\, e^{-(t-s)} \int^{s}_{0} e^{-2 v} I_{0} ( 2 \sqrt{ (t -s + v) v } ) {\mathrm d} v \, ; \quad 0 \le s \le t \, .  
\]
In particular, 
$\,\mathbb E [ X_{s}^{\dagger} X_{t}^{\dagger}] \, =\,  O (e^{-(t - 2 \sqrt{ (t + s) s})} t^{-1/4} ) \,$ for large $\,t \to \infty\,$. 
\end{enumerate} 

This asymptotic dichotomy is an answer to the first question posed in section \ref{sec: Chain}. Namely, the large system of type (\ref{eq: chain1}) diverges widely, while the large system of type (\ref{eq: MF1}) converges to the stationary distribution as $\, t \to \infty\,$ under the linear case of (\ref{eq: E1}). \hfill $\,\square\,$
\end{remark}

\begin{remark}[Repulsive case]
Instead of mean-reverting, if the drift functional $\,b\,$ is of {\it repulsive} type $\, 
\, b(t, x, \mu) \, :=\, \int_{\mathbb R} (x - y) \mu ( {\mathrm d} y) \,$, then the resulting paired process in (\ref{eq: NDu}) with $\,u \, =\, 1\,$ is described by 
\begin{equation}
{\mathrm d} X_{t}^{\dagger\!\dagger} \, =\, ( X_{t}^{\dagger\!\dagger} - \widetilde{X}_{t}^{\dagger\!\dagger}) {\mathrm d} t + {\mathrm d} B_{t}  \, ; \quad t \ge  0 \, 
\end{equation}
with the conditions (\ref{eq: Fu})-(\ref{eq: NDlaw}). 
The solution with the initial values $\,X_{0}^{\dagger\!\dagger} \, =\, \widetilde{X}_{0}^{\dagger\!\dagger} \, =\, 0 \,$ is given by 
\[
X^{\dagger\!\dagger}_{t} \, =\, \int^{t}_{0} e^{t-s} \widetilde{X}^{\dagger\!\dagger}_{s} {\mathrm d} s + \int^{t}_{0} e^{t-s} {\mathrm d} B_{s}\, , \quad 
\widetilde{X}^{\dagger\!\dagger}_{t} \, =\,  \int^{t}_{0} \sum_{k=0}^{\infty} e^{t-s} \cdot \frac{\,(-1)^{k} (t-s)^{k}\,}{\,k!\,} \, {\mathrm d} W_{s, k} \, ; \quad t \ge 0 \, 
\]
for independent Brownian motions $\, W_{\cdot, k}\,$, $\,k \in \mathbb N_{0}\,$, independent of $\, B_{\cdot}\,$. In this case the variance grows {\it exponentially} fast, i.e., $\, 
\text{Var} ( X_{t}^{\dagger\!\dagger}) \, =\, t e^{2t} (I_{0}(2t) - I_{1}(2t)) \, $, $t \ge 0 \,$. 
\hfill $\,\square\,$
\end{remark}

\begin{remark}[Discrete-time time series] The discrete-time version of (\ref{eq: E1}) with distributional constraints can be defined by the difference equation for $\, (X_{k}, \widetilde{X}_{k}) \, =\, (X_{k}^{(u)}, \widetilde{X}_{k}^{(u)})\,$, $\,k \, =\, 0, 1, \ldots \,$, for example, given constants $\, a \in (0, 1)\,$, $\, u \in [0, 1]\,$, 
\begin{equation} \label{eq: discrete-time} 
X_{k} \, =\, a X_{k-1}  + (1-a) ( u \widetilde{X}_{k-1} + (1-u) \mathbb E [ X_{k-1}] ) + \varepsilon_{k} \, ; \quad k \, =\,  1, 2, \ldots 
\end{equation}
where we assume $\, \text{Law} ( \{X_{k}, k \, =\, 0, 1, 2, \ldots \}) \equiv \text{Law} ( \{\widetilde{X}_{k}, k \, =\, 0, 1, 2, \ldots \} ) \,$ and the independently, identically distributed noise sequence $\, \varepsilon_{k}\,$, $\,k \ge 1\,$ is independent of $\, \widetilde{X}_{\cdot}\,$. We shall solve for the joint distribution of $\, (X_{k}, \widetilde{X}_{k})\,$, $\, k \ge 0\,$. Again for simplicity, let us assume $\, X_{0} \, =\, 0 \, =\,  \widetilde{X}_{0}\,$. Then it reduces to $\, \mathbb E [ X_{k}] \, =\,  \mathbb E [ \widetilde{X}_{k}]\,$, $\, k \, =\, 0, 1, 2, \ldots \,$, and hence to $\, X_{k} \, =\, a X_{k-1} + (1-a)u \widetilde{X}_{k-1} + \varepsilon_{k}\,$; $\,k \, =\,  1, 2, \ldots \,$ with distributional constraints. 

By recursive substitutions, we have $\, X_{1} \, =\,  \varepsilon_{1} \,$, $\,X_{2} \, =\,  a X_{1} + (1-a) u \widetilde{X}_{1} + \varepsilon_{2}\,$, $\,X_{3} \, =\,  a X_{2} + (1-a) u \widetilde{X}_{2} + \varepsilon_{3}\,$, and $\, 
X_{n} \, =\,  u \sum_{k=1}^{n-1} a^{k-1} ( 1-a) \widetilde{X}_{n-k} + \sum_{k=0}^{n-1} a^{k} \varepsilon_{n-k} \, $. Thanks to the distributional constraints, we represent the distribution of the solution to (\ref{eq: discrete-time}) as 
\[
X_{n} \, =\,  \sum_{0 \le \ell \le k \le n-1} {k \choose \ell} u^{\ell } ( 1- a)^{\ell} a^{k-\ell} \varepsilon_{n-k, \ell} \, , \quad \widetilde{X}_{n} \, =\, \sum_{0 \le \ell \le k \le n-1} {k \choose \ell } u^{\ell } ( 1- a)^{\ell} a^{k-\ell} \varepsilon_{n-k, \ell+1}\, , 
\]
where $\, \varepsilon_{n,m}\,$, $\,n, m \in \mathbb N\,$ are independently, identically distributed noise with $\, \varepsilon_{n, 0} \, =\,  \varepsilon_{n}\,$ for $\,n \in \mathbb N\,$. While the stochastic kernel in the stochastic integral in (\ref{eq: E0}) for the solution to the continuous time equation (\ref{eq: E1}) is a Poisson probability, the stochastic kernel for the solution to the discrete time equation (\ref{eq: discrete-time}) is a binomial probability. The variance and covariances can be calculated, e.g., 
\[
\mathbb E [ X_{n}^{2}] \, =\,  \sum_{k=0}^{n-1}\sum_{\ell=0}^{k} {k \choose \ell} ^{2} u^{2\ell} ( 1- a)^{2\ell} a^{2(k-\ell)} \, =\,  \sum_{k=0}^{n-1} u^{k} (1-a)^{k} \, _{2} F_{1} \Big( -k, -k, 1 ; \, \frac{a^{2}}{\,  ( 1-a)^{2}\, }\Big) \, , 
\]
where $\, _{2} F_{1}(\cdot)\,$ is the Gauss hypergeometric function. 
\hfill $\,\square\,$
\end{remark}

\begin{remark} We may generalize these explicit examples in this section to time-inhomogeneous, linear equations, where the dependence on expectations and marginal laws remain to exist in the expressions. The resulting expressions would become more complicated.  Here we demonstrate some simple examples with the connection to the infinite-dimensional Ornstein-Uhlenbeck processes. \hfill $\,\square\,$
\end{remark}

\subsubsection{Consistent estimation}
Let us denote by $\, \mathcal F_{t}, t \ge 0 \,$ the filtration generated by the solution pair $\, (X_{\cdot}, \widetilde{X}_{\cdot}) \, :=\, (X_{\cdot}^{(u)}, \widetilde{X}_{\cdot}^{(u)})\,$ in (\ref{eq: E1}). Thanks to the Girsanov theorem, the log Radon-Nikodym derivative of the solution $\, \mathbb P^{(u)}\,$ with respect to the Wiener measure $\, \mathbb P_{0}\,$ is given by 
\[
\log \frac{\, {\mathrm d} \mathbb P^{(u)}\,}{\, {\mathrm d} \mathbb P_{0}\,} \Big \vert_{\mathcal F_{T}} \, =\,  \int^{T}_{0} ( X_{t}- u \widetilde{X}_{t}) {\mathrm d} X_{t} + \frac{\,1\,}{\,2\,} \int^{T}_{0} ( X_{t}- u \widetilde{X}_{t})^{2} {\mathrm d} t \, . 
\]
Thus given $\,\mathcal F^{X}_{T}\,$, the observer may maximize the conditional log likelihood function  
\[
\mathbb E \Big[ - \log  \Big( \frac{\, {\mathrm d} \mathbb P^{(u)}\,}{\, {\mathrm d} \mathbb P_{0}\,} \Big \vert_{\mathcal F_{T}} \Big) \Big \vert \mathcal F^{X}_{T} \Big] 
\]
with respect to $\,u\,$, and formally obtain a unique maximizer, corresponding to (\ref{eq: MLE0}),  
\begin{equation} \label{eq: MLE}
\widehat{u} \, :=\, \Big( \int^{T}_{0} \mathbb E \big[  \widetilde{X}_{t}^{2} \vert \mathcal F_{T}^{X} \big] {\mathrm d} t \Big)^{-1} \cdot \mathbb E \Big[ \int^{T}_{0} X_{t} \widetilde{X}_{t}{\mathrm d} t + \int^{T}_{0} \widetilde{X}_{t} {\mathrm d} X_{t} \, \Big \vert \, \mathcal F_{T}^{X} \Big] \, 
\end{equation}
as an estimator of $\, u\,$. Evaluation of these conditional expectations in (\ref{eq: MLE}) is a filtering problem. 

\medskip

The detailed study of $\,\widehat{u}\,$ in (\ref{eq: MLE}) still remains an open problem. If we replace $\, \widetilde{X}_{\cdot}\,$ by $\,X_{\cdot}\,$ in (\ref{eq: MLE}), then we obtain a modified estimator 
\begin{equation}
\widehat{u}_{m}\, :=\, \Big( \int^{T}_{0} X_{t}^{2} {\mathrm d} t \Big)^{-1} \cdot \Big(  \int^{T}_{0} X_{t}^{2} {\mathrm d} t + \int^{T}_{0} X_{t} {\mathrm d} X_{t} \Big) \, =\, 1 - \Big( 2 \int^{T}_{0} X_{t}^{2} {\mathrm d} t \Big)^{-1} \big( T - X_{T}^{2} \big) \, .  
\end{equation}
It follows from (\ref{eq: varOUu}) that $\, \lim_{T\to \infty} \widehat{u}_{m} \, =\,  1 - \sqrt{ 1 - u^{2}} \le u \in [0, 1]\,$. Thus this modified estimator $\, \widehat{u}_{m}\,$ underestimates the value $\,u\,$ asymptotically as $\, T \to \infty\,$. 

\bigskip 

Another typical method of estimation of $\,u \,$ is known as the method of moments. We may obtain the method of moments estimator by matching the second moment in the limit, i.e., 
\begin{equation}
\widehat{u}_{M}\, :=\, \Big [ 1 - \Big( \frac{\,2\,}{\,T\,} \int^{T}_{0}X_{t}^{2} {\mathrm d} t \Big)^{-2} \Big]^{1/2} \, . 
\end{equation} 
It follows from (\ref{eq: varOUu}) directly that $\, \lim_{T\to \infty} \widehat{u}_{M} \, =\,  u \in [0, 1]\,$. Thus this method of moments estimator $\, \widehat{u}_{M}\,$ is asymptotically consistent to the value $\,u \in [0, 1] \,$,  as $\, T \to \infty\,$.

\section{Appendix} 
\subsection{Sketch of proof of (\ref{eq: 3.17})} \label{sec: proof1}
We shall sketch the poof of  (\ref{eq: 3.17}) for Proposition \ref{prop: 2.0} when $\,u > 0\,$. If $\,u \, =\, 0\,$, it reduces to the case of propagation of chaos results and it is given in {Sznitman} (1991). First note that by the construction, $\,\overline{X}_{\cdot, i} \,$ in (\ref{eq: 3.4})-(\ref{eq: 3.5}) is determined by the iteration $\, \overline{X}_{\cdot, i} \, =\, {\bm \Phi} (\cdot, (m_{s})_{0 \le s \le \cdot}, (\overline{X}_{s, i+1})_{0 \le s \le \cdot}, (W_{s, i})_{0 \le s \le \cdot} )\,$  as in (\ref{eq: Phi0}), where $\, \overline{X}_{\cdot, i+1}\,$ is independent of $\, W_{\cdot, i}\,$ for $\, i \, =\,  n, n-1, \ldots , 1\,$, that is, with this random iterative map and a slight abuse of notation, we may write and view 
\begin{equation} \label{eq: 5.1} 
\eta_{t, i} \, :=\, \overline{X}_{t, i} \, =\,  {\bm \Phi} \circ {\bm \Phi} \circ \cdots \circ {\bm \Phi}_{t} ( \overline{X}_{\cdot, n+1} ; W_{\cdot, i}, \ldots , W_{\cdot, n}) \, =\, {\bm \Phi}_{t}^{(n+1-i)} ( \eta_{n+1}; W_{\cdot, i}, \ldots , W_{\cdot, n})  \, 
\end{equation}
for $\, 0 \le t \le T\,$ as an element in the space $\,C([0, T], \mathbb R) \, =\,  C([0, T]) \,$ of continuous functions. Thus, $\, \eta_{i}\,$, $\, i \, =\,  n+1, n, n-1, \ldots , 1\,$ possess a discrete-time Markov chain structure. In particular, for $\,j  < k < i \,$, given $\, \eta_{k}\,$, the distribution of $\, \eta_{i}\,$ and $\,\eta_{j}\,$ are conditionally independent. 

Let us write $\, {\bm W} \, :=\, (W_{\cdot, 1}, \ldots , W_{\cdot, n})\,$ for simplicity.   
For every Lipschitz function $\, \varphi (\cdot) \,$ with Lipschitz constant $\,K\,$, there exists a constant $\,c > 0 \,$ such that the difference between the conditional expectation of $\, \varphi ( \overline{X}_{t, 1}) \,$, given $\, \overline{X}_{s, n+1}\,$, $\, 0 \le s \le T\,$ and the unconditional expectation of $\, \varphi ( \overline{X}_{t, 1}) \,$ is bounded by 
\begin{equation*}
\begin{split}
& \mathbb E [ \sup_{0 \le t \le T} \lvert \mathbb E [ \varphi ( \overline{X}_{t,1}) \vert \overline{X}_{s,n+1} , 0 \le s \le T] - \mathbb E [ \varphi ( \overline{X}_{t,1}) ] \rvert ^{2} \big ] \\
\, &=\,  \int_{C([0, T])} \sup_{0 \le t \le T} \Big \lvert \int_{C([0, T])} \Big( \mathbb E^{\bm W} [ \varphi ( {\bm \Phi}_{t}^{(n)}(\eta_{n+1}; {\bm W}_{\cdot}))]  -  \mathbb E^{\bm W} [ \varphi ( {\bm \Phi}_{t}^{(n)}( \widetilde{\eta}_{n+1}; {\bm W}_{\cdot})) ] \Big)  \mathrm m( {\mathrm d} \widetilde{\eta}_{n+1}) \Big \rvert^{2} \mathrm m ( {\mathrm d} \eta_{n+1})  \\
& \le  \int_{C([0, T])^{2}} \sup_{0 \le t \le T} \mathbb E^{\bm W}[ \lvert \varphi ( {\bm \Phi}_{t}^{(n)} (\eta_{n+1}; {\bm W}) ) - \varphi ( {\bm \Phi}_{t}^{(n)} ( \widetilde{\eta}_{n+1}; {\bm W}) )  \rvert^{2} ] \mathrm m ( {\mathrm d} \widetilde{\eta}_{n+1}) \mathrm m ({\mathrm d} \eta_{n+1}) \\
& \le K^{2} \int_{C([0, T])^{2}} \mathbb E^{\bm W} [ \sup_{0 \le t \le T} \lvert {\bm \Phi}_{t}^{(n)} (\eta_{n+1}) -  {\bm \Phi}_{t}^{(n)} ( \widetilde{\eta}_{n+1}) \rvert^{2} ] \mathrm m ( {\mathrm d} \widetilde{\eta}_{n+1}) \mathrm m ( {\mathrm d} \eta_{n+1}) \le  \frac{\,c^{n}\,}{\,n!\,} \, , 
\end{split}
\end{equation*}
where $\,\mathbb E^{\bm W}\,$ is the expectation with respect to $\,{\bm W}\,$ and the last inequality is verified in a similar way as in the proof of Proposition \ref{prop: 1}, thanks to the Lipschitz continuity (\ref{eq: Lip}) of the functional $\, b(\cdot)\,$. Similarly, there exists a constant $\, c > 0 \,$ such that we have the estimate 
\begin{equation} \label{eq: 5.2} 
\mathbb E \Big [ \sup_{0 \le t \le T, \, x \in \mathbb R} \Big \lvert \mathbb E [ \varphi (x,  \overline{X}_{t,j}) \vert \overline{X}_{s,k}, 0 \le s \le T ] - \mathbb E [ \varphi (x,  \overline{X}_{t,j}) ] \Big \rvert ^{2} \Big ] \le \frac{c^{k-j}}{(k -j)! }\, ; \quad k > j \, 
\end{equation}
for a Lipschitz function $\, \varphi (x, y) : \mathbb R^{2} \to \mathbb R\,$ with $\, \lvert \varphi (x_{1}, y_{1}) - \varphi (x_{2}, y_{2}) \rvert \le K (\lvert x_{1} - x_{2}\rvert + \lvert y_{1} - y_{2} \rvert )\,$. 

Second, note that because of the definition of $\, \overline{b}(\cdot, \cdot, \cdot) \,$ appeared in (\ref{eq: 3.17}), for every $\, j \, =\,  1, \ldots , n \,$,  
\[
\mathbb E [ \overline{b} (s, x, \overline{X}_{s, j})] \, =\, \int_{\mathbb R} \widetilde{b}(s, x, z) \mathrm m_{s} ( {\mathrm d} z) - \int_{\mathbb R} \widetilde{b}(s, x, y) \mathrm m_{s} ( {\mathrm d} y) \, =\,  0 \, ; \quad s \ge 0, \, \, x \in \mathbb R \, . 
\]
Combining this observation and the Markov chain structure with (\ref{eq: 5.2}), for $\, j < k < i \,$ we evaluate 
\begin{equation} \label{eq: 5.3}
\begin{split}
& \mathbb E [ \overline{b}(s, \overline{X}_{s,i} , \overline{X}_{s,j} ) \overline{b} (s, \overline{X}_{s,i} , \overline{X}_{s, k} ) ] \\
& \, =\,  \mathbb E [ \overline{b}(s, \overline{X}_{s, i}, \overline{X}_{s,k} ) \mathbb E [ \overline{b}(s, \overline{X}_{s,i}, \overline{X}_{s,j} ) \vert \overline{X}_{\cdot, i}, \overline{X}_{\cdot , k} ] ]  \\
& \, = \,  \mathbb E [ \overline{b}(s, \overline{X}_{s, i}, \overline{X}_{s,k} ) \mathbb E [ \overline{b}(s, x, \overline{X}_{s,j} ) \vert \overline{X}_{\cdot, i}, \overline{X}_{\cdot , k} ] \big \vert_{x = \overline{X}_{s,i}} ] \\
& \, =\, \mathbb E [ \overline{b}(s, \overline{X}_{s, i}, \overline{X}_{s,k} ) \mathbb E [ \overline{b}(s, x, \overline{X}_{s,j} ) \vert \overline{X}_{\cdot , k} ] \big \vert_{x = \overline{X}_{s,i}} ] \\
& \le (\mathbb E [ \lvert \overline{b}(s, \overline{X}_{s, i}, \overline{X}_{s,k} ) \rvert^{2} ] )^{1/2} \cdot ( \mathbb E [ \lvert \mathbb E [ \overline{b} (s, x, \overline{X}_{s,j}) \vert \overline{X}_{\cdot, k} ] \vert_{x = \overline{X}_{s,i}}   \rvert^{2} ])^{1/2} \le C \cdot \Big[ \frac{\,c^{k-j}\,}{\,(k-j)!\,} \Big]^{1/2}\, ,   
\end{split}
\end{equation}
where the constant $\,c\,$ does not depend on $\,(s, i, j, k) \,$  and we used the Lipschitz continuity of $\, b (\cdot)\,$ and a similar technique as in the proof of Proposition \ref{prop: 1.2} to show $\, \sup_{0\le s\le T}(\mathbb E [ \lvert \overline{b}(s, \overline{X}_{s, i}, \overline{X}_{s,k} ) \rvert^{2} ] )^{1/2}  \le C\,$ for some constant $\,C\,$ which does not depend on $\,(i,k)\,$. This is the case $\, 1\le j < k < i \le n\,$. 

For the case $\, i < j < k \,$ or the case $\, j < i < k\,$ we need the estimates  
\begin{equation} \label{eq: 5.4} 
\mathbb E \Big [ \sup_{0 \le t \le T, \, x \in \mathbb R} \Big \lvert \mathbb E [ \varphi (x,  \overline{X}_{t,k}) \vert \overline{X}_{s,j}, 0 \le s \le T ] - \mathbb E [ \varphi (x,  \overline{X}_{t,k}) ] \Big \rvert ^{2} \Big ] \le  \Big[ \frac{c^{k-j}}{(k -j)! }\Big]\, ; \quad k > j \, . 
\end{equation}
This is similar to (\ref{eq: 5.2}) but the condition in the conditional expectation is reverse in discrete-time. We shall construct time-reversal of the discrete-time Markov chain structure (\ref{eq: 5.1}). To do so, as in Proposition \ref{prop: 1}, given the marginal law $\, \mathrm m (\cdot) \, =\,  \mathrm m^{\ast} (\cdot) \,$ with the marginal density function $\, m_{t}: \mathbb R \to \mathbb R_{+}\,$ at time $\, t \ge 0\,$ in the assumptions (\ref{ineq: bm})-(\ref{ineq: gc}) of Proposition \ref{prop: 2.0},  let us consider the following system of the directed chain stochastic equation with mean-field interaction for $\, (Y_{\cdot}, X_{\cdot}, \widetilde{X}_{\cdot})\,$: 
\begin{equation} \label{eq: sNDu} 
\begin{split}
{\mathrm d} X_{t} \, &=\,   \Big[ u \widetilde{b}(t, X_{t}, \widetilde{X}_{t}) + (1-u) \int_{\mathbb R} \widetilde{b}(t, X_{t}, z)  \widehat{{\mathrm m}}_{t} ({\mathrm d} z) \Big] {\mathrm d} t + {\mathrm d} {B}_{t} \, , 
\\
{\mathrm d} Y_{t} \, &=\,   \Big[ u \widetilde{b}(t, Y_{t}, X_{t}) \cdot \frac{\,m_{t}(Y_{t}) \,}{\,m_{t}(X_{t})\,} + (1-u) \int_{\mathbb R} \widetilde{b}(t, Y_{t}, z)  \widehat{{\mathrm m}}_{t} ({\mathrm d} z) \Big] {\mathrm d} t + {\mathrm d} \widehat{B}_{t} \,  
\end{split}
\end{equation}
driven by independent Brownian motions $\, (B_{\cdot}, \widehat{B}_{\cdot})\,$, where we assume the distributional constraints 
\begin{equation} \label{eq: sNDu1}
\begin{split}
& \text{Law} ( \widetilde{X}_{\cdot}) \, =\, \text{Law} ( X_{\cdot}) \, =\,  \text{Law} ( Y_{\cdot}) \, , \\
&  \text{Law} (X_{0}, \widetilde{X}_{0}, Y_{0}) \, =\,  \text{Law} (X_{0}) \otimes \text{Law} ( \widetilde{X}_{0}) \otimes \text{Law} ( Y_{0}) \, , 
\end{split} 
\end{equation}
and $\, \widehat{ \mathrm m}_{t}(\cdot) \,$ is the marginal law, i.e., 
\begin{equation} \label{eq: sNDu2}
\widehat{\mathrm m}_{t} \, =\,  \text{Law} ( Y_{t}) \, =\,  \text{Law} ( X_{t}) \, =\, \text{Law} ( \widetilde{X}_{t}) \, ; \quad t \ge 0 \, , 
\end{equation}
with the independence relations, similar to (\ref{eq: NDlaw}),  
\begin{equation} \label{eq: sNDu3} 
{\bm \sigma } (\widetilde{X}_{t}, t \ge 0 ) \perp\!\!\!\perp {\bm \sigma } ( B_{t}, t \ge 0 ) \, , \quad {\bm \sigma } ((\widetilde{X}_{t}, X_{t}) , t \ge 0 ) \perp\!\!\!\perp {\bm \sigma } ( \widehat{B}_{t}, t \ge 0 ) \, .  
\end{equation}

We claim that the conditional distribution of $\, Y_{t}\,$, given $\,X_{t}\,$, is the same as the conditional distribution of $\, \widetilde{X}_{t}\,$, given $\,  X_{t}\, $, for every $\, t \ge 0  \,$, i.e., 
\begin{equation}\label{eq: NDclaw}
\text{Conditional Law} ( Y_{t} \, \vert \, X_{t}) \, =\, \text{Conditional Law} ( \widetilde{X}_{t}\, \vert \, X_{t}) \, ; \quad t \ge 0 \, . 
\end{equation}
with the condition 
\begin{equation} \label{eq: NDlaw5.6}
\, {\mathrm m}_{t} \, =\, \text{Law}(X_{t}) \, \equiv\,  \text{Law} ( \widetilde{X}_{t})   \, \equiv \,  \text{Law}(Y_{t}) \, =\, \widehat{\mathrm m}_{t} \, ; \quad t \ge 0 \, .   
\end{equation}

Indeed, thanks to (\ref{ineq: bm})-(\ref{ineq: gc}) and the fixed point argument, by some appropriate changes in the proof of Proposition \ref{prop: 1}, the weak solution $\,( Y_{\cdot}, X_{\cdot}, \widetilde{X}_{\cdot}) \,$ to (\ref{eq: sNDu}) exists with the constraints (\ref{eq: sNDu1})-(\ref{eq: sNDu3}), and its joint law and marginal laws are uniquely  determined. Since the couple $\, (X_{\cdot}, \widetilde{X}_{\cdot}) \,$ also solves the first equation in (\ref{eq: NDu}), it follows from the construction of the system (\ref{eq: sNDu}) and the uniqueness of (marginal) law in Proposition \ref{prop: 1} that $\, \text{Law} (\widetilde{X}_{\cdot}) \, =\,  \text{Law} (X_{\cdot})\,$ with the marginal $\, {\mathrm m} (\cdot) \, =\, {\mathrm m}^{\ast} (\cdot)\,$ and its marginal density $\, m_{t}\,$, $\, t \ge 0\,$. Thus we obtain (\ref{eq: NDlaw5.6}). Moreover, as in Proposition \ref{prop: 1.4}, its joint distribution $\, {\mathrm M}_{\cdot}\,$ of $\, (X_{\cdot}, \widetilde{X}_{\cdot})\,$ satisfies the integral equation (\ref{eq: IntEq}) with (\ref{eq: Ainfinty}). Similarly, the joint distribution $\, \widehat{\mathrm M}_{\cdot} \,$ of $\,(Y_{\cdot}, X_{\cdot})\,$ satisfies the integral equation 
\begin{equation} \label{eq: 5.7}
\int_{\mathbb R} g(x)  {\mathrm m}_{t}( {\mathrm d} x) \, =\,   \int_{\mathbb R} g(x) {\mathrm m}_{0}( {\mathrm d} x) + \int^{t}_{0}[ \widehat{\mathcal A}_{s}( \widehat{\mathrm M}) g] \, {\mathrm d}s \, ; \quad 0 \le t \le T \, ,  
\end{equation}
similar to (\ref{eq: IntEq}), for every test function $\,g \in C^{2}_{c}(\mathbb R)\,$, where 
\[
\widehat{\mathcal A}_{s}( \widehat{\mathrm M}) g \, :=\, u \int_{\mathbb R^{2}} \widetilde{b}(s, y_{1}, y_{2}) \cdot \frac{\,m_{s}(y_{1}) \,}{\,m_{s}(y_{2})\,} g^{\prime}(y_{1}) \widehat{\mathrm M}_{s}( {\mathrm d} y_{1} {\mathrm d} y_{2}) + (1-u) \int_{\mathbb R^{2}} \widetilde{b}(s, y_{1}, y_{2}) g^{\prime}(y_{1}) \mathrm m_{s}( {\mathrm d} y_{1}) \mathrm m_{s} ( {\mathrm d} y_{2}) 
\]
\begin{equation} \label{eq: Ainfinty5.8}
\quad {} + \frac{\,1\,}{\,2\,} \int_{\mathbb R} g^{\prime\prime}(y_{1}) \mathrm m_{s}( {\mathrm d} y_{1}) \, ; \quad 0 \le s \le T \, .  
\end{equation}
The uniqueness of solution to this integral equation (\ref{eq: 5.7}) may be shown as in Lemma 10 of {Oelschl\"ager} (1984). Thus, comparing (\ref{eq: IntEq}) with (\ref{eq: 5.7}), we obtain (\ref{eq: NDclaw}) from (\ref{eq: NDlaw5.6}) and the time-reversible relation 
\[
\,  {\, m_{s}(y_{1}) \,}\widehat{ \mathrm M}_{s}( {\mathrm d}y_{1} {\mathrm d}y_{2}) \, =\, {\,m_{s}(y_{2})\,}  \mathrm M_{s} ( {\mathrm d} y_{1} {\mathrm d} y_{2}) \, ; \quad 0 \le s \le T \, , \, \, (y_{1}, y_{2}) \in \mathbb R^{2} \, . 
\]

Thus, thanks again to the Lipschitz continuity (\ref{ineq: bm}) and linear growth condition (\ref{ineq: gc}), repeating the derivation of (\ref{eq: 5.2}) but now with this reversed discrete-time Markov chain relationship (\ref{eq: NDclaw}), we obtain (\ref{eq: 5.4}). Hence both  for the cases $\, j < i < k\,$ and $\, i <  j < k \,$ there exist constants $\, c\,$ and $\,C\,$ such that 
\[
\mathbb E [ \overline{b}(s, \overline{X}_{s,i} , \overline{X}_{s,j} ) \overline{b} (s, \overline{X}_{s,i} , \overline{X}_{s, k} ) ]  \le C \cdot \Big[ \frac{\,c^{k-j}\,}{\,(k-j)!\,}\Big]^{1/2} \, . 
\]

Therefore, we conclude (\ref{eq: 3.17}), because there exist constants $\, c, C > 0\,$ such that 
\[
\frac{1}{\, n\, }\sum_{j=1}^{n} \sum_{k=1}^{n} \mathbb E [ \overline{b}(s, \overline{X}_{s,i} , \overline{X}_{s,j} ) \overline{b} (s, \overline{X}_{s,i} , \overline{X}_{s, k} ) ]  \le \frac{2 C}{n} \sum_{j=1}^{n} \sum_{k=j}^{n} \Big[ \frac{\,c^{k-j}\,}{\,(k-j)!\,}\Big]^{1/2}  \le  2C \sum_{k=0}^{\infty} \Big[ \frac{\,c^{k}\,}{\,k!\,} \Big]^{1/2} < + \infty \,  . 
\]

\section*{Acknowledgements} We wish to thank Professors E. Bayraktar, I. Karatzas, K. Kardaras, H. Kawabi, K.  Ramanan, M. Shkolnikov, H. Tanemura and Z. Zhang for several helpful discussions in the early stage of this paper. We are thankful to Drs. E. R. Fernholz and I. Karatzas for kindly sharing with us their unpublished manuscripts on the Gaussian cascades. We also appreciate the editors and the reviewers comments and suggestions. Research is supported in part by the National Science Foundation under grants NSF-DMS-1409434 and NSF-DMS-1814091 for the second author and under grants NSF-DMS-13-13373 and NSF-DMS-16-15229 for the third author.

\bigskip 

\section*{Bibliography}

\noindent \textsc{Athreya, S.R., Bass, R.F. \& Perkins, E.A.} (2005) {H\"older norm estimates for elliptic operators on finite and infinite-dimensional spaces.} {\it Trans. Amer. Math. Soc.} {\bf 357} 5001--5029. 

\medskip 

\noindent \textsc{Bain, A. \& Crisan, D.} (2009) {\it Fundamentals of stochastic filtering}. Stochastic modeling and applied probability {\bf 60}, Springer.  

\medskip 

\noindent \textsc{Batt, A.G., Kallianpur, G., Karandikar, R.L. \& Xiong, J.} (1998) {On interacting systems of Hilbert-space-valued differential equations}. {\it Appl. Math. Optim.} {\bf 37} 151--188. 

\medskip 

\noindent \textsc{Bayraktar, E., Cosso, A. \& Pham, H.} (2018) {Randomized dynamic programming principle and Feynman-Kac representation for optimal control of McKean-Vlasov dynamics.} {\it Trans. Amer. Math. Soc.} {\bf 370} 2115--2160. 

\medskip 

\noindent \textsc{Bolley, F., Guillin, A. \& Malrieu, F.} (2010) {Trend to equilibrium and particle approximation for a weakly self consistent Vlasov-Fokker-Planck equation.} {\it ESAIM: M2AN} {\bf 44} 867--884. 

\medskip

\noindent \textsc{Carmona, R., Fouque, J.-P., \& Sun, L.-H.} (2015) {
Mean Field Games and Systemic Risk}.
{\it Comm. in Math. Sci.}, {\bf 13}(4) 911-933.
\medskip 

\noindent \textsc{Chong, C. \& Kl\"uppelberg, C.} (2017) {\it Partial mean field limits in heterogeneous networks}. {\it Preprint} ArXiv:1507.01905. 

\medskip 

\noindent \textsc{Da Prato, G. \& Zabczyk, J.} (1992) {\it Stochastic Equations in Infinite Dimensions}. Cambridge,
UK: Cambridge University Press.

\medskip 

\noindent \textsc{Dawson, D.A.} (1972) {Stochastic evolution equations}. Math. Biosci. {\bf 15} 287--316. 

\medskip 

\noindent \textsc{Ethier, S.N. \& Kurtz, T.G.} (1985) {\it Markov processes: Characterization and Convergence}. Wiley, New York. 

\medskip 

\noindent \textsc{Friedman, D.} (1971) {\it Markov Chains.} Holden-Day, San Francisco. 

\medskip 

\noindent \textsc{Funaki, T.} (1984) {A certain class of diffusion processes associated with nonlinear parabolic equations.} {\it Z. Wahrscheinlichkeitstheorie verw. Gebiete} {\bf 67} {331-348}. 

\medskip

\noindent \textsc{Graham, C.} (1992) {McKean-Vlasov Ito-Skorohod equations, and nonlinear diffusions with discrete jump sets}. {\it Stochastic Process. Appl.} {\bf 40} 69-82.

\medskip 

\noindent \textsc{Graham, C. \& M\'el\'eard, S.} (1997)  {Stochastic particle approximations for generalized Boltzmann models and convergence estimates.} {\it Ann. Probab.} {\bf 25}, 115-132. 
\medskip 

\noindent \textsc{Kac, M.} (1956) {Foundation of kinetic theory}. In {\it Proceedings of the Third Berkeley Symposium on Mathematical Statistics and Probability, 1954-1955}, vol {\bf III} Berkeley and Los Angeles, University of California Press.

\medskip

\noindent \textsc{Kallianpur, G. \& Xiong, J. } (1995) {Stochastic differential equations in infinite dimensions.} {\it Lecture Notes-Monograph Series} {\bf 26} Institute of Mathematical Statistics.

\medskip 

 \noindent    \textsc{Karatzas, I. \& Shreve, S.E.}  (1991)  {\it Brownian Motion and Stochastic Calculus.}   Springer-Verlag, NY.

\medskip 

\noindent \textsc{Kolokoltsov, V.N.} (2010) {\it Nonlinear Markov processes and kinetic equations} Cambridge Tracts in Mathematics {\bf 182}, Cambridge University Press, Cambridge. 

%

%

\medskip 

\noindent \textsc{Lacker, D., Ramanan, K. \& Wu, R.} (2019) Large sparse networks of interacting diffusions. {\it Preprint} ArXiv:1904.02585. 

\medskip

\noindent \textsc{Liptser, R.S. \& Shiryayev, A.N.} (2001) {\it Statistics of Random Processes. I \& II.} Applications of Mathematics Stochastic Modeling and Applied Probability {\bf 5-6},  Springer, New York. 

\medskip 

\noindent \textsc{McKean, JR., H.P.} (1967) {An exponential formula for solving Boltzmann's equation for a Maxwellian gas.} {\it Journal of Combinatorial Theory} {\bf 2} 258--372. 

\medskip 

\noindent \textsc{M\'el\'eard, S.} (1995) {Asymptotic behavior of some interacting particle systems; McKean-Vlasov and Boltzmann models.}  In: Probabilistic Models for Nonlinear Partial Differential Equations, Montecatini Terme, {\it Lecture Notes in Mathematics} {\bf 1627} 42--95. 

\medskip

\noindent \textsc{Mischler, S. \& Mouhot, C.} (2013) {Kac's program in kinetic theory} {\it Inventiones Matematicae} {\bf 193} 1-147.  

\medskip 

\noindent \textsc{Mischler, S., Mouhot, C. \& Wennberg, B.} (2015) {A new approach to quantitative propagation of chaos for drift, diffusion and jump processes.} {\it Probab. Theory Relat. Fields} {\bf 161} 1--59.   

\medskip 

\noindent \textsc{Oelschl\"ager, K.} (1984) {A Martingale approach to the law of large numbers for weakly interacting stochastic processes}. {\it Ann. Probab.} {\bf 2} 458--479. 

\medskip 

\noindent \textsc{Sznitman, A.S.} (1984) {Nonlinear reflecting diffusion processes, and the propagation of chaos and fluctuations associated.} {\it J. Funct. Anal.} {\bf 56} 311-336. 

\medskip

\noindent \textsc{Sznitman, A.S.} (1991) {Topics in propagation of chaos}. \'Ecole d'\'Et\'e de Probabilit\'es de Saint-Flour XIX -- 1989 In {Lecture Notes in Math.} {\bf 1464} 165--251, Springer Berlin.  

\medskip 

\noindent \textsc{Tanaka, H.} (1978) {Probabilistic treatment of the Boltzmann equation of Maxwellian molecules.} {\it Z. Wahrsch. Verw. Gebiete} {\bf 46} 67--105. 
\end{document}